\RequirePackage{fix-cm}

\documentclass[smallextended]{svjour3}       

\smartqed  
\usepackage{graphicx}
\usepackage{amsmath}
\DeclareMathOperator*{\argmin}{arg\,min}

\usepackage[utf8x]{inputenc}
\usepackage{cite}
\usepackage{url}
\usepackage{amsmath,amsfonts,amssymb}
\usepackage{fancyhdr}
\usepackage{algorithm}
\usepackage{algpseudocode}
\usepackage{accents}
\newcommand{\ubar}[1]{\underaccent{\bar}{#1}}
\usepackage{tkz-tab}
\usepackage{xpatch}
\usepackage{enumerate}   
\usepackage{multirow}
\usepackage{arydshln}
\usepackage{slashbox}
\usepackage{adjustbox}

\spnewtheorem{assumption}{Assumption}{\bf}{\it}
\spnewtheorem{corrolary}{Corrollary}{\bf}{\it}

\usepackage[title]{appendix}

\usepackage{verbatim}

\newcommand{\norms}[1]{\Vert#1\Vert}
\newcommand{\iprods}[1]{\langle #1\rangle}
\newcommand{\Lc}{\mathcal{L}}

\DeclareUnicodeCharacter{2212}{-}

\begin{document}

\title{Complexity of a linearized augmented Lagrangian  method for  nonconvex minimization  with nonlinear equality constraints
\thanks{The research leading to these results has received funding from: the European Union's Horizon 2020 research and innovation programme under the Marie Sklodowska-Curie  grant agreement No. 953348. 
 }}

\titlerunning{Linearized augmented Lagrangian  for nonconvex optimization}        

\author{Lahcen El Bourkhissi         \and
        Ion Necoara }
        
\authorrunning{L. El Bourkhissi         \and
        I. Necoara}

\institute{Lahcen El Bourkhissi \at
             Department of Automatic Control and Systems Engineering, National University of Science and Technology  Politehnica Bucharest, 
              Bucharest, 060042, Romania\\
              \email{lel@stud.acs.upb.ro}           
           \and
           Ion Necoara \at
             Department of Automatic Control and Systems Engineering, National University of Science and Technology  Politehnica Bucharest, 
             Bucharest, 060042, Romania.\\
             Gheorghe Mihoc-Caius Iacob Institute of Mathematical Statistics and Applied Mathematics of the Romanian Academy, 050711 Bucharest, Romania. \\
             \email{ion.necoara@upb.ro}
}

\date{Received: date / Accepted: date}

\maketitle

\begin{abstract}
In this paper, we consider a nonconvex optimization
problem with nonlinear equality constraints. We assume
that both, the objective function and the functional constraints
are locally smooth. For solving this problem, we propose a linearized augmented Lagrangian method, i.e., we linearize the objective function and the functional constraints in a Gauss-Newton fashion at the current iterate within the augmented Lagrangian function and add a quadratic regularization, yielding a  subproblem that is easy to solve, and  whose solution is the next primal iterate. The update of the dual multipliers is also based on the linearization of functional constraints.  Under a novel dynamic regularization parameter choice,  we prove boundedness and  global asymptotic convergence of the iterates to a first-order solution of the problem. We also derive convergence  guarantees for the iterates of our method to an $\epsilon$-first-order solution  in $\mathcal{O}(\sqrt{\rho} \epsilon^{-2})$ Jacobian evaluations, where $\rho$ is the penalty parameter. Moreover, when the problem exhibits a benign nonconvex property, we derive improved convergence results to an $\epsilon$-second-order solution.  Finally, we  validate the  performance of the proposed algorithm by numerically comparing it with the existing methods and software  from the literature.


\end{abstract}


\section{Introduction}
\label{intro}
In many fields, such as machine learning, matrix optimization, statistics, control and signal processing, one finds applications that can be recast  as  nonconvex optimization problems with nonlinear functional equality constraints, see, e.g., \cite{LukSab:19,HonHaj:17,Roy:19}. In this paper, we solve  this  optimization problem  by means of an augmented Lagrangian approach. The augmented Lagrangian method, also known as the method of multipliers, was initially proposed in \cite{Hes:69,Pow:69} to minimize  objective functions subject to (linear) equality constraints. It
provides many theoretical advantages, even for non-convex
problems, e.g., no duality gap and exact penalty representation, see \cite{RocWet:98}. Moreover, the augmented Lagrangian framework is at the heart of the Alternating Direction Method of Multipliers (ADMM), a very efficient method for  optimization problems with separable structure  \cite{BoyPar:11, GloTal:89,CohHal:21,BotNgu:20,ElbNec:24}. 

\medskip 
\noindent \textit{\textbf{Related work.}} 
 The augmented Lagrangian approach has been extensively studied in the literature for convex problems, see e.g., \cite{Ber:15,BoyPar:11,SheTeb:14,GloTal:89}, and recently it has been extended to non-convex (smooth/non-smooth) problems with linear equality constraints in e.g.,  \cite{JiaLin:19,ZhaLuo:20,HonHaj:17,HajHon:19,KeMa:17}. However, there are relatively few studies on the use of the augmented Lagrangian framework for nonconvex optimization with \textit{nonlinear equality constraints}, see e.g., \cite{XieWri:21,CohHal:21, HalTeb:23, ElbNecPan:25, SahEft:19}. For example,  in \cite{XieWri:21} a Proximal Augmented Lagrangian (Proximal AL) method is proposed to solve smooth nonconvex optimization problems with nonlinear equality constraints. In this approach, a static regularization term is added to the standard augmented Lagrangian function. The authors show that if an approximate first- (or second-) order solution of the nonconvex  subproblem is found, with an error asymptotically approaching zero, then an $\epsilon$-first- (or $\epsilon$-second-) order solution to the original problem is obtained within $\mathcal{O}(\epsilon^{\eta-2})$ outer iterations, where $\eta \in [0,2]$ is a user-defined parameter. However, the total iteration complexity is shown to be $\mathcal{O}(\epsilon^{-5.5})$ Jacobian evaluations to obtain an $\epsilon$-first-order solution when the Newton conjugate gradient method from \cite{RoyOne:19} is used to solve the nonconvex subproblem at each outer iteration.

\medskip

\noindent Another augmented Lagrangian-based method is Algencan \cite{AndBir:08}. This method can handle problems with equality and inequality constraints and its complexity analysis was recently published in \cite{BirMar:20}. It has been proved that an $\epsilon$- first-order solution of the original problem can be obtained in $\mathcal{O}(|\text{log}(\epsilon)|)$ outer iterations when the penalty parameter is bounded. However, it should be noted that in Algencan the penalty parameter is increasing and therefore its boundedness  does not seem to be guaranteed. Moreover, Algencan also considers the full augmented Lagrangian in the subproblem, which is highly nonconvex due to the nonlinearity of the constraints,  similar to Proximal AL in \cite{XieWri:21}. Therefore, solving the subproblem in Algencan is also very difficult and lead to high computational complexity in terms of Jacobian evaluations. 

\medskip

\noindent Furthermore, \cite{SahEft:19}  proposed an augmented Lagrangian-based method for solving nonsmooth nonconvex problems with nonlinear equality constraints. The authors considered the computation of an  inexact solution of the nonconvex subproblem, whose objective is based on  the full
augmented Lagrangian,  and introduced a sufficiently decreasing stepsize for updating the dual variables to ensure their boundedness. The convergence analysis showed that this inexact augmented Lagrangian method  achieves an $\epsilon$-first-order solution within  $\mathcal{O}(\epsilon^{-4})$ Jacobian evaluations, assuming the penalty parameter scales  as $\mathcal{O}(\epsilon^{-1})$. 

\medskip

\noindent Recently, \cite{ElbNecPan:25} introduced the Linearized Perturbed Augmented Lagrangian (LPAL) method for solving nonsmooth nonconvex problems with nonlinear equality constraints. LPAL perturbs the augmented Lagrangian by scaling the dual variables with a sub-unitary parameter and linearizes the smooth parts of the objective and constraints at each iterate in a Gauss–Newton fashion, while keeping the nonsmooth term. This yields a convex subproblem that is simple to solve. Its solution becomes the next primal iterate, followed by a perturbed dual ascent step. Under a new constraint qualification condition, the authors establish boundedness of the dual iterates and prove convergence to an $\epsilon$-first-order solution in $\mathcal{O}(\epsilon^{-3})$ Jacobian evaluations.

\medskip

\noindent A different approach, that does not rely on the augmented Lagrangian framework, however still related to our work in the sense that it linearizes the nonconvex terms,  is presented in \cite{MesBau:21,TraDie:10} and is called Sequential Convex Programming (SCP). This method solves a sequence of convex approximations of the original problem by linearizing the nonconvex parts of the objective and of the functional constraints and preserving the structures that can be exploited by convex optimization techniques. In this case the subproblem has a (strongly) convex objective and linear constraints, for which efficient solution methods exist, e.g., \cite{GraBoy:14,NecKva:15}. However, to the best of our knowledge,  SCP methods converge  under mild assumptions only  locally \cite{MesBau:21,TraDie:10}.   
 
\medskip 

\noindent\textit{\textbf{Drawback of existing works.}}
A primary challenge when employing augmented Lagrangian methods lies in simultaneously ensuring feasibility and  optimality  of a test point. A common approach to address this challenge involves assuming the \textit{boundedness of the dual iterates} and progressively increasing the penalty parameter, as exemplified in \cite{AndBir:08,BirMar:20,HalTeb:23}. However, this boundedness assumption presents a significant limitation, as it is imposed on the algorithm's generated sequence rather than being an inherent property of the problem itself. Indeed, paper \cite{HalTeb:23} acknowledged the difficulty of ensuring boundedness of the multiplier sequence in nonconvex settings, stating: “the boundedness of the multiplier sequence in the nonconvex setting is a very difficult matter and not at all obvious because coercivity arguments do not apply directly, and we are not aware of any breakthrough in this area”.  To circumvent this restrictive assumption of bounded multipliers,   \cite{SahEft:19}  employs an augmented Lagrangian algorithm with a sufficiently decreasing stepsize for the dual updates. This strategy combined with a regularity condition help to control the growth of the dual iterates and to manage  feasibility. Paper \cite{ElbNecPan:25} adopts a different approach to control the dual iterates, combining a perturbation technique  for the augmented Lagrangian with a new constraint qualification condition.  However, the requirements of small dual steps in \cite{SahEft:19} and of perturbation in \cite{ElbNecPan:25}, respectively,  may slow the overall  performance of the  algorithms. 

\medskip 

\noindent Consequently, the methods proposed in \cite{HalTeb:23} and \cite{SahEft:19} exhibit relatively high computational complexity in terms of Jacobian evaluations, specifically of order $\mathcal{O}(\epsilon^{-4})$,  while  \cite{ElbNecPan:25}  improves the computational complexity to $\mathcal{O}(\epsilon^{-3})$. On the other hand, Proximal AL algorithm in \cite{XieWri:21} incurs an even higher computational complexity $\mathcal{O}(\epsilon^{-5.5})$. Meanwhile, the ADMM type algorithms in \cite{CohHal:21, ElbNec:24} require the functional constraints to be separable and linear in one block of variables, and exhibit a computational complexity of order $\mathcal{O}(\rho \epsilon^{-2})$ in terms of Jacobian evaluations, where $\rho$ is the penalty parameter. Finally, SCP type schemes \cite{MesBau:21, TraDie:10} offer only local convergence guarantees. 

\medskip 

\noindent Another key disadvantage of existing augmented Lagrangian algorithms, see  e.g., \cite{XieWri:21, AndBir:08, SahEft:19},  lies in  calling complicated subroutines, as the subproblems that need to be solved at each iteration are  highly nonconvex. 

\medskip 

\noindent \textit{\textbf{Our contribution.}}  
In this paper we  propose a new  Linearized Augmented Lagrangian method (called L-AL) for solving smooth nonconvex problems with nonlinear equality constraints. Our method overcomes some of the limitations of existing approaches. In particular, using  a prox-linear-type (Gauss-Newton) mechanism, allows us to obtain an easily solvable subproblem at each iteration. Moreover, to ensure the boundedness of the primal and dual iterates, we introduce a novel dynamic regularization parameter choice. 
This choice for the regularization parameter, in conjunction with  the Linear Independence Constraint Qualification condition, enables us to prove boundedness of the iterates generated by our proposed algorithm and to establish improved convergence guarantees. Hence, our method combines the advantages of both Proximal AL and SCP approaches, as it enjoys global convergence guarantees and features subproblems that are easy to solve. More precisely, our main contributions~are:
\begin{enumerate}[(i)]
\item We linearize both the objective function and the functional constraints within the augmented Lagrangian function at the current iterate in a Gauss-Newton fashion and add a dynamic regularization term. The solution of this subproblem is the next primal iterate.  The update of the dual multipliers is also based on a linearization technique using the solution of this subproblem. These primal and dual updates lead to a \textit{novel  Linearized Augmented Lagrangian} method, called L-AL.   Our algorithm exhibits several desirable properties. Notably, it only necessitates evaluations of the problem’s function values and their first-order derivatives. Moreover, each iteration requires minimizing   a simple unconstrained quadratic convex  subproblem that  reduces to solving a linear system of equalities. These properties enable the efficient handling of large-scale nonconvex problems by L-AL algorithm.  

\item Under a novel dynamic regularization parameter choice and under the Linear Independence Constraint Qualification (LICQ) condition, we prove that the primal and dual iterates generated by L-AL method are bounded. \textit{To the best of our knowledge, this represents one of the first results of this nature obtained within the linearized augmented Lagrangian framework}.  

\item We also establish global convergence guarantees, proving that any limit point of the primal and dual sequences is a KKT point of the original  problem. Furthermore, by leveraging the Kurdyka-Lojasiewicz property, we demonstrate
the convergence of the entire sequence generated by L-AL algorithm.   We also demonstrate that the primal iterates of our method reach an $\epsilon$-first-order  solution of the  problem in at most $\mathcal{O}(\sqrt{\rho}\epsilon^{-2})$ Jacobian evaluations, where $\rho$ is the penalty parameter. \textit{To the best of our knowledge, this is the optimal computational complexity in the context of augmented Lagrangian and penalty methods for smooth nonconvex constrained optimization problems}.

\item  
The theoretical complexity bounds  for augmented Lagrangian algorithms  do not always reflect the observed good practical performance of these methods.  To address this discrepancy, we specialize our L-AL algorithm for a class of optimization problems exhibiting a benign nonconvex property (also called strict saddle condition). For this restricted class, we establish improved complexity bounds ranging from $\mathcal{O}(\epsilon^{-2})$ to $\mathcal{O}(\epsilon^{-1})$ to attain even an $\epsilon$-second-order solution. \textit{It seems that this is the first complexity result for an augmented Lagrangian-type algorithm specifically designed for solving this class of benign nonconvex problems with nonlinear equality constraints.}

\item Finally, in addition to proposing a novel algorithm and providing its convergence guarantees, we demonstrate the algorithm’s efficiency through numerical experiments using test problems from the CUTEst library,  numerically comparing it with some well-known  existing methods and software such as SCP \cite{MesBau:21}, IPOPT \cite{WacBie:06} and Algencan \cite{AndBir:08}. 
\end{enumerate}


\noindent The paper is structured as follows. In Section \ref{sec2}, we introduce our problem of interest  and some notions necessary for our analysis. In Section \ref{sec3}, we present our algorithm, followed in Sections \ref{sec4} and  \ref{sec5} by its convergence analysis. Finally, in Section \ref{sec6}, we compare numerically our method with existing algorithms. 


\section{Problem formulation and  preliminaries}
\label{sec2}
In this paper, we consider the following nonlinear optimization problem:
\begin{equation}
\begin{aligned}\label{eq1}
& \underset{x\in\mathbb{R}^n}{\min}
& & f(x)\\
& \hspace{0.2cm}\textrm{s.t.}
& & \hspace{0.07cm} F(x)=0,
\end{aligned}
\end{equation}
where $f:\mathbb{R}^{n}\to {\mathbb{R}}$ and  $F(x)\triangleq{(f_1(x),...,f_m(x))}^T$, with $f_i:\mathbb{R}^{n}\to {\mathbb{R}}$ for all $i=1:m$. We assume the functions  $f, f_i \in \mathcal{C}^2$ for all  $i=1:m$   and $F$ is nonlinear. Moreover, we assume that the problem is well-posed i.e., the feasible set is nonempty and the optimal value is finite. Before introducing the main assumptions for our analysis, we would like to clarify some notations.  We use $\|\cdot \|$ to denote the $2-$norm of a vector or of a matrix, respectively.  For a differentiable function $f:\mathbb{R}^n\to\mathbb{R}$, we denote by $\nabla f(x)\in\mathbb{R}^n$ its gradient at a point $x$. For a differentiable vector function $F:\mathbb{R}^n  \to\mathbb{R}^m$, we denote its Jacobian at a given point $x$ by ${J}_F(x)\in\mathbb{R}^{m\times n}$. Moreover, for a given matrix $J\in\mathbb{R}^{m\times n}$, we denote by $\sigma_{\text{min}}(J)$ the smallest singular value. In our analysis, we often use the following inequality:
\begin{equation}\label{inequality_prop}
   \langle a  ,b\rangle \leq \frac{1}{2r}\|a\|^2+\frac{r}{2}\|b\|^2 \hspace{0.5cm}\forall a,b\in\mathbb{R}^m \text{ and } r>0.
\end{equation}

\noindent Let us now present the main assumptions considered  for problem \eqref{eq1}:

\begin{assumption}\label{assump1}
Assume that there exists $\rho_0\geq0$ such that $f(x)+\frac{\rho_0}{2}\|F(x)\|^2$ has compact level sets, i.e., for all $\alpha\in\mathbb{R}$, the following set is empty or compact:
\[
\mathcal{S}_{\alpha}^0\triangleq{\{x:\, f(x)+\frac{\rho_0}{2}\|F(x)\|^2\leq\alpha\}}.
\]
\end{assumption}

\begin{assumption}
\label{assump2}
For any compact set $\mathcal{S}\subseteq\mathbb{R}^n$, there exist positive constants $M_f, M_F, \sigma, L_f, L_F$ such that $f$ and $F$ satisfy the following conditions:
\begin{enumerate}[(i)]
  \item $ \|\nabla f(x)\|\leq M_f, \;\; \|\nabla f(x)-\nabla f(y)\|\leq L_f\|x-y\| \;\text{ for all } x, y\in\mathcal{S}$.\label{ass1}
  \item $  \|{J_F}(x)\|\leq M_F, \;\; \sigma_{\emph{min}}({{J_F}(x)}) \geq \sigma>0\;\text{ for all } x \in\mathcal{S}$. \label{ass2}
  \item $    \|{J_F}(x)-{J_F}(y)\|\leq L_F\|x-y\|\;\text{ for all } x, y\in\mathcal{S}${.}\label{ass3}
\end{enumerate}
\end{assumption}

\begin{assumption}\label{assump3}
There exist finite $\ubar{f}$ and $\bar{f}$ such that $f(x)\leq\bar{f}$ for all $x\in\{x \in \mathbb{R}^n:\, \|F(x)\|\leq1\}$ and $f(x) \geq \ubar{f}$ for all $x \in \mathbb{R}^n$.
\end{assumption}

\noindent Note that these assumptions are standard in the nonconvex optimization literature, see e.g., \cite{XieWri:21,CohHal:21, GraYua:21, ElbNecPan:25, ElbNec:25, LiChe:21, DemJia:23}.  In fact, these assumptions are not restrictive as they need to hold only  locally. Indeed,  large classes of problems satisfy these assumptions as discussed below.

\begin{remark}
Assumption \ref{assump1} holds e.g., when $f(\cdot)+\rho_0/2\|F(\cdot)\|^2$ is coercive for some $\rho_0\geq0$; when $f(\cdot)$ is strongly convex or $f(\cdot)$ is bounded from bellow and the components of $F(\cdot)$ are strongly convex, as in the case of dictionary learning applications. It also holds when  $f(x)=\frac{1}{2}x^TQx-p^Tx, F(x)=Ax-b$ and $Q$ is a positive definite matrix on $\text{null}(A):=\{x: \,Ax=0\}$. Note that Assumption \ref{assump1} is introduced here just to avoid assuming that the primal iterates of our algorithm  are bounded (boundedness of the primal iterates is commonly assumed in the literature, see e.g.,  \cite{CohHal:21, HalTeb:23, ElbNec:25, ElbNecPan:25,GraYua:21}).
\end{remark}

\begin{remark}
Assumption \ref{assump2} allows general classes of problems. In particular, conditions \textit{(\ref{ass1})} hold if $f(\cdot)$ is differentiable and $\nabla f(\cdot)$ is \textit{locally}  Lipschitz continuous on a neighborhood of $\mathcal{S}$. Conditions \textit{(\ref{ass2})} hold when
$F(\cdot)$ is differentiable on a neighborhood of $\mathcal{S}$ and satisfies an LICQ condition over $\mathcal{S}$ (hence, $m \leq n$). Finally, condition \textit{(\ref{ass3})} holds
if ${J_F}(\cdot)$ is \textit{locally} Lipschitz continuous on $\mathcal{S}$. Note that any twice continuously differentiable function is locally Lipschitz and locally smooth on a bounded set.
\end{remark}

\begin{remark}
For  Assumption \ref{assump3} to hold, it is sufficient that the set $\{x: \,\|F(x)\|\leq 1\}$ is compact and that $f(\cdot)$ is coercive. In fact, we do not need this assumption if we can choose the starting point of our algorithm, $x_0$, such that $F(x_0)=0$, that is, the initial point is feasible and the objective function is strongly convex.
\end{remark}

\noindent The following lemma is an immediate consequence of Assumption \ref{assump1}.
\begin{lemma} If  Assumption \ref{assump1} holds, then $f(\cdot)+\frac{\rho_0}{2}\|F(\cdot)\|^2$ is lower bounded:
\begin{equation}\label{lem1}
 \ubar{P}\triangleq{\inf_{x\in\mathbb{R}^n}\{ f(x)+\frac{\rho_0}{2}\|F(x)\|^2\}}>-\infty.   
\end{equation}
\end{lemma}

\noindent Further,   let us  introduce the following definition: 
\begin{definition}\label{firstorder}[First-order solution  and $\epsilon$-first-order solution of \eqref{eq1}]
The vector $x^*$ is said to be a first-order solution of problem \eqref{eq1} if  $\exists\lambda^*\in\mathbb{R}^m$ such that: 
\begin{equation*}
\nabla f(x^*)+{J_F}(x^*)^T\lambda^*=0\hspace{0.3cm} \text{and} \hspace{0.3cm} F(x^*)=0.
\end{equation*} 
Moreover, $x_{\epsilon}^*$ is an $\epsilon$-first-order solution of \eqref{eq1} if  $\exists \lambda_\epsilon^* \in \mathbb{R}^m$ such that:
\begin{equation*}
    \|\nabla f(x_{\epsilon}^*)+{J_F}(x_{\epsilon}^*)^T\lambda_{\epsilon}^*\|\leq\epsilon\hspace{0.3cm} \text{and} \hspace{0.3cm} \|F(x_{\epsilon}^*)\|\leq\epsilon.
\end{equation*} 
\end{definition}
The pair $(x^*,\lambda^*)$ is called a KKT point of  problem \eqref{eq1}.  Let us also introduce the notion of an $\epsilon$-second-order solution to \eqref{eq1}.

\begin{definition} \label{de:2nd_approx_sol} [{\(\epsilon\)-second-order solution of \eqref{eq1}}] 
A vector $x^{*}_\epsilon$ is called an \(\epsilon\)-second-order  solution of  problem \eqref{eq1}  if  $\exists \lambda^{*}_\epsilon \in \mathbb{R}^m$ such that:
\begin{align}
\label{eq:2nd_approx_sol}
& \Vert \nabla{f}(x^{*}_\epsilon) + J_F(x^{*}_\epsilon)^T \lambda^{*}_\epsilon \Vert   \leq  \epsilon, \quad  \Vert \nabla F(x^{*}_\epsilon)\Vert   \leq  \epsilon, \; \text{ and} \nonumber \\
& d^T \left( \nabla^2f(x^{*}_\epsilon) + \sum_{i=1}^m (\lambda^{*}_\epsilon)_i \nabla^2F_i(x^{*}_\epsilon) \right) d  \geq  -\epsilon \quad \forall d \in \mathbb{B}_1(x^{*}),
\end{align}
where $\mathbb{B}_1(x) := \{d \in \mathbb{R}^n \mid  J_F(x) d = 0, \ \Vert d \Vert = 1 \}$.
\end{definition}

\noindent Finally, let us  introduce the  Kurdyka-Lojasiewicz (KL) property, a condition widely used in the context of nonconvex optimization \cite{AttBol:13}. Since our functions are all continuously differentiable we adapt the KL definition to this setting.  For a function	 $\Phi:\mathbb{R}^d\to {\mathbb{R}}$ and  $−\infty<\tau_1<\tau_2\leq
+\infty$, we define $[\tau_1<\Phi<\tau_2]=\{x \in\mathbb{R}^d :\tau_1<\Phi(x)<\tau_2\}$. 
\begin{definition} \label{def2}
Let $\Phi : \mathbb{R}^d \to {\mathbb{R}}$  be a continuously differentiable  function that takes constant value on a set $\Omega$. We say that $\Phi$ satisfies the KL property on $\Omega$ if there exists $ \epsilon>0, \tau>0$ and $\varphi\in\Psi_{\tau}$ (where  $\Psi_{\tau}$ denotes the set of all continuous concave functions $\varphi: [0, \tau] \to [0,+\infty)$ satisfying $\varphi(0) = 0$ and $\varphi$ is continuously differentiable on $(0, \tau)$, with $\varphi' > 0$ over $(0, \tau)$) such that for every
$x^* \in \Omega$ and every element $x$ in the intersection $\{x\in\mathbb{R}^d: \text{ dist}(x,\Omega)<\epsilon\}\cap[\Psi(x^*)<\Psi(x)<\Psi(x^*)+\tau]$, we have:
\[
    \varphi'\big(\Phi(x) − \Phi(x^*)\big) \cdot  \|\nabla \Phi(x)\|  \geq1.
\] 
\end{definition}
This definition covers many classes of functions arising in practical optimization. For example, if $\Phi$ is a  semialgebraic function (including convex piecewise linear/quadratic functions), then $\Phi$ is a KL function with $\varphi(s)=s^{1-\nu}$, where  $\nu \in [0,1)$, see \cite{AttBol:13}. The function $g(Ax)$, where $g$ is strongly convex on a compact set and  twice differentiable, and $A \in \mathbb{R}^{m\times n}$, is also 
a KL function.


\section{A linearized augmented Lagrangian method}\label{sec3}
In this section, we propose a new algorithm for solving problem \eqref{eq1} using the augmented Lagrangian framework. Let us first introduce few  notations. The augmented Lagrangian function associated with the problem \eqref{eq1} is:
\[
\mathcal{L}_{\rho}(x,\lambda)=f(x)+\langle \lambda,F(x) \rangle+\frac{\rho}{2}{\|F(x)\|^2}, 
\]
where the penalty parameter $\rho \geq 0$. In the sequel, we also use the  notations:
\begin{gather*}
     l_f(x;\bar{x}):=f(\bar{x})+\langle\nabla f(\bar{x}),x-\bar{x}\rangle,\quad    l_F(x;\bar{x}):=F(\bar{x})+ J_F(\bar{x})(x-\bar{x}) \;\; \forall x,\bar{x}.
\end{gather*}
Further, let us denote the following function derived from linearization of  objective and the functional constraints in a Gauss-Newton fashion, at a given point $\bar{x}$, within the augmented Lagrangian function:
\begin{align*}
 \bar{\mathcal{L}}_{\rho}(x,\lambda;\bar{x}) = l_f(x;\bar{x}) + \langle \lambda  , l_F(x;\bar{x}) \rangle+\frac{\rho}{2}{\| l_F(x;\bar{x})\|^2}.   
\end{align*}

\noindent  For the convergence analysis let us define the following Lyapunov  function:
\begin{equation}\label{P}
  P(x,\lambda,y,\gamma)=\mathcal{L}_{\rho}(x,\lambda)+\frac{\gamma}{2}\|x-y\|^2.
\end{equation}
Note that such Lyapunov function is standard in the analysis of augmented Lagrangian based methods, see e.g., \cite{XieWri:21, CohHal:21, HalTeb:23}. The evaluation of the Lyapunov function along the iterates of L-AL algorithm  is denoted by:
\begin{equation}\label{lyapunov_function}
 P_{k}=P\left(x_k,\lambda_k,x_{k-1},\frac{\beta_k}{2}\right) \hspace{0.3cm} \forall k\geq 0,
\end{equation}
with the convention that $x_{-1}=x_0$, and $\beta_0$ can be any positive real number.
In the sequel, we also denote: 
\[
\Delta x_{k}=x_{k}-x_{k-1} \hspace{0.3cm} \text{and} \hspace{0.3cm} \Delta\lambda_{k}=\lambda_{k}-\lambda_{k-1}\hspace{0.3cm}\forall k\geq 0,
\]
with the convention that $\lambda_{-1}=\lambda_0$. To solve the optimization problem \eqref{eq1} we propose the following \textit{Linearized Augmented Lagrangian} (L-AL) algorithm, i.e., we linearize the objective function and the functional constraints in the augmented Lagrangian function at the current iterate using  a Gauss-Newton type mechanism and add a quadratic regularization.

\begin{algorithm}
\caption{Linearized augmented Lagrangian (L-AL)}\label{alg1}
\begin{algorithmic}[1]
\State  $\textbf{Initialization: } x_{-1}=x_0, \lambda_0, \text{ and } \rho \geq 1, \mu>1, \beta_1 \geq \ubar{\beta}>0$. 
\State $k \gets 0$
\While{$\text{ stopping criterion is not satisfied }$}
    \State find the smallest $i_k \geq 0$ such that the points 
    \State $x_{k+1}\gets\argmin_{x}{\bar{\mathcal{L}}_{\rho}(x,\lambda_{k};x_{k})+\frac{\mu^{i_k}\beta_{k+1}}{2}{\|x-x_{k}\|}^2}$
    \State $  \lambda_{k+1}\gets\lambda_{k}+\rho \left(F(x_{k})+{J_F}(x_{k})(x_{k+1}-x_{k})\right)$
    \State satisfy 
    \begin{equation} \label{decrease_algorithm}
        P_{k+1} - P_{k} \leq \frac{3}{2\rho}\left\|\Delta\lambda_{k+1}\right\|^2-\frac{\mu^{i_k}\beta_{k+1}}{4}\|\Delta x_{k+1}\|^2-\frac{\beta_{k}}{4}\|\Delta x_{k}\|^2.
    \end{equation}    
    \State $\beta_{k+1} \gets \mu^{i_k}\beta_{k+1}$
    \State $\beta_{k+2} \gets  \max \left\{ \frac{\beta_{k+1}}{\mu}, \ubar{\beta} \right \}$
    \State  $k \gets k+1$
\EndWhile
\end{algorithmic}
\end{algorithm}

\medskip 

\noindent \textit{To the best of our knowledge L-AL algorithm is new and its convergence behaviour has not been analyzed before in the literature.}   Note that  the objective function of the subproblem in step 5 of  Algorithm \ref{alg1}, which is unconstrained, is quadratic and strongly convex. Therefore, finding a solution of the subproblem in step 5 is  equivalent to solving a linear system of equalities.  Hence,   efficient solution methods exist for solving the subproblem, see e.g., \cite{GraBoy:14,NecKva:15}.

\medskip 

\noindent It is also important to note that our update of the dual multipliers is different from the literature, i.e., instead of evaluating the functional constraints at the new test point $x_{k+1}$ and updating clasically $\lambda_{k+1}=\lambda_{k}+\rho F(x_{k+1})$ as e.g., in  \cite{XieWri:21,CohHal:21}, we evaluate their linearization at $x_{k}$ in the new point $x_{k+1}$ and update as $\lambda_{k+1} = \lambda_{k}+\rho (F(x_{k})+{J_F}(x_{k})(x_{k+1}-x_{k}))$.


\section{Convergence analysis}\label{sec4}
In this section, we derive the asymptotic convergence of the iterates of  L-AL algorithm (Algorithm \ref{alg1}) and the computational complexity to obtain an $\epsilon$-first-order solution for  problem \eqref{eq1}. In the rest of this paper, for the sake of clarity, we provide the proofs of all the lemmas in Appendix.
Let us start by bounding $\|\Delta\lambda_{k+1}\|^2$. 
\begin{lemma}\label{lambda_bound}[Bound for $\|\Delta\lambda_{k+1}\|$] Consider Algorithm \ref{alg1}.
Suppose that for a fixed $k\geq 1$, Assumption \ref{assump2} holds for some set $\mathcal{S}$ and that $x_{k-1}, x_k  \in \mathcal{S}$. Then, 
\begin{align}
    \label{lambda_squared}
\|\Delta\lambda_{k+1}\|^2\leq c (\beta_{k+1})\|\Delta x_{k+1}\|^2+c (\beta_{k})\|\Delta x_{k}\|^2,
\end{align}
where \;$ c (\beta)=\frac{4(1+3\mu)^2\left(L_f M_F + M_f L_F\right)^2 }{\sigma^4} + \frac{4(1+3\mu)^2 M_F^2}{\sigma^4}(\beta - \mu L_f)^2$.
\end{lemma}

\begin{proof}
See Appendix.
\end{proof}

\noindent Next, we show that under a \textit{novel dynamic regularization parameter choice}, $\beta_{k+1}$, Algorithm \ref{alg1} is well-defined,  in particular, the inner process terminates in a finite number of steps.

\begin{lemma}\label{lemma3}[Existence of $\beta_{k+1}$] Consider Algorithm \ref{alg1}.
Suppose that for a fixed $k\geq 0$, Assumption \ref{assump2} and \ref{assump3} hold for some set $\mathcal{S}$ and that $ x_k , x_{k+1} \in \mathcal{S}$ together with $\lambda_k \in \Lambda$, where $\Lambda$ is a compact set of $\mathbb{R}^m$. If $\beta_{k+1}$ is chosen to satisfy:
\begin{equation} \label{eq_assu}
    \beta_{k+1} \geq L_f + L_F\sqrt{2\rho}\sqrt{\mathcal{L}_{\rho}(x_k,\lambda_k) + \frac{1}{2 \rho}\|\lambda_k\|^2 -\ubar{f}}, 
\end{equation}
 then  inequality  \eqref{decrease_algorithm} holds.
\end{lemma}

\begin{proof}
See Appendix.
\end{proof}

\noindent Note that  for $k=0$, $\lambda_0$ is bounded, and for $k>0$, in addition to having $x_k, x_{k+1} \in \mathcal{S}$, if we also have $x_{k-1} \in \mathcal{S}$, then from the proof of Lemma \ref{lambda_bound}, there exists a ball in $\mathbb{R}^m$, denoted by $\Lambda$, such that $\lambda_k \in \Lambda$.  Clearly,  for any $\rho \geq 1$ at any iteration $k \geq 0$, the inner process in Algorithm \ref{alg1} terminates in at most $i_k$ steps, where $i_k$ satisfies (see also Remark \ref{w-d} below):
\[
\beta_{k+1} \gets \mu^{i_k} \beta_{k+1} \geq L_f + L_F\sqrt{2\rho}\sqrt{\mathcal{L}_{\rho}(x_k,\lambda_k) + \frac{1}{2 \rho}\|\lambda_k\|^2 -\ubar{f}}. 
\]
Additionally,   $\beta_{k+1}$ can be always  bounded as follows:
\begin{equation} 
\label{bar_beta}
    \ubar{\beta}\leq\beta_{k+1} \leq \mu \left(L_f+ L_F\sqrt{2\rho}\sqrt{\mathcal{L}_{\rho}(x_k,\lambda_k) + \frac{1}{2 \rho}\|\lambda_k\|^2 - \ubar{f}}\right) \quad \forall k \geq 0.
\end{equation}

\noindent Let $\rho \geq 1$. In the sequel, we assume that $x_0$ is chosen such that:
\begin{equation}\label{eq4}
    \|F(x_0)\|^2\leq\min\left\{1,\frac{c_0}{\rho}\right\} \hspace{0.7cm} \text{ for some } c_0>0.
\end{equation}
Then,  from Assumption \ref{assump3}, we have $f(x_0)\leq\bar{f}$. Let us define:
\begin{equation}\label{alpha_hat}
    \bar{P}\triangleq \bar{f}+c_0+4\|\lambda_0\|^2+ 2,
\end{equation}
and
\begin{equation}\label{beta_bar_def}
   \bar{\beta}\triangleq \mu \left(L_f+ L_F\sqrt{2\rho}\sqrt{\bar{P}  - \ubar{f}}\right).
\end{equation}
Furthermore, we define the diameter of  compact set $\mathcal{S}^0_{\bar{P}}$ (see Assumption \ref{assump1}):
\begin{equation} \label{diameter}
   D_{\bar{P}} = \max \{ \|x - y\| \mid x, y \in \mathcal{S}^0_{\bar{P}} \},
\end{equation}
and
\begin{equation} \label{bar_gamma_before}
    \bar{\gamma} \triangleq \frac{8\mu^2L_F^2D_{\bar{P}}^2(\bar{P}   - \ubar{f})}{\sigma^2} + 1.
\end{equation}

\noindent The following lemma shows the decrease of the Lyapunov function along any  two consecutive iterates.
\begin{lemma}\label{conditional_decrease}[Decrease] Consider Algorithm \ref{alg1}.
Suppose that for a fixed $k\geq 1$, Assumption \ref{assump2} holds for some set $\mathcal{S}$ and that $x_{k-1}, x_k , x_{k+1} \in \mathcal{S}$. If, we have
\begin{align}\label{theright_choice_of_rho}
\rho \geq \max\Bigg\{&\frac{48(1+3\mu)^2\left(L_f M_F + M_f L_F\right)^2 }{\mu L_f\sigma^4}, \frac{48(1+3\mu)^2 M_F^2}{\sigma^4}(\beta_k - \mu L_f), \nonumber\\
&\quad \frac{48(1+3\mu)^2 M_F^2}{\sigma^4}(\beta_{k+1} - \mu L_f)\Bigg\},
\end{align}
then the Lyapunov function decreases  according to the following formula:
 \begin{equation}\label{decrease_Lyapunov_lem}
     P_{k+1}-P_{k}\leq-\frac{\beta_{k+1}}{8}\|\Delta x_{k+1}\|^2-\frac{\beta_{k}}{8}\|\Delta x_{k}\|^2. 
 \end{equation} 
\end{lemma}

\begin{proof}
See appendix.
 \end{proof}

\noindent Let us now bound the gradient of the augmented Lagrangian function.
\begin{lemma}\label{bounded_gradient}[Boundedness of $\nabla\mathcal{L}_{\rho}$]
Consider Algorithm \ref{alg1}.
Suppose that for a fixed $k\geq 1$, Assumption \ref{assump2} holds for some set $\mathcal{S}$ and that $x_{k-1}, x_k , x_{k+1} \in \mathcal{S}$. Then, we have:
\[
    \|\nabla\mathcal{L}_{\rho}(x_{k+1},\lambda_{k+1})\|\leq\Gamma_{k+1}\|\Delta x_{k+1}\|  +\Gamma_{k}\|\Delta x_{k}\| + c_{k+1} \|\Delta x_{k+1}\|^2 + c_k\|\Delta x_{k}\|^2,
\]
where   $c_{k} = \frac{L_F}{2}\left(1+\frac{2\beta_{k} +\rho M_F \sigma}{\sigma}\right)$ and 
\[
    \Gamma_{k}=\left(M_F+\frac{1}{\rho}\right)\frac{(2+ 3\mu)(L_f M_F + M_f L_F)  + (2+ 3\mu) M_F (\beta_{k}- \mu L_f)}{\sigma^2}. 
\]
\end{lemma}
\begin{proof}
See Appendix.
\end{proof}
 
\noindent In the remainder of this paper, we assume that $\rho$ is chosen as follows:
\begin{align}\label{rho_def}
& \rho \geq \!\max\Bigg\{1, \rho_0\!+\!1, \rho_0 \!+\! \frac{12 M_f^2}{\sigma^2}, 2\rho_0\! +\!\frac{48\mu L_f}{\sigma^2}, \frac{48(1\!+\!3\mu)^2\left(L_f\! M_F \!+\! M_f\! L_F\right)^2 }{\mu L_f\sigma^4}, \nonumber \\
&\frac{4608\mu^2\!(1\!+\!3\mu)^4\!M_F^4\!L_F^2(\bar{P}\!-\!\ubar{f})}{\sigma^8}, \rho_0\!+\!\frac{2M_f^2 \! +\! 4 \mu^2\!L_f^2D_{\bar{P}}^2\! +\! 8\mu^2\!L_F^2\!D_{\bar{P}}^2\!(\bar{P}   \!-\! \ubar{f})\rho_0}{\sigma^2}\Bigg\}.
\end{align}
Note that this choice of $\rho$ depends only on the parameters of the problem's functions. Note that by making use of the definition of $\bar{\beta}$, see \eqref{beta_bar_def}, it is easy to see that \eqref{rho_def} implies:
\begin{align}\label{choice_rho}
    &\rho\geq \max\Bigg\{1, \rho_0+1, \rho_0 + \frac{12 M_f^2}{\sigma^2}, \rho_0 +\frac{24\bar{\beta}}{\sigma^2}, \frac{48(1+3\mu)^2\left(L_f M_F + M_f L_F\right)^2 }{\mu L_f\sigma^4}, \nonumber\\
    & \frac{48(1\!+\!3\mu)^2 M_F^2(\bar{\beta} \! - \! \mu L_f)}{\sigma^4}, \rho_0\!+\!\frac{2M_f^2 \!+ \!4 \mu^2L_f^2D_{\bar{P}}^2 \!+\! 8\mu^2L_F^2D_{\bar{P}}^2(\bar{P}  \! - \! \ubar{f})\rho_0}{\sigma^2}\Bigg\}.
\end{align}

\noindent  Before proving that  iterates $\{(x_{k},\lambda_{k})\}_{k\geq1}$ generated by Algorithm \ref{alg1} are bounded, we present the following technical lemma, which is useful for our convergence analysis and for well-definiteness of Algorithm \ref{alg1}.
\begin{lemma}\label{usefull_lemma}  Consider Algorithm \ref{alg1}.
Suppose that for a fixed $k\geq 1$, Assumption \ref{assump2} and \ref{assump3} hold for some set $\mathcal{S}$ and that $  x_{k-1} \in \mathcal{S}$. If $\beta_{k} \leq \bar{\beta}$ and $\rho$ is chosen as in \eqref{rho_def}, then we have the following:
\begin{subequations}\label{important00}
  \begin{gather}
    \frac{3}{\rho}\|\lambda_{k}\|^2 - \frac{\beta_k}{4}\|x_{k} -x_{k-1}\|^2 \leq 1, \label{first_bound}\\
    \mathcal{L}_{\rho}(x_k,\lambda_k) + \frac{1}{2\rho}\|\lambda_k\|^2 \leq P_k +1. \label{second_bound} 
  \end{gather}
\end{subequations}
\end{lemma}
\begin{proof}
See Appendix.
\end{proof}

\begin{remark}
\label{w-d}
\noindent From the previous lemmas it is clear that any iteration $k \geq 0$, the inner process in Algorithm \ref{alg1} terminates in at most $i_k$ steps, where $i_k$ satisfies
\[
\mu^{i_k}\ubar{\beta} \geq L_f + L_F\sqrt{2\rho}\sqrt{\mathcal{L}_{\rho}(x_k,\lambda_k) + \frac{1}{2 \rho}\|\lambda_k\|^2 -\ubar{f}},
\]
 as the right hand side in the previous relation is finite: 
\[    \mathcal{L}_{\rho}(x_k,\lambda_k) + \frac{1}{2\rho}\|\lambda_k\|^2  \overset{\eqref{second_bound}}{\leq}  P_k +1   \overset{\eqref{decrease_Lyapunov_lem}}{\leq} \max\{P_0, P_1\} +1,  \]
provided that  $\rho$ is chosen as in  \eqref{rho_def}. 
\end{remark}

\noindent The following two lemmas show that the sequence $\{(x_{k},\lambda_{k})\}_{k\geq1}$ generated by  Algorithm \ref{alg1}  is bounded. These results are new in the context of linearized augmented Lagrangian framework (see our discussion in the section Introduction and also in \cite{HalTeb:23}) and they are important for our convergence analysis.
\begin{lemma}\label{bbound} Consider  Algorithm \ref{alg1} and let $\{P_{k}\}_{k\geq1}$ as defined in  \eqref{lyapunov_function}. If Assumptions \ref{assump1}, \ref{assump2} and \ref{assump3} hold with $\mathcal{S}=\mathcal{S}_{\bar{P}}^0$ and $\bar{P}$ defined in \eqref{alpha_hat} for any fixed constant $c_0$ and $D_{\bar{P}}$  the diameter of $\mathcal{S}_{\bar{P}}^0$. If $\rho$ is chosen as in \eqref{rho_def} and $x_0$ is chosen to satisfy \eqref{eq4}, then  for any $k\geq1$  the following  holds:
\begin{subequations}\label{important}
  \begin{gather}
    \beta_{k} \leq \bar{\beta}, \label{bound_beta}\\
    x_{k}\in\mathcal{S}^0_{\bar{P}} \label{bound_x},\\
     \|\lambda_{k}\|^2\leq  2 \bar{\gamma}(\rho-\rho_0) \label{lambda},\\
    P_{k}\leq \bar{P} - 1, \label{bound_of_P}\\
        P_{k+1}-P_{k}\leq-\frac{\beta_{k+1}}{8}\|\Delta x_{k+1}\|^2-\frac{\beta_{k}}{8}\|\Delta x_{k}\|^2. \label{decrease_Lyapunov} 
  \end{gather}
\end{subequations}
\end{lemma}
\begin{proof}
See Appendix.
\end{proof}

\noindent Next, we show that the dual iterates are bounded and that the Lyapunov sequence $\{P_{k}\}_{k\geq1}$ is bounded from below.
\begin{lemma}\label{bounded_below} Consider  Algorithm \ref{alg1} and let $\{P_{k}\}_{k\geq1}$  defined in  \eqref{lyapunov_function}. If Assumptions \ref{assump1}, \ref{assump2} and \ref{assump3}  hold with $\mathcal{S}=\mathcal{S}_{\bar{P}}^0$ and $\bar{P}$ defined in \eqref{alpha_hat} for any fixed constant $c_0$ and $D_{\bar{P}}$  the radius of $\mathcal{S}^0_{\bar{P}}$. If $\rho$ is chosen as in \eqref{rho_def} and $x_0$ is chosen to satisfy \eqref{eq4}, then  for any $k\geq1$  the following  holds:
\begin{equation}
P_k\geq \ubar{P}-1, \label{bound_from_below}
\end{equation}
where $\ubar{P}$ is  defined in \eqref{lem1}.
\end{lemma}
\begin{proof}
See Appendix.
\end{proof}

\noindent Note that since $\beta_k$ is bounded (see Lemma \ref{bbound}), it follows that $\Gamma_k$ and $c_k$  in Lemma \ref{bounded_gradient} are also bounded. Hence,  in the sequel we denote these bounds as:
\begin{equation}\label{boundgam} 
    \bar{\Gamma}:=\sup_{k\geq1}\{\Gamma_k\} \quad \text{ and } \quad \bar{c}:=\sup_{k\geq1}\{c_{k}\}. 
\end{equation}

\noindent Let us  also bound the full gradient $\nabla P(\cdot)$ (recall that  $P(\cdot)$ is the function defined in \eqref{P}).
\begin{lemma}\label{bounded_grad}[Boundedness of $\nabla P$] Let $\{(x_{k},\lambda_{k})\}_{k\geq1}$ be the sequence generated by  Algorithm \ref{alg1}. If Assumptions \ref{assump1}, \ref{assump2} and \ref{assump3} hold with $\mathcal{S}=\mathcal{S}_{\bar{P}}^0$ and $\bar{P}$ defined in \eqref{alpha_hat} for any fixed constant $c_0$,  $D_{\bar{P}}$ is the diameter of $\mathcal{S}_{\bar{P}}^0$  and $\rho$ is chosen as in \eqref{rho_def}, then we have for any $k\geq1$:
\[
\|\nabla{P}(x_{k+1},\lambda_{k+1},x_{k},\frac{\beta_{k+1}}{2})\|\leq(\bar{\Gamma}+(\bar{c}+1)D_{\bar{P}}+\bar{\beta})\left(\|\Delta x_{k+1}\|+\|\Delta x_{k}\|\right),
\]
where, $\bar{\Gamma}, \bar{c}$ are  defined in \eqref{boundgam} and $\bar{\beta}$ is defined in \eqref{beta_bar_def}.
\end{lemma}
\begin{proof}
See Appendix.
\end{proof}
\noindent The above lemma directly implies the following:
\begin{equation}\label{key_formul}
\|\nabla{P}(x_{k+1},\lambda_{k+1},x_{k},\frac{\beta_{k+1}}{2})\|^2\leq2(\bar{\Gamma}+(\bar{c}+1)D_{\bar{P}}+\bar{\beta})^2\left(\|\Delta x_{k+1}\|^2+\|\Delta x_{k}\|^2\right).
\end{equation}
Then, it follows from \eqref{key_formul} and \eqref{decrease_Lyapunov}, that:
 \begin{equation}\label{rate}
        P_{k+1}-P_{k}\leq-\frac{\ubar{\beta}}{16(\bar{\Gamma}+(\bar{c}+1)D_{\bar{P}}+\bar{\beta})^2}\left\|\nabla{P}(x_{k+1},\lambda_{k+1},x_{k},\gamma_{k+1})\right\|^2.
 \end{equation}
\noindent Let us denote $z_{k}=(x_{k},\lambda_{k})$ and  $u_{k}=(x_{k},\lambda_{k},x_{k-1},\frac{\beta_{k}}{2})$. Moreover,  $\texttt{Stat}P$ denotes the set of stationary points of the function $P(\cdot)$ defined in \eqref{P}. Furthermore, we denote $\mathcal{E}_{k}=P_{k}-P^{*}$, where $P^{*}=\lim_{k\to\infty}{P_{k}}$ (recall that  the sequence $\{P_k\}_{k\geq1}$ is decreasing and bounded from bellow according to \eqref{decrease_Lyapunov} and Lemma \ref{bounded_below}, respectively, hence it is convergent). Denote the set of limit points of $\{u_k\}_{k\geq1}$ by:
\[
\Omega:=\{u^{*}\;:\; \exists \text{ a convergent subsequence} \;  \{u_k\}_{k\in\mathcal{K}} \; \text{such that} \lim_{k\in\mathcal{K}}{u_k}=u^{*}\}.
\]
Let us now prove the following lemma.
\begin{lemma}\label{added_lemma}
Consider Algorithm \ref{alg1} and let $\{P_{k}\}_{k\geq1}$ be defined as in \eqref{P}. If Assumptions \ref{assump1}, \ref{assump2} and \ref{assump3} hold, with $\mathcal{S}=\mathcal{S}_{\bar{P}}^0$ and $\bar{P}$ defined in \eqref{alpha_hat} for any fixed constant $c_0$,  $D_{\bar{P}}$ is the radius of $\mathcal{S}_{\bar{P}}^0$  and $\rho$ is chosen as in \eqref{rho_def}, then the following statements hold:
\begin{enumerate}[(i)]
  \item $\Omega$  is a compact subset of \texttt{Stat}$P$ and   $ \lim_{k\to\infty}{\text{dist}(u_k,\Omega)}=0$.\label{lem_item1}
     \item For any $u\in\Omega,$ we have $P(u)=P^{*}$.\label{lem_item2}
  \item  For any $(x,\lambda, y, \gamma)\in\texttt{Stat}P,$ we have $(x,\lambda)$ a KKT point of \eqref{eq1}. \label{lem_item3}
\end{enumerate}
\end{lemma}
\begin{proof}
    See Appendix.
\end{proof}


\subsection{Global asymptotic convergence}  
Based on the previous lemmas, we first  prove the global asymptotic convergence of the iterates of  Algorithm \ref{alg1}.

\begin{theorem}\label{unused_lemma}[Limit points are KKT points] If  Assumptions \ref{assump1}, \ref{assump2} and \ref{assump3}  hold with $\mathcal{S}=\mathcal{S}_{\bar{P}}^0$ and $\bar{P}$ defined in \eqref{alpha_hat} for any fixed constant $c_0$ and $D_{\bar{P}}$  the diameter of $\mathcal{S}_{\bar{P}}^0$. If $\rho$ is chosen as in \eqref{rho_def} and $x_0$ is chosen to satisfy \eqref{eq4}, then any limit point $(x^*,\lambda^*)$ of the sequence $\{(x_{k},\lambda_{k})\}_{k\geq1}$, generated by  Algorithm \ref{alg1}, is a stationary  point of the augmented Lagrangian function, i.e., $\nabla\mathcal{L}_{\rho}(x^*,\lambda^*)=0$. Equivalently, $(x^*,\lambda^*)$ is a KKT point of problem  \eqref{eq1}:
\[
\nabla f(x^*)+{{J_F}(x^*)}^T\lambda^*=0,\hspace{0.7cm}F(x^*)=0.
\]
\end{theorem}
\begin{proof}
From  \eqref{decrease_Lyapunov}, we have:
\begin{align*}
 \frac{\beta_{k+1}}{8}\|\Delta x_{k+1}\|^2+\frac{\beta_{k}}{8}\|\Delta x_{k}\|^2 \leq P_{k}-P_{k+1} \quad \forall k\geq1.
\end{align*}
Hence, for $k\geq1$, by summing up the above inequality from $i=1$ to $i=k$, we obtain:
\begin{align}
\sum_{i=1}^{k}{\left(\frac{\beta_{i+1}}{8}\|\Delta x_{i+1}\|^2+\frac{\beta_i}{8}\|\Delta x_i\|^2\right)}&\leq P_{1}-P_{k+1}{\overset{{\text{Lemma } \ref{bounded_below}}}{\leq}}P_1-(\ubar{P}-1)\nonumber\\
&{\overset{{\eqref{bound_of_P}}}{\leq}} \bar{P}-2 -\ubar{P}+1\leq \bar{P}-\ubar{P}.\label{limit}
\end{align}
Since \eqref{limit} holds for any $k\geq1$, we have:
\[
\sum_{i=1}^{\infty}{\left(\frac{\beta_{i+1}}{8}\|\Delta x_{i+1}\|^2+\frac{\beta_i}{8}\|\Delta x_i\|^2\right)}<\infty.
\]
This, together with the fact that  $\beta_k\geq\ubar{\beta}>0$, yields that:
\begin{equation}\label{zero_limit}
    \lim_{k\to\infty}{\|\Delta x_{k}\|}=0.
\end{equation}
From \eqref{lambda}, \eqref{bound_x} and the fact that $\mathcal{S}^0_{\bar{P}}$ is compact, it follows that the sequence $\{(x_{k},\lambda_{k})\}_{k\geq1}$ is bounded and  there exists a convergent subsequence, let us say  $\{(x_{k},\lambda_{k})\}_{k\in\mathcal{K}}$, with the limit $(x^*,\lambda^*)$.
From {Lemma \ref{bounded_gradient}} and \eqref{boundgam}, we have:
\begin{align*}
&\|\nabla\mathcal{L}_{\rho}(x^{*},\lambda^{*})\| = \lim_{k\in\mathcal{K}}{\|\nabla\mathcal{L}_{\rho}(x_{k},\lambda_{k})\|}\\
&\leq \bar{\Gamma}\lim_{k\in\mathcal{K}}\left(\|\Delta x_{k}\|+\|\Delta x_{k-1}\|\right)+\bar{c}\lim_{k\in\mathcal{K}}\left(\|\Delta x_k\|+\|\Delta x_{k-1}\|\right)\overset{\eqref{zero_limit}}{=}0.
\end{align*}
Therefore, $ \nabla\mathcal{L}_{\rho}(x^*,\lambda^*)=0$, which completes our proof.\qed
\end{proof}


\subsection{Convergence of the entire sequence under KL}
In this section, under the KL condition,   we prove that the whole sequence  $\{(x_{k},\lambda_{k})\}_{k\geq1}$ generated by   Algorithm \ref{alg1} converges, or, equivalently,   the sequence $\{{\|\Delta x_{k}\|+\|\Delta\lambda_{k}\|}\}_{k\geq1}$ has  finite length.
\begin{theorem}\label{finite_length}Let $\{(x_{k},\lambda_{k})\}_{k\geq1}$ be the sequence generated by Algorithm \ref{alg1}. Let Assumptions \ref{assump1}, \ref{assump2} and \ref{assump3} hold, with $\mathcal{S}=\mathcal{S}_{\bar{P}}^0$ and $\bar{P}$ defined in \eqref{alpha_hat} for any fixed constant $c_0$, and $D_{\bar{P}}$ is the radius of $\mathcal{S}_{\bar{P}}^0$. Moreover, assume that $ P(\cdot)$ defined in \eqref{P} satisfies the KŁ property on $\Omega$.  Then,  $\{z_{k}\}_{k\geq1}=\{(x_{k},\lambda_{k})\}_{k\geq1}$ satisfies the finite length property, i.e.,
\[
\sum_{k=1}^{\infty}{\|\Delta x_{k}\|+\|\Delta\lambda_{k}\|}<\infty,
\]
and consequently the whole sequence $\{(x_{k},\lambda_{k})\}_{k\geq1}$ converges to a KKT point of problem \eqref{alg1}.
\end{theorem}
\begin{proof}
From  the boundedness of $\|\Delta \lambda_{k+1}\|^2$ derived in  \eqref{lambda_squared}, we have:
 \begin{align}\label{llambda}
    \|\Delta\lambda_{k+1}\|^2&\leq c(\bar{\beta})\left(\|\Delta x_{k+1}\|^2+\|\Delta x_{k}\|^2\right)\nonumber\\
    &\leq c(\bar{\beta})\left(\|\Delta x_{k+1}\|^2+\|\Delta x_{k}\|^2\right).
 \end{align}
Adding the term $\|\Delta x_{k+1}\|^2+\|\Delta x_{k}\|^2$ on both sides in \eqref{llambda}, we have:
\begin{align}\label{z_k}
 \|z_{k+1}-z_{k}\|^2&=\|\Delta x_{k+1}\|^2+\|\Delta\lambda_{k+1}\|^2\nonumber\\
 &\leq\|\Delta x_{k+1}\|^2+\|\Delta\lambda_{k+1}\|^2+\|\Delta x_{k}\|^2\nonumber\\
 &{\overset{{\eqref{llambda}}}{\leq}}\big(c(\bar{\beta})+1\big)\left(\|\Delta x_{k+1}\|^2+\|\Delta x_{k}\|^2\right).
 \end{align}
We can then rewrite \eqref{decrease_Lyapunov} as follows: 
\begin{align}\label{llyap}
     P_{k+1}-P_{k}&{\overset{{\eqref{decrease_Lyapunov}}}{\leq}}-\frac{\ubar{\beta}}{8}\left(\|\Delta x_{k+1}\|^2+\|\Delta x_{k}\|^2\right)\nonumber\\
    &{\overset{{\eqref{z_k}}}{\leq}}-\frac{\ubar{\beta}}{8\big(c(\bar{\beta})+1\big)}\|z_{k+1}-z_{k}\|^2.
\end{align}
 Since $ P_{k}\to P^{*}$ and  $\{P_{k}\}_{k\geq 1}$ is monotonically decreasing to $P^{*}$,  it follows that the error sequence $\{\mathcal{E}_{k}\}_{k\geq 1}$, is non-negative, monotonically decreasing and converges to $0$. We distinguish  two cases.\\
\noindent \textbf{{Case 1}}: There exists  $k_1\geq 1$ such that $\mathcal{E}_{k_1}=0$. Then, $\mathcal{E}_{k}=0 \;  \forall k\geq k_1$ and using \eqref{llyap}, we have:
\[
\|z_{k+1}-z_{k}\|^2\leq\frac{8\big(c(\bar{\beta})+1\big)}{\ubar{\beta}}(\mathcal{E}_{k}-\mathcal{E}_{k+1})=0 \hspace{0.3cm}\forall k\geq k_1.
\]
From  Lemmas \ref{bbound} and \ref{bounded_below} the sequence $\{z_{k}\}_{k\geq1}$ is bounded, and thus:
 \[
 \sum_{k=1}^{\infty}{\|\Delta x_{k}\|+\|\Delta\lambda_{k}\|}=\sum_{k=1}^{k_1} \|\Delta x_{k}\|+\|\Delta\lambda_{k}\| < \infty.
 \]
 \noindent \textbf{{Case 2}}: The error $\mathcal{E}_{k}>0 \;  \forall k\geq 1$. Then,  there exists  $k_1=k_1(\epsilon,\tau)\geq 1$  such that $\forall k\geq k_1$ we have $\text{dist}(u_k,\Omega)\leq \epsilon$,  $P^{*}<P(u_k)<P^{*}+\tau$
 and
 \begin{equation}\label{KL}
     \varphi'(\mathcal{E}_{k})\|\nabla P(x_{k},\lambda_{k},x_{k-1},\frac{\beta_{k}}{2})\|\geq1,
 \end{equation}
where $ \epsilon>0, \tau>0$ and $\varphi\in\Psi_{\tau}$ are well defined and correspond to those in Definition \ref{def2} (recall that we assume that   $P(\cdot)$ satisfies the KL property on $\Omega$). Since $\varphi$ is concave, we have $\varphi(\mathcal{E}_{k})-\varphi(\mathcal{E}_{k+1})\geq\varphi'(\mathcal{E}_{k})(\mathcal{E}_{k}-\mathcal{E}_{k+1})$. Then, from \eqref{llyap} and \eqref{KL} we get:
 \begin{align*}
 &\|z_{k+1}-z_{k}\|^2\leq\varphi'(\mathcal{E}_{k})\|z_{k+1}-z_{k}\|^2\|\nabla P(x_{k},\lambda_{k},x_{k-1},\frac{\beta_{k}}{2})\|\nonumber\\
 &\leq\frac{8\big(c(\bar{\beta})+1\big)}{\ubar{\beta}}\varphi'(\mathcal{E}_{k})(\mathcal{E}_{k}-\mathcal{E}_{k+1})\|\nabla P(x_{k},\lambda_{k},x_{k-1},\frac{\beta_{k}}{2})\|\nonumber\\
 &\leq\frac{8\big(c(\bar{\beta})+1\big)}{\ubar{\beta}}\Big(\varphi(\mathcal{E}_{k})-\varphi(\mathcal{E}_{k+1})\Big)\|\nabla P(x_{k},\lambda_{k},x_{k-1},\frac{\beta_{k}}{2})\|.
 \end{align*}
Since $ \|z_{k+1}-z_{k}\|^2={\|\Delta x_{k+1}\|^2+\|\Delta\lambda_{k+1}\|^2}$. Using the fact that for any $a,b,c,d\geq0$, if $ {a^2+b^2}\leq c\times d$, then $ (a+b)^2\leq 2a^2+2b^2\leq 2c\times d\leq c^2+d^2\leq (c+d)^2$, it follows that for any $\theta>0$, we have:
\begin{align}\label{lmit}
    {\|\Delta x_{k+1}\|+\|\Delta\lambda_{k+1}\|}\leq&\frac{8\big(c(\bar{\beta})+1\big)\theta}{\ubar{\beta}}\Big(\varphi(\mathcal{E}_{k})-\varphi(\mathcal{E}_{k+1})\Big)\nonumber\\
    &+\frac{1}{\theta}\|\nabla P(x_{k},\lambda_{k},x_{k-1},\frac{\beta_{k}}{2})\|.
\end{align}
Furthermore, we have:
\begin{align*}
    \| \nabla{P}(x_{k},\lambda_{k},x_{k-1},\frac{\beta_{k}}{2})\|&\leq\|\nabla{\mathcal{L}_{\rho}}(x_{k},\lambda_{k})\|+  
     \bar{\beta}\|x_{k}-x_{k-1}\|\\
    &{\overset{{\eqref{needed1_lemma8}, \eqref{needed2_lemma8}}}{\leq}}(\bar{\Gamma}+(\bar{c}+1)D_{\bar{P}}+\bar{\beta})\left(\|\Delta x_{k}\|+\|\Delta \lambda_{k}\|\right).
\end{align*}
Then, \eqref{lmit} becomes:
\begin{align*}
    {\|\Delta x_{k+1}\|+\|\Delta\lambda_{k+1}\|}\leq&\frac{8\big(c(\bar{\beta})+1\big)\theta}{\ubar{\beta}}\Big(\varphi(\mathcal{E}_{k})-\varphi(\mathcal{E}_{k+1})\Big)\nonumber\\
&+\frac{\bar{\Gamma}+(\bar{c}+1)D_{\bar{P}}+\bar{\beta}}{\theta}\Big(\|\Delta x_{k}\|+\|\Delta\lambda_{k}\|\Big).
\end{align*}
Let us now choose $\theta>0$ so that $0<\frac{\bar{\Gamma}+(\bar{c}+1)D_{\bar{P}}+\bar{\beta}}{\theta}<1$ and define the parameter $\delta_0$ as: $\delta_0=1-\frac{\bar{\Gamma}+(\bar{c}+1)D_{\bar{P}}+\bar{\beta}}{\theta}>0$. Then, by
summing up the above inequality from $k=k_1$ to $k=K$ and using the property: $\sum_{k=k_1}^{K}{\|\Delta x_{k}\|}=\sum_{k=k_1}^{K}{\|\Delta x_{k+1}\|}+\|\Delta x_{k_1}\|-\|\Delta x_{{K+1}}\|$, we get: 
\begin{align*}
   \sum_{k=k_1}^{K}&\|\Delta x_{k+1}\|  +\|\Delta\lambda_{k+1}\|\leq \frac{8\big(c(\bar{\beta})+1\big)\theta}{\ubar{\beta}\delta_0}\Big(\varphi(\mathcal{E}_{k_1})-\varphi(\mathcal{E}_{K+1})\Big)\nonumber\\
&+\frac{\bar{\Gamma}+(\bar{c}+1)D_{\bar{P}}+\bar{\beta}}{\theta\delta_0}\Big(\|\Delta x_{k_1}\|+\|\Delta\lambda_{k_1}\|\Big) \\
& -\frac{\bar{\Gamma}+(\bar{c}+1)D_{\bar{P}}+\bar{\beta}}{\theta\delta_0}\Big(\|\Delta x_{{K+1}}\|+\|\Delta\lambda_{{K+1}}\|\Big).
\end{align*}
Using the fact that  $\{\mathcal{E}_{k}\}_{k\geq k_1}$ is monotonically decreasing and that the function $\varphi$ is positive and increasing, which  yields $\varphi(\mathcal{E}_{k})\geq\varphi(\mathcal{E}_{k+1})>0$,  we get:
\begin{align*}
    \sum_{k=k_1}^{K}{\|\Delta x_{k+1}\|+\|\Delta\lambda_{k+1}\|}\leq&\frac{8\big(c(\bar{\beta})+1\big)\theta}{\ubar{\beta}\delta_0}\varphi(\mathcal{E}_{\ubar{k}})\nonumber\\
&+\frac{\bar{\Gamma}+(\bar{c}+1)D_{\bar{P}}+\bar{\beta}}{\theta\delta_0}\Big(\|\Delta x_{k_1}\|+\|\Delta\lambda_{k_1}\|\Big).
\end{align*}
It is clear that the right-hand side of the above inequality is bounded for any $K\geq k_1$. Letting  $K\to\infty$, we get that:
\[
    \sum_{k=k_1}^{\infty}{\|\Delta x_{k+1}\|+\|\Delta\lambda_{k+1}\|}<\infty.
\]
From Lemma \ref{bbound}, the sequence  $\{(x_{k},\lambda_{k})\}_{k\geq1}$ is bounded. Then, it follows that:
\[
    \sum_{k=1}^{\ubar{k}}{\|\Delta x_{k}\|+\|\Delta\lambda_{k}\|}<\infty.
\]
Hence: $ \sum_{k=1}^{\infty}{\|\Delta x_{k}\|+\|\Delta\lambda_{k}\|}<\infty$. 
Let $m, n\in \mathbf{Z}_{+}$ such that $n\geq m$, we have:
\begin{align*}
    \|z_n-z_m\|=\|\sum_{k=m}^{n-1}{\Delta z_{k+1}}\|
    \leq\sum_{k=m}^{n-1}{\|\Delta z_{k+1}\|}
\leq\sum_{k=m}^{n-1}{\|\Delta x_{k+1}\|+\|\Delta\lambda_{k+1}\|}.
\end{align*}
Since  $ \sum_{k=1}^{\infty}{\|\Delta x_{k+1}\|+\|\Delta\lambda_{k+1}\|}<\infty$, it follows that $\forall \varepsilon>0, \exists N\in\mathbf{Z}_{+}$ such that $\forall m, n \geq N$, where $n\geq m$, we have: $ \|z_n-z_m\|\leq\varepsilon$. This implies that $\{z_k\}_{k\geq1}$ is a Cauchy sequence and thus converges. Moreover, by Theorem \ref{unused_lemma}, the whole sequence $\{z_k\}_{k\geq1} = \{(x_{k},\lambda_{k})\}_{k\geq1}$  converges to a KKT point of problem \eqref{eq1}. This concludes our proof. \qed
\end{proof}
\begin{remark}
Theorem \ref{finite_length} shows that, under the KL property, the entire sequence $\{(x_k, \lambda_k)\}_{k \geq 1}$ generated by Algorithm~\ref{alg1} converges. Consequently, the set of its limit points reduces to a singleton, which, according to Theorem \ref{unused_lemma} is a KKT point of problem \eqref{eq1}. Convergence rates can be easily derived for some particular choices of the desingularization function $\varphi$ as e.g., in  \cite{ElbNecPan:25}. 
\end{remark}


\subsection{First-order complexity}

\noindent Let us now present another important result of this paper, which derives  the computational  complexity of  Algorithm \ref{alg1} to find an $\epsilon$-first-order solution of problem~\eqref{eq1}.

\begin{theorem}\label{main_result}[First-order complexity]
Consider Algorithm \ref{alg1} and let $\{P_{k}\}_{k\geq1}$ be defined as in \eqref{lyapunov_function}. If Assumptions \ref{assump1}, \ref{assump2} and \ref{assump3} hold with $\mathcal{S}=\mathcal{S}_{\bar{P}}^0$ and $\bar{P}$ defined in \eqref{alpha_hat}, $\rho$ is chosen as in \eqref{rho_def} and $\ubar{\beta}$ is such that $\ubar{\beta}=\mathcal{O}\left(\sqrt{\rho}\right)$, then for any $\epsilon>0$, Algorithm \ref{alg1} yields an $\epsilon$-first-order solution of problem \eqref{eq1} after $K=\mathcal{O}\left(\frac{\sqrt{\rho}}{\epsilon^2}\right)$ Jacobian evaluations.
\end{theorem}

\begin{proof}
According to Theorem \ref{unused_lemma}, we have  $\lim_{k\in\mathcal{K}}{\|\nabla\mathcal{L}_{\rho}(x_{k},\lambda_{k})\|} =0$. Let  $k^*\geq 1$ be the first integer such that:
\begin{equation}
\label{eq:24}
    \|\nabla\mathcal{L}_{\rho}(x_{k^*+1}, \lambda_{k^*+1})\|\leq \epsilon.
\end{equation}
From Lemma \ref{bounded_gradient}, it follows that at each iteration $k \in [1:k^*]$ we encounter one of the following two cases:\\
\noindent \textit{Case 1: }  $$\Gamma_{k+1} \|\Delta x_{k+1}\| + \Gamma_{k} \|\Delta x_{k}\|\geq c_{k+1} \|\Delta x_{k+1}\|^2 + c_{k} \|\Delta x_{k}\|^2,$$ then we get:
\begin{align} \label{case1}
& \|\nabla\mathcal{L}_{\rho}(x_{k+1},\lambda_{k+1})\|^2 \leq 8 \Gamma_{k+1}^2 \|\Delta x_{k+1}\|^2 + 8 \Gamma_k^2 \|\Delta x_k\|^2 \nonumber\\
&\overset{\eqref{decrease_Lyapunov}}{\leq} \left(\frac{64\Gamma_{k+1}^2}{\beta_{k+1}} + \frac{64\Gamma_k^2}{\beta_k}\right)\left(P_k - P_{k+1}\right)  \overset{\eqref{boundgam}}{\leq} \frac{128\bar{\Gamma}^2}{\ubar{\beta}}\left(P_k - P_{k+1}\right).
\end{align}
\noindent \textit{Case 2: } Otherwise,  the following is valid: 
$$\Gamma_{k+1} \|\Delta x_{k+1}\| + \Gamma_{k} \|\Delta x_{k}\| < c_{k+1} \|\Delta x_{k+1}\|^2 + c_{k} \|\Delta x_{k}\|^2,$$ 
which yields:
\begin{align}\label{case2}
&\|\nabla\mathcal{L}_{\rho}(x_{k+1},\lambda_{k+1})\|\leq 2 c_{k+1} \|\Delta x_{k+1}\|^2 + 2 c_k \|\Delta x_k\|^2 \nonumber\\
&\overset{\eqref{decrease_Lyapunov}}{\leq} \left(\frac{16c_{k+1}}{\beta_{k+1}} + \frac{16 c_k}{\beta_k}\right)\left(P_k - P_{k+1}\right) \overset{\eqref{boundgam}}{\leq} \frac{32\bar{c}}{\ubar{\beta}} \left(P_k - P_{k+1}\right).
\end{align}
Define $\mathcal{I}_1$ as the set of iterations in $ [1:k^*-1]$ at which   \textit{Case 1} holds, and  $\mathcal{I}_2$ as the set of iterations in $[1:k^*-1]$ at which   \textit{Case 2} holds. Clearly: $k^*=|\mathcal{I}_1|+|\mathcal{I}_2|+1$. We first derive an upper bound for $|\mathcal{I}_1|$. Summing \eqref{case1} over $\mathcal{I}_1$ yields:
\begin{align*}
 |\mathcal{I}_1|\epsilon^2  & \overset{\eqref{eq:24}}{<} \sum_{k\in\mathcal{I}_1} \|\nabla\mathcal{L}_{\rho}(x_{k+1},\lambda_{k+1})\|^2  \overset{\eqref{case1}}{\leq}  \sum_{k\in\mathcal{I}_1} \frac{128 \bar{\Gamma}^2}{\ubar{\beta}}\left(P_k - P_{k+1}\right)\\
   & \leq \frac{128 \bar{\Gamma}^2}{\ubar{\beta}} \sum_{k=1}^{k^*\!-1} \left(P_k - P_{k+1}\right) = \frac{128 \bar{\Gamma}^2}{\ubar{\beta}}  \left(P_1 - P_{k^*}\right)   {\overset{\eqref{bound_of_P}, \eqref{bound_from_below}}{\leq}}\frac{128 \bar{\Gamma}^2}{\ubar{\beta}} \left(\bar{P}-\ubar{P}\right).
\end{align*}
Thus, we have: $|\mathcal{I}_1|< \frac{128 \bar{\Gamma}^2\left(\bar{P}-\ubar{P}\right)}{\ubar{\beta}\epsilon^2} $. Similarly, we derive an upper bound for $|\mathcal{I}_2|$. Summing  \eqref{case2} over $\mathcal{I}_2$ yields:
\begin{align*}
 |\mathcal{I}_2|\epsilon  & \overset{\eqref{eq:24}}{<} \sum_{k\in\mathcal{I}_2} \|\nabla\mathcal{L}_{\rho}(x_{k+1},\lambda_{k+1})\|  \overset{\eqref{case2}}{\leq} \sum_{k\in\mathcal{I}_2} \frac{32 \bar{c}}{\ubar{\beta}}\left(P_k - P_{k+1}\right)\\
   & \leq \frac{32 \bar{c}}{\ubar{\beta}} \sum_{k=1}^{k^*\!-1} \left(P_k - P_{k+1}\right) = \frac{32 \bar{c}}{\ubar{\beta}}  \left(P_1 - P_{k^*}\right)   {\overset{\eqref{bound_of_P}, \eqref{bound_from_below}}{\leq}}\frac{32 \bar{c}}{\ubar{\beta}} \left(\bar{P}-\ubar{P}\right).
\end{align*}
Therefore, we obtain: $|\mathcal{I}_2|< \frac{32 \bar{c}\left(\bar{P}-\ubar{P}\right)}{\ubar{\beta}\epsilon} $.  Consequently, we have:
\[
k^* \leq \left(\bar{P}-\ubar{P}\right) \left({\frac{128\bar{\Gamma}^2}{\ubar{\beta}\epsilon^2}} + \frac{32\bar{c}}{\ubar{\beta}\epsilon} \right).
\]
Note that $\bar{\Gamma} = \mathcal{O}\left(\sqrt{\rho}\right)$ and $\bar{c} = \mathcal{O}\left(\rho\right)$ (see Lemma \ref{bounded_gradient}, \eqref{boundgam} and definition of $\bar{\beta}$ in \eqref{beta_bar_def}). Hence, assuming $\epsilon\leq1$, we get:
\[
k^* \leq \mathcal{O}\left( \frac{\rho}{\ubar{\beta}}\frac{1}{\epsilon^2} + \frac{\rho}{\ubar{\beta}}\frac{1}{\epsilon}\right).
\]
Consequently, if we fix $\ubar{\beta}=\mathcal{O}({\sqrt{\rho}})$, then after $K=\mathcal{O}\left(\frac{\sqrt{\rho}}{\epsilon^2}\right)$ Jacobian evaluations,  Algorithm \ref{alg1} yields  an $\epsilon$-first-order solution of optimization problem \eqref{eq1}. This concludes our proof. \qed
\end{proof}

\noindent From the previous theorems, one can see that, in addition to its straightforward implementation,   Algorithm \ref{alg1} also enjoys global convergence results, giving it an advantage over approaches where only local convergence can be guaranteed, such as SCP schemes \cite{MesBau:21}. Moreover, our method guarantees global  convergence to an $\epsilon$-first-order solution in at most $\mathcal{O}(\sqrt{\rho} \epsilon^{-2})$ Jacobian evaluations, which, to the best of our knowledge, \textit{is the optimal complexity in the context of augmented Lagrangian and penalty-based methods for smooth nonconvex constrained optimization problems}, as the penalty parameter $\rho$ enters  under the square root and the desired accuracy $\epsilon$ enters quadratically  in the algorithm's complexity \cite{ElbNec:25,LiuLin:25}. Our convergence rate greatly \textit{improves the existing complexity results} for augmented Lagrangian type methods, measured through the Jacobian evaluations, on the same class of problems: e.g., $\mathcal{O}(\epsilon^{-5.5})$ in \cite{XieWri:21};  $\mathcal{O}(\epsilon^{-4})$ in \cite{SahEft:19}; or  $\mathcal{O}(\epsilon^{-3})$ recently derived in \cite{ElbNecPan:25}.  Another key advantage lies in its avoidance of calling complicated subroutines, as the \textit{unconstrained}  subproblem in L-AL algorithm has a \textit{quadratic} strongly convex objective function, making it  remarkably efficient compared to e.g., \cite{XieWri:21, AndBir:08, SahEft:19}, where the subproblem is highly nonconvex.   Hence, its simplicity and effectiveness make it an attractive algorithm for a wide range of large-scale practical applications.


\subsection{Special case of $F$ affine function}

In this section, we analyze the affine case, i.e.,  $F(x) = Ax - b$, with $A \in \mathbb{R}^{m \times n}$ having full row rank, equivalently,  $\sigma_{\min}(A) = \sigma>0$. We study how this assumption affects the complexity of our algorithm for solving problem \eqref{eq1} and compare it with the lower bounds for smooth nonconvex problems with affine equality constraints derived in \cite{LiuLin:25} for primal first-order methods. Note that since $F$ is affine, the smoothness constant $L_F = 0$. Below we briefly present  the corresponding  results  for Lemmas \ref{lambda_bound}, \ref{lemma3}, and \ref{bounded_gradient}, along with the updated choice of $\rho$ in \eqref{rho_def}, and explain how the complexity bound is modified in this case. The proofs of these results follow similar reasoning as for their counterparts in the general case, with the simplification that $L_F = 0$ and $F(x_{k+1}) = l_F(x_{k+1}, x_k)$. Therefore, we omit their proofs. We begin bounding $\|\Delta \lambda_{k+1}\|$ (see Lemma \ref{lambda_bound}):
\begin{align}
    \label{lambda_squared_linear}
    \|\Delta\lambda_{k+1}\|^2 \leq c(\beta_{k+1}) \|\Delta x_{k+1}\|^2 + c(\beta_k) \|\Delta x_k\|^2,
\end{align}
where now \; $c(\beta) = \frac{2 (L_f +\beta)^2}{\sigma^2}$.  Next, we provide the counterpart of Lemma \ref{lemma3}: if $\beta_{k+1}$ is chosen such that 
\begin{equation} \label{eq_assu_linear}
    \beta_{k+1} \geq L_f,
\end{equation}
then inequality \eqref{decrease_algorithm} holds. We also state the corresponding result of Lemma \ref{bounded_gradient}:
\[
    \|\nabla \mathcal{L}_{\rho}(x_{k+1}, \lambda_{k+1})\| \!\leq\! \Gamma_{k+1} \|\Delta x_{k+1}\| + \Gamma_k \|\Delta x_k\|, \text{where} \; 
    \Gamma_k \!=\! \frac{M_F \!+\! \frac{1}{\rho}}{\sigma} \! \left(L_f + 2 \beta_k\right).
\]

\noindent Hence, the corresponding choice of $\rho$ for the bound in \eqref{rho_def} in the case when $F$ is affine becomes:
\begin{align}\label{rho_def_linear}
 \rho \geq \!\max\Bigg\{1,&\ \rho_0\!+\!1,\ \frac{1}{M_F},\ \rho_0 \!+\! \frac{12 M_f^2}{\sigma^2},\ 2\rho_0\! +\!\frac{48\mu L_f}{\sigma^2},\ \frac{48(1\!+\!\mu)^2L_f }{\mu \sigma^2},\nonumber\\
&\rho_0+\frac{2M_f^2  + 4 \mu^2L_f^2D_{\bar{P}}^2}{\sigma^2}\Bigg\}.
\end{align}
With the above choice, Lemmas \ref{bbound} and \ref{bounded_below} follow with the same constants as before, but now using that $L_F = 0$. Hence, $\bar{P}$, $\ubar{P}$ and the  diameter $D_{\bar{P}}$, remain unchanged, and we have $\bar{\beta} = \mu L_f$ and $\bar{\gamma} = 1$. We are now ready to provide the complexity bound of our Algorithm \ref{alg1} in the case when $F$ is affine (we denote  $\kappa_A = \frac{M_F}{\sigma}$  the condition number of the matrix $A$).

\begin{corrolary}
    \label{main_result_linear}[First-order complexity: affine $F$]
Consider Algorithm \ref{alg1}, and let $\{P_k\}_{k\geq1}$ be defined as in \eqref{lyapunov_function}. If Assumptions \ref{assump1}, \ref{assump2} and \ref{assump3} hold  for problem \eqref{eq1} with affine constraints,  $\mathcal{S} = \mathcal{S}_{\bar{P}}^0$ and $\bar{P}$ defined in \eqref{alpha_hat}, $\rho$ chosen as in \eqref{rho_def_linear}, and $\ubar{\beta} = \mathcal{O}(L_f)$, then for any $\epsilon > 0$, Algorithm \ref{alg1} yields an $\epsilon$-first-order solution of \eqref{eq1} after
\[
K = (\bar{P} - \ubar{P}) \cdot \frac{128 \bar{\Gamma}^2}{\ubar{\beta} \epsilon^2} = \mathcal{O}\left(\frac{\kappa_A^2 L_f}{\epsilon^2} (\bar{P} - \ubar{P})\right)
\]
matrix-vector multiplications with $A$ and $A^T$.
\end{corrolary}

\begin{remark}
Note that the  complexity bound  from Corrollary \ref{main_result_linear} is similar to the lower bound derived in \cite{LiuLin:25} for the same problem class, i.e., problem  \eqref{eq1} with affine $F$, when using a primal first-order method that projects (inexactly) into the feasible set $\{x : Ax = b\}$ via matrix-vector multiplications with $A$ and $A^T$. More precisely, the  lower bound in \cite{LiuLin:25} for an inexact projected first-order method is of order $\mathcal{O}\left(\frac{\kappa_A L_f}{\epsilon^2} (f(x_0) - \ubar{f})\right)$ matrix-vector multiplications with $A$ and $A^T$. We believe that if one adds in Algorithm \ref{alg1} an extrapolation step  either in the dual variables (see \cite{KeMa:17}), in the primal variables (see  \cite{SunLiu:17}), or in both (see \cite{BotCse:23}), then it may be possible to achieve a complexity with condition number dependence, $\kappa_A$, matching the lower bound derived in \cite{LiuLin:25}.
\end{remark}


\subsection{Selection of the penalty parameter $\rho$}
The results above, which describe the total number of Jacobian evaluations required to find an $\epsilon$-first-order solution to the problem, assume that the penalty parameter $\rho$ exceeds a certain threshold, specifically the one given in \eqref{rho_def}. However, determining this threshold in advance is challenging, as it depends on unknown parameters of the functions involved in the problem as well as the algorithm's settings. To address this issue, we propose a scheme that allows for the determination of a sufficiently large $\rho$ without requiring explicit knowledge of these parameters. Inspired by Algorithm 3 in \cite{XieWri:21}, our approach repeatedly invokes Algorithm \ref{alg1} within an inner loop. If Algorithm \ref{alg1} fails to converge within a given number of iterations, we increase geometrically the penalty parameter $\rho$ by a constant multiple in the outer loop. The full implementation of this procedure is provided in Algorithm \ref{alg2}.

\begin{algorithm} 
\caption{L-AL Method with Trial Values of $\rho$}
\label{alg2}
\begin{algorithmic}[1]
\State \textbf{Initialization:}  $(x_{-1}^{*}, \lambda_{-1}^{*}) \in \mathbb{R}^n  \times \mathbb{R}^m$, $\mu,\tau, \eta > 1$, $\epsilon > 0$, $\rho_{0} > 1, \ubar{\beta}>0$ and $K_0 > 0$.
\State $t \gets 0$
\While{$\epsilon$-KKT conditions are not satisfied}
    \State Call Algorithm \ref{alg1} with $\rho_t\geq 1$ and $\mu, \beta_t^0\geq\ubar{\beta}$.
   \Statex{~~~~}
 Warm start with $(x_t^0, \lambda_t^0) \gets (x_{t-1}^{*}, \lambda_{t-1}^{*})$ for $K_t$  iterations of Algorithm 1.
    \State Update $\rho_{t+1} \gets \tau \rho_t$ and $K_{t+1} \gets \eta K_t$.
    \State $t \gets t+1$
\EndWhile
\end{algorithmic}
\end{algorithm}

\medskip 

\noindent Algorithm \ref{alg2} is well-defined and terminates in a finite number of iterations, provided that the parameter $\tau > 1$ and $\eta>1$. Specifically, during the $l$-th stage of Algorithm \ref{alg2}, we have $\rho_{l+1} = \tau^{l+1} \rho_0$ and $K_{l+1} = \eta^{l+1} K_0$. Let $M$ denote the maximum bound in \eqref{rho_def} that $\rho$ must exceed, with $1 \leq M < \infty$ (this is finite since all bounds depend only on constants from the problem data and the algorithm's parameters). Also, let $N$ denote the total number of iterations required to obtain an $\epsilon$-first-order solution of problem \eqref{eq1}. Note that $N \leq \mathcal{O}\left(\frac{\sqrt{\rho}}{\epsilon^2}\right)$, where $\rho$ is a value of the penalty parameter satisfying \eqref{rho_def}.
Consequently, we have:
\[
\rho_{l+1} = \tau^{l+1} \rho_0 \geq M \quad \text{and} \quad K_{l+1} = \eta^{l+1} K_0 \geq N
\]
provided that 
\[
l+1 \geq \max\left\{\frac{\log\left(\frac{M}{\rho_0}\right)}{\log(\tau)}, \frac{\log\left(\frac{N}{K_0}\right)}{\log(\eta)}\right\}.
\]
Therefore, $\rho$ needs to be increased at most $\frac{\log\left(\frac{M}{\rho_0}\right)}{\log(\tau)}$ times to meet the threshold in \eqref{rho_def}. Hence, Algorithm \ref{alg2} yields an $\epsilon$-first-order solution of problem \eqref{eq1} after at most  $\max\left\{\frac{\log\left(\frac{M}{\rho_0}\right)}{\log(\tau)}, \frac{\log\left(\frac{N}{K_0}\right)}{\log(\eta)}\right\} $ calls of Algorithm \ref{alg1}.


\section{Improved convergence under  strict saddle property} \label{sec5}
In this section, we explore the impact of the \textit{strict saddle property} (also called benign nonconvexity) \cite{GoyRoy:24} on the convergence rate of our augmented Lagrangian algorithm.  To achieve second-order convergence rates, we must impose additional  structure on the problem \eqref{eq1} such as the strict saddle property. For the purposes of this analysis,  in this section, we assume that the functions  $f$ and $F_i$ for all  $i=1, \ldots, m$ in the optimization problem  \eqref{eq1} are twice continuously differentiable.  We begin this analysis by adapting our previous assumptions to the new setting of problem \eqref{eq1} studied in this section.

\begin{assumption}
\label{as:Assu_SSF}
For any compact set $\mathcal{S}\subseteq\mathbb{R}^n$, there exist positive constants $M_f$, $M_F$, $L_f$, $L_F$, $H_f$ and $M_{F_i}$, $L_{F_i}$, $H_{F_i}$ for all $i=1,\ldots, m$,  such that $f$ and $F$ satisfy the following conditions:
\begin{enumerate}[(i)]
\item $\Vert \nabla f(x)\Vert \leq M_f$,\;  $\Vert \nabla^2    f(x)\Vert \leq L_f$ and $\Vert \nabla^2 f(x)-\nabla^2 f(y)\Vert \leq H_f\Vert x-y\Vert$  for all $x, y\in\mathcal{S}$ \label{4.1}.
\item $\Vert {J_F}(x)\Vert \leq M_F$ and $\Vert {J_F}(x)-{J_F}(y)\Vert \leq L_F\Vert x-y\Vert$ for all $x, y\in\mathcal{S}$.
\item  $\Vert \nabla{F}_i(x)\Vert \leq M_{F_i}$, \; $\Vert \nabla^2{F}_i(x)\Vert \leq L_{F_i}$ and $\Vert \nabla^2{F}_i(x)-\nabla^2{F}_i(y)\Vert \leq H_{F_i}\Vert x-y\Vert$ for all $x, y\in\mathcal{S}$.
\end{enumerate}
\end{assumption}


\noindent Note that Assumption \ref{as:Assu_SSF}.\textrm{(iii)} implies Assumption \ref{as:Assu_SSF}.\textrm{(ii)}, but possibly with more conservative constants.  Moreover, Assumption \ref{as:Assu_SSF} requires the gradients and the Hessians of the objective function and of the functional constraints to be locally Lipschitz continuous.   Next, we define the class of problems that has a strict saddle function structure on a given set $\mathcal{X} \subseteq \mathbb{R}^n$. 
The definition is inspired by \cite{GoyRoy:24} and encompasses many real-world applications such as deep learning, matrix factorization  and inverse problems.

\medskip 

\begin{definition}[Strict saddle function]\label{strict_saddle}
 Let $\varphi : \mathbb{R}^n \rightarrow \mathbb{R}$ be twice differentiable and let $\alpha$, $\theta$, $\gamma$ and $\xi$ be given positive constants.
The function $\varphi $ is $(\alpha, \theta, \gamma, \xi)$-strict saddle on $\mathcal{X} \subseteq\mathbb{R}^n$ if the subset $\mathcal{X}$ is decomposed into $\mathcal{X} = \mathcal{R}_1 \cup \mathcal{R}_2 \cup \mathcal{R}_3$, where
\begin{equation*}
\arraycolsep=0.2em
\begin{array}{lcl}
\mathcal{R}_1 & = & \{x \in  \mathbb{R}^n : \Vert \nabla \varphi (x)\Vert  \geq \alpha \}, \vspace{1ex} \\
\mathcal{R}_2 & = &  \{x \in \mathbb{R}^n : \lambda_{\text{min}}(\nabla^2 \varphi (x)) \leq -\theta \}, \vspace{1ex}\\
\mathcal{R}_3 & = & \Big\{x \in \mathbb{R}^n : \exists \;  \textrm{local minimizer $x^{*}$ of $\varphi$ such that $\mathrm{dist}(x, x^{*}) \leq \xi$ and $\varphi$ is} \vspace{-0.25ex}\\
&& \hspace{1.5cm} \textrm{ $\gamma$-strongly convex over set $\{ y \in  \mathbb{R}^n : \mathrm{dist}(y, x^{*}) \leq \xi \}$} \Big\}.
\end{array}
\end{equation*}
\end{definition}
From Definition \ref{strict_saddle} it follows that if $\varphi(\cdot)$ is a strict saddle function, then either the  gradient is sufficiently large,  the  Hessian has a sufficiently negative curvature, or $x$ is close to an \textit{isolated} local minimum.  Note that, the regions $\mathcal{R}_2$ and $\mathcal{R}_3$ are mutually exclusive, but the first region $\mathcal{R}_1$ may occur simultaneously with one of the other two. We can easily compute the gradient and Hessian of  $\mathcal{L}_{\rho}(\cdot, \lambda)$ as follows:
\begin{align}
\label{eq:gradhes}
& \nabla_x \mathcal{L}_{\rho}(x,\lambda)  =  \nabla f(x) + J_F(x)^T\left(\lambda + \rho F(x)\right), \\
& \nabla^2_{xx}\mathcal{L}_{\rho}(x,\lambda)  =  \nabla^2 f(x) + \sum_{i=1}^{m} \left(\lambda + \rho F(x)\right)_i \nabla^2 f_i(x) + \rho J_F^T(x) J_F(x). \nonumber 
\end{align}
It follows immediately that  \(\nabla_x \mathcal{L}_{\rho}\) is locally Lipschitz w.r.t. $x$ for any fixed $\lambda$ with the Lipschitz constant: 
\begin{equation*} 
L_{\rho}  \triangleq \sup_{(x,y) \in \mathcal{S} \times \Lambda} \left\{L_f + L_F \Vert \lambda  + \rho F(x)\Vert  + \rho M_F^2 \right\}, 
\end{equation*}
where  $\Lambda \subseteq \mathbb{R}^m$  is  any compact set (containing the dual variables).   By Assumption \ref{as:Assu_SSF}, we can also prove that the Hessian  \(\nabla^2_{xx} \mathcal{L}_{\rho}\) is locally Lipschitz with the Lipschitz constant: 
\begin{equation*} 
H_{\rho} \triangleq \sup_{(x,y) \in \mathcal{S} \times \Lambda } \Big\{H_f +  \Vert \lambda + \rho F(x)\Vert _{\infty} \cdot \sum_{i=1}^{m}{H_{f_i}} + \rho \Big(2M_FL_F+\sum_{i=1}^{m}{M_{f_{i}}L_{f_i}} \Big) \Big\}.
\end{equation*}
In what follows, we also make the following assumption.

\begin{assumption}\label{assum_strict_sadlle}
For fixed $\epsilon > 0$, we assume that $\mathcal{L}_{\rho}(\cdot,\lambda)$  is an $(\alpha, \theta \rho^{\zeta_1},  \gamma \rho^{\zeta_2}, \xi)$-strict saddle function on an $\epsilon$-feasible set $\{ x \; : \; \Vert F(x)\Vert \leq \epsilon \}$, where $\zeta_1,\zeta_2 \in[0,1]$ and $\lambda \in \Lambda \subseteq\mathbb{R}^m$ is given.
\end{assumption}
Note that since the augmented Lagrangian function $\mathcal{L}_{\rho}$ depends on $\rho$, it is reasonable to assume that the strict saddle parameters related to its Hessian, i.e., $\theta$ and $\gamma$ also depend  on $\rho$. 
Note that when $\zeta_1= \zeta_2=0$, we cover the case when these parameters are independent on $\rho$.  We modify Algorithm~\ref{alg1}  to obtain a new variant presented in Algorithm \ref{alg:ALSSF_alg} below.  Before presenting our adapted algorithm, let us first introduce the following quadratic approximation of $\mathcal{L}_{\rho}$:
\begin{equation*} 
\hat{\mathcal{Q}}_{\mathcal{L}_{\rho}}(x,\lambda;\bar{x}):=\mathcal{L}_{\rho}(\bar{x},\lambda)+\iprods{\nabla_x\mathcal{L}_{\rho}(\bar{x},\lambda), x - \bar{x}} + \frac{1}{2}(x-\bar{x})^T\left(\nabla_{xx}^2 \mathcal{L}_{\rho}(\bar{x},\lambda)\right)(x-\bar{x}), 
\end{equation*}
for all $x, \bar{x}, \lambda$. Now, we are ready to present Algorithm \ref{alg:ALSSF_alg} that  exploits the strict saddle property.

\begin{algorithm}
\caption{(Augmented Lagrangian for Strict Saddle Functions (ALSSF))}\label{alg:ALSSF_alg} 
\begin{algorithmic}[1]
\State \textbf{Initialization:}  $x^0\in \mathbb{R}^n$ and $\lambda^0\in\mathbb{R}^m$ and parameters $\epsilon>0$, $\rho > 0$, $\beta>0$ and $\upsilon > 0$.
\For{$k = 0, \ldots, K$}
	\If{$\Vert \nabla_x \mathcal{L}_{\rho}(x^k,\lambda^k)\Vert  > \alpha$}
	\State $x_{k+1}\gets\argmin_{x}{\bar{\mathcal{L}}_{\rho}(x,\lambda_{k};x_{k})+\frac{\beta}{2}{\|x-x_{k}\|}^2}$
	\Else    
	\If{$\lambda_{\text{min}}\left(\nabla_{xx}^2 \mathcal{L}_{\rho}(x^k,\lambda_{k})\right) < -\theta  \rho^{\zeta_1}$}
	\State $x^{k+1}\gets {\displaystyle\argmin_{x \in \mathbb{R}^n}}  \hat{\mathcal{Q}}_{\mathcal{L}_{\rho}}(x,\lambda_{k};x^k)+  \frac{\upsilon}{6}\norms{x - x^k}^3 $
	\Else  
	\State $x^{k+1}\gets {\displaystyle\argmin_{x \in \mathbb{R}^n}} \iprods{\nabla f(x^k) + J_F(x^k)^T(\lambda_{k} + \rho F(x^k)) , x - x^k} +  \frac{\beta}{2}\norms{x - x^k}^2 $
	\EndIf
	\EndIf
	\If{$\Vert F(x^{k+1})\Vert >\epsilon$} 
	\State $\lambda^{k+1}\gets \lambda^k+\rho F(x^{k+1})$
	\Else   
	\State $\lambda^{k+1}\gets \lambda^k$
	\EndIf
\EndFor
\end{algorithmic}
\end{algorithm}

\medskip 

\noindent Let us explain Algorithm~\ref{alg:ALSSF_alg}.
Leveraging Assumption \ref{assum_strict_sadlle}, the augmented Lagrangian function $\mathcal{L}_{\rho}$ satisfies the strict saddle  property.  Accordingly, we employ specific updates tailored to each of the three characteristic regions outlined by this property: 
\begin{itemize}
\item In the region $\mathcal{R}_1$, where the gradient norm of $\Lc_{\rho}$ is substantial (Line 3), we apply the primal update from our  Algorithm \ref{alg1} (Line 5).  This update is simple, has a closed form, and ensures a decrease in the augmented Lagrangian function $\Lc_{\rho}$. 
\item In contrast, when the current iterate lies in the region $\mathcal{R}_2\setminus\mathcal{R}_1$ (Line 6), characterized by the presence of strict saddle points, we employ a cubic regularization of the Newton method from \cite{Nes:18} to produce the next primal iterate (Line 7). 
This update effectively ensures that strict saddle points are avoided. 
Note that in the region $\mathcal{R}_2$, $\Lc_{\rho}$ is nonconvex.
Hence, we need a cubic regularized Newton update  instead of a gradient-type scheme, since the former can escape strict saddle points. 

\item In the region $\mathcal{R}_3\setminus\mathcal{R}_1$, where the augmented Lagrangian function is strongly convex (Line 8), we utilize the standard gradient method (instead, one can also use an accelerated gradient algorithm) \cite{Nes:18} (Line 9). 
The gradient method not only guarantees global convergence once in this region but it is also computationally efficient. 

\item Finally, the dual multipliers are updated classically  outside the approximate feasible region (Line 13). Otherwise, they are kept unchanged (Line~15).  
\end{itemize}

\medskip 

\noindent The following lemma proves that Algorithm \ref{alg:ALSSF_alg} guarantees that after only one iteration, we reach an $\epsilon$-approximate feasible region and the iterates remain there afterwards.  

\begin{lemma}\label{le:feasibility}
Suppose that $\{(x^{k},\lambda^{k})\}$  generated by  Algorithm \ref{alg:ALSSF_alg} is bounded and  $\rho \geq \frac{2M}{\epsilon}$, where $\epsilon>0$ and $M>0$ is such that $\Vert \lambda^k\Vert \leq M$ for any $k\geq 0$. Then, for all $k \geq 1$, we have
\begin{equation*}
    \Vert F(x^k)\Vert \leq \epsilon.
\end{equation*}
\end{lemma}

\begin{proof}
See Appendix.
\end{proof}

\noindent From  Lemma \ref{le:feasibility}, it follows that Algorithm~\ref{alg:ALSSF_alg} updates only once the dual variables. 
Hence, $\lambda^k = \lambda^1$ for all $k \geq 1$.  
Moreover, Lemma \ref{le:feasibility} together with Assumption \ref{assum_strict_sadlle}, shows that  the function  $\mathcal{L}_{\rho}(\cdot,\lambda^k)$ is $(\alpha, \theta  \rho^{\zeta_1}, \gamma\rho^{\zeta_2}, \delta)$-strict saddle  for any $k \geq 1$.   Let us now prove that if $x^k\in\mathcal{R}_1$, then $\mathcal{L}_{\rho}(\cdot,\lambda^k)$ strictly decreases at the next iterate.

\begin{lemma}\label{le:lemma_R1}
Let $\{(x^{k},\lambda^{k})\}$  be generated by  Algorithm \ref{alg:ALSSF_alg} and $\epsilon>0$.  Suppose that Assumptions \ref{as:Assu_SSF} and  \ref{assum_strict_sadlle} hold on a compact set $\mathcal{S}$ on which the primal iterates belong to, there exists  $M>0$ such that $\Vert \lambda^k\Vert \leq M$ for any $k\geq 0$,  $\rho\geq\frac{2M}{\epsilon}$,  $x^k \in \mathcal{R}_1$  and $\beta\geq L_{\rho}$. Then, for all $k \geq 1$, we have
\begin{equation*}
\mathcal{L}_{\rho}(x^{k+1},\lambda^{k+1})-\mathcal{L}_{\rho}(x^{k},\lambda^{k})\leq -\frac{\alpha^2}{8\beta}.
\end{equation*}
\end{lemma}

\begin{proof}
See Appendix.
\end{proof}

\noindent Next, we prove that if  $x^k\in\mathcal{R}_2\setminus\mathcal{R}_1$, i.e.,  the region in which the cubic regularized Newton update (Line 7) is used, then $\Lc_{\rho}$ also strictly decreases at the next iterate.

\begin{lemma}\label{le:lemma_R2}
Let $\{(x^{k},\lambda^{k})\}$  be generated by  Algorithm \ref{alg:ALSSF_alg} and $\epsilon>0$. 
Suppose that Assumptions \ref{as:Assu_SSF} and  \ref{assum_strict_sadlle} hold on a compact set $\mathcal{S}$ on which the primal iterates belong to, there exists  $M>0$ such that $\Vert \lambda^k\Vert \leq M$ for any $k\geq 0$, $\rho \geq \frac{2M}{\epsilon}$,  $H_{\rho} \leq \upsilon \leq 2 H_{\rho}$ and   $x^k \in \mathcal{R}_2\setminus\mathcal{R}_1$. 
Then, for all $k \geq 1$, we have
\begin{equation*}
\mathcal{L}_{\rho}(x^{k+1},\lambda^{k+1})-\mathcal{L}_{\rho}(x^{k},\lambda^{k})\leq -\frac{\theta^3  \rho^{3\zeta_1}}{96 H_{\rho}^2}.
\end{equation*}
\end{lemma}

\begin{proof}
See Appendix.
\end{proof}

\noindent Let us now provide the convergence rate of Algorithm \ref{alg:ALSSF_alg} in the region $\mathcal{R}_3\setminus\mathcal{R}_1$ in which the gradient updates (Line 7) are used. 

\begin{lemma}\label{le:quadratic}
Let $\{(x^{k},\lambda^{k})\}$  be generated by  Algorithm \ref{alg:ALSSF_alg} and $\epsilon>0$. 
Suppose that Assumptions \ref{as:Assu_SSF} and \ref{assum_strict_sadlle} hold on a compact set $\mathcal{S}$ on which the primal iterates belong to, there exists  $M > 0$ such that $\Vert \lambda^k\Vert \leq M$ for any $k\geq 0$, $\rho \geq \frac{2M}{\epsilon}$, $x^k \in \mathcal{R}_3\setminus\mathcal{R}_1$ and $ \beta \geq L_{\rho}$.
Then, Algorithm \ref{alg:ALSSF_alg} converges to $x^{*}$ with the following rate (we denote $q_{\rho} = \frac{\gamma \rho^{\zeta_2}}{\beta}$):
\begin{equation*}
	\Vert \nabla_x \mathcal{L}_{\rho}(x^{k+N}, \lambda^k)\Vert \leq L_{\rho}\Vert x^{k+N} - x^{*}\Vert  \leq L_{\rho} \left(1- q_{\rho}\right)^\frac{N}{2} \Vert x^k - x^{*}\Vert. 
\end{equation*}
Moreover, after $N= \mathcal{O}{\left( \frac{1}{q_{\rho}} \log{\left(\frac{L_{\rho}\xi}{\epsilon}\right)}\right)} = \mathcal{O}\left( \frac{1}{\epsilon^{1-\zeta_2}} \log\left(\frac{1}{\epsilon}\right)\right)$  iterations, Algorithm~\ref{alg:ALSSF_alg} achieves $\Vert \nabla_x \mathcal{L}_{\rho}(x^{k+N}, \lambda^k)\Vert  \leq \epsilon$.
\end{lemma}

\begin{proof}
See Appendix.
\end{proof}

\noindent Let us derive the maximum number of iterations required by Algorithm \ref{alg:ALSSF_alg} to enter into $\mathcal{R}_3\setminus \mathcal{R}_1$.

\begin{lemma} \label{le:R3_enter}
Let $\{(x^{k},\lambda^{k})\}_{k\geq0}$ be generated by  Algorithm \ref{alg:ALSSF_alg} and $\epsilon>0$. 
Suppose that Assumptions \ref{as:Assu_SSF} and  \ref{assum_strict_sadlle} hold on a compact set $\mathcal{S}$ on which the primal iterates belong to, there exists  $M>0$ such that $\Vert \lambda^k\Vert \leq M$ for any $k\geq 0$,  $\rho\geq\frac{2M}{\epsilon}$,  $\beta\geq L_{\rho}$ and $H_{\rho} \leq \upsilon \leq 2 H_{\rho}$.
Then, Algorithm \ref{alg:ALSSF_alg} takes at most 
\begin{equation*}
	\left(\mathcal{L}_{\rho}(x^{1},\lambda^{k}) - \mathcal{L}_{\rho}(x^{*},\lambda^{k})\right)\left(\frac{8\beta}{\alpha^2}+ \frac{96 H_{\rho}^2}{\theta^3  \rho^{3\zeta_1}}\right)=\mathcal{O}\left( \frac{1}{\epsilon} + \frac{1}{\epsilon^{2-3\zeta_1}}\right)
\end{equation*}
iterations to enter into the region $\mathcal{R}_3\setminus \mathcal{R}_1$.
\end{lemma}

\begin{proof}
See Appendix.
\end{proof}

\noindent  Finally, we estimate the total complexity of Algorithm \ref{alg:ALSSF_alg} to reach an $\epsilon$-second-order solution $\bar{x}$ for problem \eqref{eq1}.
 
\begin{theorem}\label{th:ALSSS_alg_convergence}
Let $\{(x^{k},\lambda^{k})\}_{k\geq0}$  be generated by  Algorithm~\ref{alg:ALSSF_alg} and $\epsilon>0$.  Suppose that Assumptions \ref{as:Assu_SSF} and  \ref{assum_strict_sadlle} hold on a compact set $\mathcal{S}$ on which the primal iterates belong to, there exists  $M>0$ such that $\Vert \lambda^k\Vert \leq M$ for any $k\geq 0$,  $\rho\geq\frac{2M}{\epsilon}$,   $\beta\geq L_{\rho} $ and $H_{\rho} \leq \upsilon \leq 2 H_{\rho}$. Then,  Algorithm \ref{alg:ALSSF_alg}  yields an $\epsilon$-second-order  solution to problem \eqref{eq1} after at most $K$ iterations, where
\begin{equation*}
K = \mathcal{O}\left(\frac{1}{\epsilon^{1-\zeta_2}} \log \left(\frac{1}{\epsilon}\right) + \frac{1}{\epsilon} + \frac{1}{\epsilon^{2-3\zeta_1}}\right).
\end{equation*}
\end{theorem}

\begin{proof}
For all $k \geq 1$, by Lemma \ref{le:feasibility}, it is guaranteed that $\Vert F(x^k)\Vert  \leq \epsilon$. 
In addition, after \(\mathcal{O}\left(\frac{1}{\epsilon}+ \frac{1}{\epsilon^{2-\zeta_1}}\right)\) iterations, Algorithm~\ref{alg:ALSSF_alg}  enters into the region $\mathcal{R}_3 \setminus \mathcal{R}_1$ due to Lemma \ref{le:R3_enter}.  
Finally, by Lemma \ref{le:quadratic}, it takes $\mathcal{O}\left(\frac{1}{\epsilon^{1-\zeta_2}}\log\frac{1}{\epsilon}\right)$  iterations  to reach a point $x^K$ such that
\begin{equation*}
	\Vert F(x^K)\Vert \leq\epsilon, \quad \Vert \nabla_x\mathcal{L}_{\rho}(x^K,\lambda^K)\Vert \leq \epsilon, \quad \nabla^2_{xx}\mathcal{L}_{\rho}(x^K,\lambda^K)\succeq \gamma \rho^{\zeta_2} I_n\succ 0.
\end{equation*}
Using the expressions of $\nabla_x\mathcal{L}_{\rho}(x^K,\lambda^K)$ and $\nabla^2_{xx}\mathcal{L}_{\rho}(x^K,\lambda^K)$, see \eqref{eq:gradhes},  denoting $\bar{\lambda}=\lambda^K+\rho F(x^K)$ and letting $d \in \mathbb{R}^n$ such that $J_F(x^K) d = 0$, we get:
\begin{equation*}
\arraycolsep=0.2em
\left\{\begin{array}{ll}
       &  \Vert \nabla f(x^K)+ J_F(x^K)^T\bar{\lambda}\Vert \leq \epsilon, \vspace{1ex}\\
       & \Vert F(x^K)\Vert \leq\epsilon, \vspace{1ex} \\
       & d^T \nabla^2_{xx}\mathcal{L}_{\rho}(x^K,\lambda^K)d =  d^T \left(\nabla^2 f(x^K) + \sum_{i=1}^m \bar{\lambda}_i\nabla^2 f_i(x^K)\right)d \geq \gamma \rho^{\zeta_2} > 0.
\end{array}\right.
\end{equation*}
Therefore, we conclude that $x^K$ is an $\epsilon$-second-order solution to \eqref{eq1} in the sense of Definition \ref{de:2nd_approx_sol}. This completes our proof. \qed 
\end{proof}

\medskip 

\noindent As we can observe from Theorem \ref{th:ALSSS_alg_convergence}, Algorithm \ref{alg:ALSSF_alg} is able to escape strict saddle points due to the cubic regularized Newton step but it may take long to do so (around $\mathcal{O}\left(\frac{1}{\epsilon^{2-3\zeta_1}}\right)$ iterations).  Note that, on the one hand, if $\zeta_1$ and $\zeta_2$ are close to zero, then the complexity required to yield an $\epsilon$-second-order solution to problem  \eqref{eq1} is of  order  $\mathcal{O}\left(\frac{1}{\epsilon^{2}}\right)$ Jacobian evaluations. 
On the other hand, if $\zeta_1$ and $\zeta_2$ are close to $1$, then this complexity reduces to  $\mathcal{O}\left(\frac{1}{\epsilon}\right)$.  It is worth noting that Algorithm \ref{alg:ALSSF_alg} achieves improved complexity bounds ranging from  $\mathcal{O}\left(\frac{1}{\epsilon^2}\right)$ to $\mathcal{O}\left(\frac{1}{\epsilon}\right)$ to obtain an $\epsilon$-second-order solution,  compared to $\mathcal{O}\left(\frac{1}{\epsilon^2}\right)$ complexity of  Algorithm~\ref{alg1}  for merely obtaining an $\epsilon$-first-order  solution.  This improvement narrows the gap between  the theoretical and practical performance of augmented Lagrangian-based methods and may explain why these methods often work well in practical applications and identify  (global) minima.


\section{Numerical results}\label{sec6}
In this section we numerically compare Algorithm \ref{alg1} (L-AL) with SCP algorithm \cite{MesBau:21}, IPOPT \cite{WacBie:06} and Algencan \cite{AndBir:08} (which is also an augmented Lagrangian based method), on nonconvex optimization problems with nonlinear equality constraints. The simulations are implemented in Python and executed on a PC with (CPU 2.90GHz, 16GB RAM). Since one cannot guarantee that the SCP iterates converge to a first-order (KKT) point, we choose the following stopping criteria: we stop the algorithms when the difference between two consecutive values of the objective function is less than a tolerance $\epsilon_1=10^{-3}$ and the norm of constraints is less than a tolerance $\epsilon_2=10^{-5}$.  For the implementation of our method, we fix the penalty parameter $\rho$ for each problem to $\rho = 10^7$. The parameter $\beta_k$ is selected dynamically to satisfy the inequality \eqref{decrease_algorithm}. A problem is considered successfully solved by a method if the stopping criteria are met within 30 minutes. If an algorithm fails to meet these conditions, we indicate this with a “–” in the results. The numerical results are illustrated in Table \ref{tab1} and Figure \ref{fig1}. 

\begin{table}
\small
{
    
      \centering
\begin{adjustbox}{width=\columnwidth,center}

    \begin{tabular}{|c|cc|cc|cc|cc|}
    \hline
   \multirow{3}{*}{\backslashbox{(n,m)}{Alg}} 
     & \multicolumn{2}{c|}{L-AL} &
     \multicolumn{2}{c|}{SCP } &
      \multicolumn{2}{c|}{IPOPT}&
       \multicolumn{2}{c|}{Algencan}\\ \cline{2-9}
              & \# iter     & cpu  
          & \# iter    & cpu &
           \# iter     & cpu &
          \# iter     & cpu     \\ 
         & $f^*$    & $\|F\|$  
          & $f^*$    & $\|F\|$ &
           $f^*$    & $\|F\|$  &
          $f^*$    & $\|F\|$\\
    \hline

    OPTCTRL3 &   8 &  0.14
    &5 &  0.16 & 
      7 & 7.40 & 
    6 & \textbf{0.01}\\
    (119,80)& 2048.01 &  8.77e-10
      &2048.01 & 4.52e-10 & 
      2048.01& 1.84e-08 & 
    2048 & 3.15 e-10\\\hline
      
   OPTCTRL3 &  36 &  \textbf{1.23} 
    &7 &  1.55 & 
      10 & 11.99& 
    13 & 3.47\\
    \shortstack{(1199,800)}& 18460.22 & 5.29e-08
      &18460.22 & 1.84e-09 & 
      18460.22& 6.33e-09 & 
   18460 & 7.49 e-09\\\hline

    OPTCTRL3&  56 &  \textbf{19.57}
    &24 &  105.08 &  
      11 & 26.95 & 
    11 & 102.68\\
   (4499,3000)& 74465.03 & 1.76e-08
      &74465.03 & 6.87e-09 &
      74465.03& 1.09e-08 & 
    74470 & 8.66 e-09\\\hline

        DTOC4 &  4 &  \textbf{0.98} 
    &4 &  5.81 &  
      3 & 23.51 & 
    13 & 4.88\\
    (2997,1998)& 2.87 & 2.82-07
      &2.87 & 2.83e-07 &  
      2.87& 9.33e-09 & 
    2.87 & 6.59 e-09\\\hline
      
   DTOC4 &  4 &  \textbf{2.01} 
    &3 &  16.74 &  
      3 & 29.02 & 
    13 & 12.35\\
    (4497,2998)& 2.87 & 3.02e-07
      &2.87 & 4.87e-10 &  
      2.87& 3.66e-08 & 
    2.87 & 3.56e-08\\\hline

    DTOC4 &  4 &  \textbf{46.91} 
    &4 &  566.80 &  
      3 & 146.73 & 
    18 & 149.66\\
    (14997,9998)& 2.87 & 3.40e-07
      &2.87 & 1.05e-07 &  
      2.86& 4.49e-09 & 
    2.86 & 7.27e-09\\\hline

    DTOC5 &  7 &  \textbf{0.33}
    &7 &  0.98 &  
      3 & 12.06 & 
    19 & 0.52\\
       \shortstack{(998,499)} & 1.53 & 3.45e-06
      &1.53 & 3.45e-06 &  
      1.53& 7.76e-07 & 
    1.53 & 9.72 e-08\\\hline
      
   DTOC5 &  10 &  \textbf{1.32}
        &10 &  4.52 &  
      3 & 18.74 & 
    12 & 1.72\\
      \shortstack{ (1998,999)} & 1.53 & 1.56e-06
      &1.53 & 1.56e-06 &  
      1.53& 6.88e-08 & 
    1.53 & 3.11 e-08\\\hline
      
    DTOC5&  23 &  \textbf{42.40 }
    &24 &  799.07 &  
      3 & 75.25 & 
    18 & 48.27\\
    (9998,4999)& 1.54 & 2.19e-07
      &1.54 & 1.96e-07 &  
      1.53& 2.49e-07 & 
    1.53 & 3.31 e-07\\\hline

    ORTHREGA&  37 &  {0.91}
    & 39 &  1.73 &  
      76 & 10.14 & 
    19 & \textbf{0.28}\\
    (517,256)& 1414.05 & 1.23e-06
      &1664.80 & 1.24e-06 &  
      1414.05& 6.19e-10 & 
    1414 & 3.34 e-09\\\hline
      
    ORTHREGA &  53 &  \textbf{13.27}
    &67 &  31.78 &  
      14 & 23.99 & 
    32 & {14.21}\\
    (2053,1024) & 5661.43 & 7.90e-07
      &6654.78 & 2.07e-06 &  
      5661.43& 9.25e-07 & 
    5661 & 2.17 e-08\\
       \hline

        ORTHREGA &  58 & \textbf{65.72}
    & - &   - &  
      20 & 71.78 & 
     40 & 68.61 \\
    (8197,4096) &  22647.84 & 1.83e-07
      & - & - &  
      22674.84 & 1.86e-09 & 
     22674.84& 6.32e-08\\\hline
      

   %
       MSS1 & 70  &  \textbf{1.23}
    & 12& 0.15  &  
      53 & 13.52 &  
    15 & 0.53\\
       \shortstack{(90, 73)} & -15.99 & 8.11e-06
      &-8.71e-08 & 1.76e-06 &  
      -16.00& 4.17e-08 & 
    -15.00 & 3.29 e-08\\\hline
      
     
       MSS2 & 58  &  \textbf{21.99}
    & 21& 8.05  &  
      7 & 14.65 & 
    - & -\\
       \shortstack{(756, 703)} & -123.99 & 3.11e-06
      & -2.53e-10& 6.12e-06 &  
      -26.97& 5.96e-08 &  
    - & -\\\hline
      
   
       MSS3 & 58  &  \textbf{106.79}
    & 22& 135.15  &  
      -&- &  
    - & -\\
       \shortstack{(2070, 1981)} & -338.91 & 9.42e-07
      &-5.29e-09  & 7.76e-06  &  
      -& - & 
    - & -\\\hline
      
     %
       OPTCTRL6 & 56  &  19.03
    & 24&   \textbf{13.46}  &  
      13 & 27.42 & 
    11 & 101.34\\
       \shortstack{(4499, 3000)} & 74465.03 & 1.85e-08
      &74465.03 & 3.27e-09 &  
      74465.03& 2.32e-09 & 
    74470 & 8.47 e-09\\\hline

       OPTCDEG2 & 375  &  \textbf{198.97}
    & 3& 2.71  &  
      4 & 7.92 & 
    19 & 107.27\\
       \shortstack{(4499, 3000)} & 7.80 & 1.00e-07
      &59.08 & 5.85e-08 &  
      227.72& 6.55e-08 &  
    227.7 & 2.58e-07\\\hline

    OPTCDEG3 & 9  &  \textbf{9.01}
    & 43 & 22.64  &  
      11 & 21.33 & 
    25 & 84.92\\
       \shortstack{(4499, 3000)} & 12.13 & 8.28e-06
      &12.13 & 6.12e-06 &  
      12.13& 4.61e-07 &  
    12.13 & 7.04e-07\\\hline

    
       ORTHREGC & 28  &  \textbf{20.09}
    & 29& {20.93}  &  
      16 & 31.14 &  
    26 & 17.82\\
       \shortstack{(5005, 2500)} & 94.81 & 9.92e-06
      &94.81 & 8.42e-06 &  
      94.81& 7.52e-07 & 
    94.81 & 3.07e-07\\\hline

               EIGENB2 & 6  & \textbf{2.21}
    &5 &  5.28 &  
      27 & {56.61} & 
    6 & 303.18\\
       \shortstack{(2550, 1275)} & 0.00 & 7.27e-06
      &110.50 & 1.61e-14 &  
      0.00& 5.45e-09 & 
    0.00 & 8.33e-08\\\hline

       EIGENC2 & 6  &  \textbf{1.95}
    & 6 & 4.68  &  
      13 & {24.93} &  
    6 & 32.02\\
       \shortstack{(2652, 1326)} & 0.01  & 5.98e-06
      & 11162.75& 4.64e-16 &  
      0.00& 8.43e-10 & 
    0.00 & 3.51e-10\\\hline
    
     
       EIGENACO & 5  &  {2.43}
    & 8& 1.75  &  
      - & - & 
    2 & \textbf{1.87}\\
       \shortstack{(2550, 1275)} & 0.01 & 4.22e-06
      &22425.04 & 2.37e-18 &  
      -& - & 
    0.00 & 3.21e-09\\\hline

       EIGENBCO & 7  &  \textbf{3.37}
    & 5 & 1.23  &  
      9 & {19.58} & 
    - & -\\
       \shortstack{(2550, 1275)} & 0.01 & 1.45e-06
      &49.50 & 5.79e-16 &  
      0.00& 3.18e-17 & 
    - & -\\\hline
      
     
       EIGENCCO & 8  &  \textbf{3.52}
    & 7& 1.55  &  
      13 & {42.88} & 
    6 & 1203.96\\
       \shortstack{(2652, 1326)} & 0.00 & 5.69e-06
      &11100.51 & 1.99e-10 &  
      0.00& 2.75e-12 &  
    0.00 & 1.86e-10\\\hline


       DTOC1NA & 29  &  54.24
    & 4& 3.86  &  
      5 & {12.08} & 
    5 & \textbf{0.23}\\
       \shortstack{(5994, 3996)} & 4.14 & 3.09e-06
      &47.66 & 5.03e-13 &  
      4.15& 7.44e-11 & 
    4.14 & 8.03e-10\\\hline

     
       DTOC1NB & 17  &  38.46
    & 4& 3.89  &  
      5 &{11.71} &
    5 &  \textbf{0.36}\\
       \shortstack{(5994, 3996)} & 7.15 & 9.61e-06
      &48.47 & 1.68e-14 &  
      7.13& 6.31e-12 &  
    7.14 & 1.19e-10\\
      

       \hline

       DTOC1NC & 23  &  42.77
    & 6 & 5.76  &  
      3 & {8.12} &
    7 & \textbf{0.48}\\
       \shortstack{(5994, 3996)} & 35.21 & 7.52e-06
      & 58.64 & 3.20e-11 &  
      35.21& 5.80e-10 &
    35.20 & 6.42e-09\\\hline

     
       DTOC1ND & 37  &  65.97
    & 7& 6.70  &  
      3 & {12.79} & 
    4 & \textbf{0.26}\\
       \shortstack{(5994, 3996)} & 47.61 & 8.82e-06
      &66.66 & 3.38e-11 &  
      47.63& 7.85e-09 & 
    47.60 & 1.76e-10\\
      

       \hline

       SPINOP & 101  &  \textbf{76.31}
    & -& -  &  
      - & - & 
    - & -\\
       \shortstack{(1327, 1325)} & 150.50 & 9.34e-06
      &- & - &  
      -& - &  
    - & -\\\hline

       DTOC2 & 8  &  \textbf{11.51}
    & 21& 19.73  &  
      11 & 22.64 & 
    21 & 105.04\\
       \shortstack{(5994, 3996)} & 0.51 & 3.38e-06
      &0.91 & 9.66e-06 &  
      0.50& 6.21e-09 & 
    0.51 & 6.28e-09\\\hline

       ROBOTARM & 131  &  \textbf{106.41}
    & -& -  &  
      7 & {109.20} &
     23 & 377.26\\
       \shortstack{(4400, 3202)} & 7.84 & 9.62e-06
      &- & - &  
      9.14& 2.05e-08 & 
    9.14 & 1.20e-08\\\hline

       ROCKET & 163  &  8.31
    & -& -  &  
      5 & \textbf{7.49} &  
    - & -\\
       \shortstack{(2403, 2002)} & -1.00 & 5.44e-07
      &- & - &  
      -1.00& 2.73e-07 & 
    - & -\\\hline

       CATMIX & 29  &  \textbf{1.48}
    & 9 & 2.84  &  
      3 & 4.36 &  
    22 & 26.94\\
       \shortstack{(2401, 1600)} & -0.03 & 1.12e-06
      &-0.03 & 5.71e-09 &  
      -0.04& 4.88e-09 &  
    -0.04 & 2.55e-09 \\\hline
      
    \end{tabular}%
    \end{adjustbox}
}
\caption{Comparison between L-AL, SCP, IPOPT and Algencan algorithms  on test problems with equality constraints from the CUTEst collection. }
\label{tab1}
\end{table}

\medskip 
\noindent In Table \ref{tab1}, we report the number of iterations, CPU time (in seconds), objective value, and feasibility violation (measured as the Euclidean norm of the functional constraints) for L-AL, SCP, IPOPT, and Algencan on a set of real-world problems with nonlinear  equality constraints  selected from the CUTEst collection \cite{GouOrb:15}. Notably, for the majority of  test cases,  L-AL algorithm is able to yield optimal solutions faster than the other methods (the best CPU time is  highlighted in bold in the table). However, for a few problems, our method appears to be slower, although it still produces solutions of comparable quality to those obtained by IPOPT and/or Algencan. Moreover, Table \ref{tab1} shows that our method successfully solves all selected problems within the 30 minutes time limit. In contrast, the other methods fail on multiple instances, demonstrating the robustness of L-AL compared to SCP, IPOPT and Algencan.
    
\begin{figure}[htp]
    \centering
  \includegraphics[width=12cm,height=5cm]{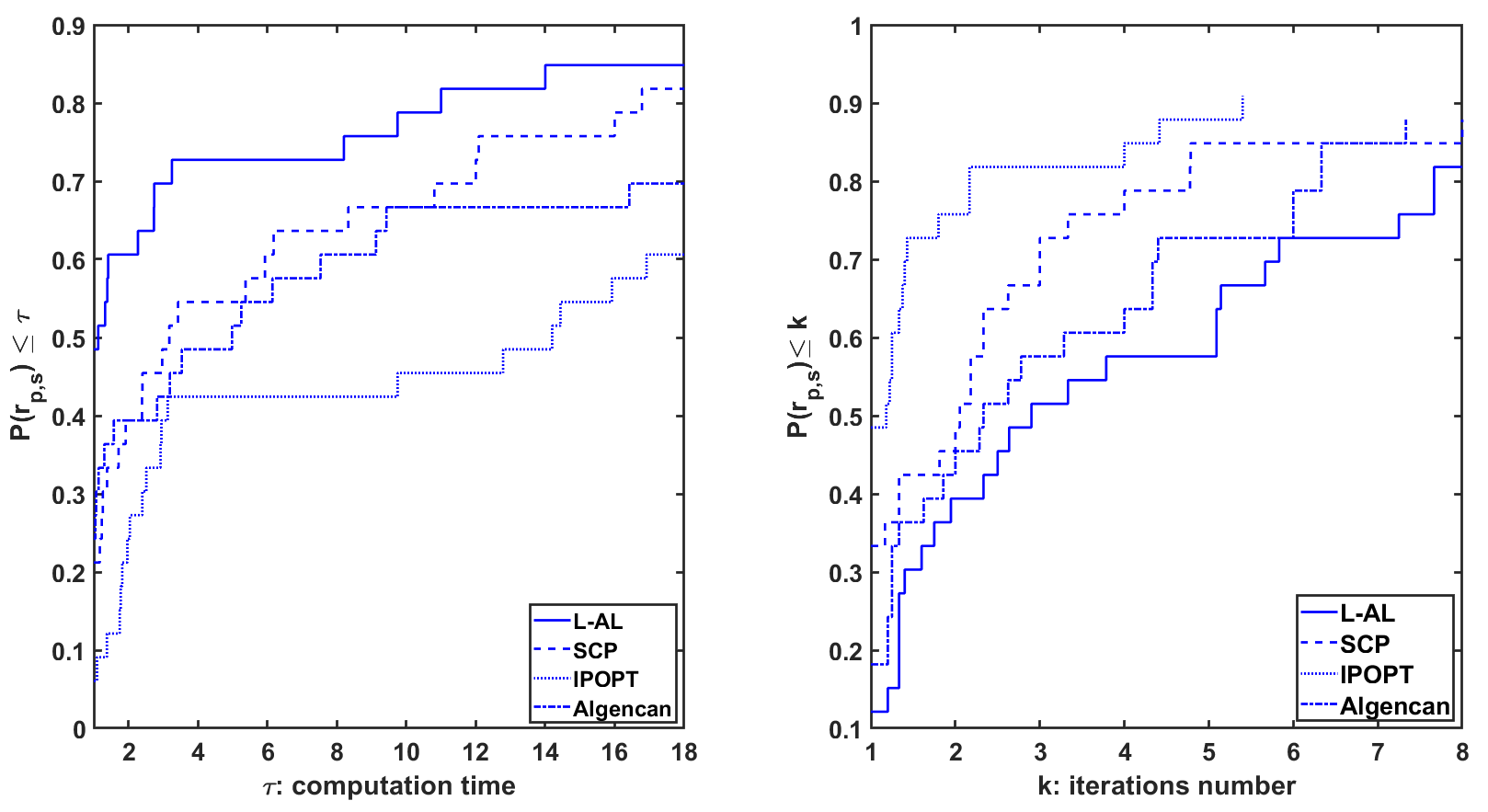}
    \caption{Performance profiles for  computation time (left) and  number of iterations (right). }
    \label{fig1}
\end{figure}

\medskip
\noindent  Figure \ref{fig1} presents performance profiles for computation time (left) and number of iterations (right) across the four algorithms. In these profiles, the vertical axis $P(r_{p,s} \leq \tau)$ (or $P(r_{p,s} \leq k)$) represents the proportion of test problems for which the performance ratio $r_{p,s}$ does not exceed a threshold $\tau$ (or $k$), respectively. Here, $r_{p,s}$ denotes the ratio of the computational time (or iteration count) required by solver $s$ to solve problem $p$ compared to the best time (or iteration count) achieved by any of the four solvers for the same problem.
From the time profile in Figure \ref{fig1} (left), it is clear that the proposed algorithm, L-AL, reaches a performance ratio of $1$ more quickly than SCP, IPOPT and Algencan, indicating superior efficiency. However, this trend is not consistently observed in terms of the number of iterations. Nonetheless, based on these preliminary experiments, we can conclude that L-AL is an efficient and robust algorithm for solving optimization problems with nonlinear equality constraints, often outperforming established solvers like IPOPT and Algencan in terms of computational speed.


\section{Conclusions}\label{sec7}
In this paper, we have proposed a linearized augmented Lagrangian  method for solving (locally) smooth  optimization problems with nonlinear equality constraints. In this method  we have linearized the objective and the functional  constraints  within the augmented Lagrangian function and added a regularization term. By dynamically generating the regularization (proximal) parameter,  we have proved global  asymptotic convergence and  convergence rate to an $\epsilon$-first-order solution. Moreover, we have numerically shown that the proposed algorithm is efficient, comparing it with several known algorithms and solvers,  such as SCP,  IPOPT and Algencan. 



\section*{Conflict of interest}
The authors declare that they have no conflict of interest.

\section*{Data availability}
It is not applicable.

\section*{Appendix}

\noindent \textbf{Proof of Lemma \ref{lambda_bound}}
Let $k \geq 1$.  Using the optimality condition for $x_{k+1}$, we have:
\begin{align*}
   \nabla f(x_{k})&+{{J_F}(x_{k})}^T\lambda_{k}+\rho{{J_F}(x_{k})}^T\Big(F(x_{k})+{J_F}(x_{k})(x_{k+1}-x_{k})\Big)\\
  &+\beta_{k+1}(x_{k+1}-x_{k})=0. 
\end{align*}
Combining  this with the update in Step 6 of Algorithm \ref{alg1}, we get:
\begin{equation}\label{eq2}
    \nabla f(x_{k})+{{J_F}(x_{k})}^T\lambda_{k+1}+\beta_{k+1}(x_{k+1}-x_{k})=0.
\end{equation}
By replacing $k$ with $k-1$, we obtain:
\begin{equation}\label{eq3}
    \nabla f(x_{k-1})+{{J_F}(x_{k-1})}^T\lambda_{k}+\beta_{k}(x_{k}-x_{k-1})=0.
\end{equation}
Subtracting \eqref{eq3} from \eqref{eq2}, we have: 
\begin{align*}
     \nabla f(x_{k})-\nabla f(x_{k-1})&+{{J_F}(x_{k})}^T\Delta\lambda_{k+1}+{\big({J_F}(x_{k})-{J_F}(x_{k-1})\big)}^T\lambda_{k}\\
    &+\beta_{k+1}\Delta x_{k+1}-\beta_{k}\Delta x_{k}=0\hspace{0.5cm} \forall k\geq1.
\end{align*}
Further, since $x_{k-1}, x_k \in\mathcal{S}$, then using  Assumption \ref{assump2}, we have: 
\begin{align}
    \|\Delta\lambda_{k+1}\|\leq&\frac{1}{\sigma}\Big(\|\nabla f(x_{k})-\nabla f(x_{k-1})\|+\|{J_F}(x_{k})-{J_F}(x_{k-1})\|\|\lambda_{k}\|\nonumber\\
     &+\beta_{k+1}\|\Delta x_{k+1}\|+\beta_{k}\|\Delta x_{k}\|\Big) \hspace{0.4cm} \forall k\geq1. \label{delta_to_replace}
\end{align} 
From \eqref{eq3}, we also have:
\begin{equation}\label{lam}
   \|\lambda_{k}\|\leq\frac{1}{\sigma}{\Big(\|\nabla f(x_{k-1})\|+\beta_{k}\|\Delta x_{k}\|\Big)} \leq\frac{1}{\sigma}{\Big(M_f+\beta_{k}\|\Delta x_{k}\|\Big)}.  
\end{equation}
Moreover, from Assumption \ref{assump2}, we  have: 
\[
   \|{J_F}(x_{k})-{J_F}(x_{k-1})\|\leq L_F\|\Delta x_{k}\| \hspace{0.2cm} \text{ and } \hspace{0.2cm} \|{J_F}(x_{k})-{J_F}(x_{k-1})\|\leq 2M_F.  
\]
By replacing, the above inequalities and \eqref{lam} in \eqref{delta_to_replace}, we obtain:
\begin{align}
\label{delta_lambda}
  & \|\Delta\lambda_{k+1}\|\nonumber\\
   &\leq\frac{1}{\sigma}\left({L_f\|\Delta x_{k}\|+\frac{M_fL_F+2M_F\beta_{k}}{\sigma}\|\Delta x_{k}\|+\beta_{k+1}\|\Delta x_{k+1}\|+\beta_{k}\|\Delta x_{k}\|}\right) \nonumber\\
      &=\frac{\beta_{k+1}}{\sigma}\|\Delta x_{k+1}\|+\frac{L_f\sigma + M_fL_F+(2M_F+\sigma)\beta_{k}}{\sigma^2}\|\Delta x_{k}\| \nonumber \\
   &\overset{M_F\geq \sigma}{\leq}\! \frac{L_f M_F\! +\! M_f L_F \!+\! 3 M_F \beta_{k+1}}{\sigma^2}\|\Delta x_{k+1}\|\!+\!\frac{L_f M_F \!+\! M_f L_F \!+\! 3 M_F \beta_{k}}{\sigma^2}\|\Delta x_{k}\|\nonumber \\
   & \leq \frac{(1+ 3\mu)(L_f M_F + M_f L_F) + (1+ 3\mu) M_F (\beta_{k+1}- \mu L_f)}{\sigma^2}\|\Delta x_{k+1}\| \nonumber\\
   &\quad +\frac{(1+ 3\mu)(L_f M_F + M_f L_F) + (1+ 3\mu) M_F (\beta_{k} - \mu L_f)}{\sigma^2}\|\Delta x_{k}\|,
\end{align}
where in the second inequality we use  that $M_F/\sigma \geq 1$.
Further, since $(a+b)^2\leq2a^2+2b^2$, we finally get \eqref{lambda_squared}. 
\qed
\medskip

\noindent {\textbf{Proof of Lemma \ref{lemma3}}.
 Note that the subproblem's objective function $x \mapsto \bar{\mathcal{L}}_{\rho}(x\cdot,\lambda_k;x_k)+\frac{\beta_{k+1}}{2}\|x-x_k\|^2$ is strongly convex with strong convexity constant $\beta_{k+1}$. Combining this with  the optimality of $x_{k+1}$  and the fact that $\bar{\mathcal{L}}_{\rho}(x_{k},\lambda_k;x_k)=\mathcal{L}_{\rho}(x_{k},\lambda_k)$, we get: 
\begin{equation} \label{initial_decrease}
 \bar{\mathcal{L}}_{\rho}(x_{k+1},\lambda_k;x_k)\leq \mathcal{L}_{\rho}(x_{k},\lambda_k) -\beta_{k+1}\|x_{k+1}-x_k\|^2.
\end{equation}
\noindent Further, since $x_{k}, x_{k+1} \in \mathcal{S}$ and $f$ has Lipschitz continuous gradient on $\mathcal{S}$, we have:
\begin{equation}\label{smooth_object}
f(x_{k+1}) - l_f(x_{k+1};x_k) \leq \frac{L_f}{2}\|x_{k+1}-x_k\|^2. 
\end{equation}
\noindent Moreover, using properties of the norm and the fact that $x_{k}, x_{k+1} \in \mathcal{S}$ and $F$ has Lipschitz continuous Jacobian on $\mathcal{S}$, we obtain: 
\begin{align}
   &\langle \lambda_k, F(x_{k+1}) \rangle + \frac{\rho}{2}\|F(x_{k+1})\|^2 - \frac{\rho}{2}\|l_F(x_{k+1};x_k)\|^2 - \langle \lambda_k, l_F(x_{k+1},x_k) \rangle\nonumber\\
   & =  \! \langle \lambda_k + \rho l_F(x_{k+1};x_k), F(x_{k+1})-l_F(x_{k+1};x_k) \rangle + \frac{\rho}{2}\|F(x_{k+1})-l_F(x_{k+1};x_k)\|^2 \nonumber\\
   & \leq  \!  \|\lambda_k \!+ \rho l_F(x_{k+1};x_k)\| \|F(x_{k+1})-l_F(x_{k+1};x_k) \| \! + \frac{\rho}{2}\|F(x_{k+1})-l_F(x_{k+1};x_k)\|^2 \nonumber\\
    &\overset{\text{Ass. } \ref{assump2}}{\leq} \|\lambda_k+\rho l_F(x_{k+1};x_k)\|  \frac{L_F}{2}\|\Delta x_{k+1}\|^2 + \frac{\rho}{2} \left(\frac{L_F}{2}\|\Delta x_{k+1}\|^2\right)^2.  \label{decrease_feasib}
\end{align}

\noindent  Using the fact that for any $a, b \geq 0$, we have $ab \leq \frac{a^2 + b^2}{2}$, we can bound $  \|\lambda_k + \rho l_F(x_{k+1};x_k)\|$ as follows:
\begin{align}
&\|\lambda_k+\rho l_F(x_{k+1};x_k)\|  - \frac{ \beta_{k+1} - L_f}{2L_F} \leq  \frac{L_F}{ 2(\beta_{k+1} - L_f)}    \| \lambda_k + \rho l_F(x_{k+1};x_k)\|^2  \nonumber\\
&= \! \frac{L_F\rho}{ \beta_{k+1} - L_f}\!  \left( \! \bar{\mathcal{L}_{\rho}} (x_{k+1},\lambda_k;x_k) - f(x_{k+1}) + f(x_{k+1}) - l_f(x_{k+1};x_k) +\frac{1}{2\rho}\|\lambda_k\|^2 \! \right)  \nonumber\\
& \overset{\eqref{initial_decrease}, \eqref{smooth_object}}{\leq} \!\!  \frac{L_F\rho}{\beta_{k+1} - L_f} \left( \mathcal{L}_{\rho}(x_k,\lambda_k) - f(x_{k+1}) - \frac{2\beta_{k+1} - L_f}{2}\|\Delta x_{k+1}\|^2 +\frac{1}{2\rho}\|\lambda_k\|^2 \right)  \nonumber\\
& \leq \frac{L_F\rho}{\beta_{k+1} - L_f} \left( \mathcal{L}_{\rho}(x_k,\lambda_k) - f(x_{k+1}) - \left(\beta_{k+1} - L_f\right)\|\Delta x_{k+1}\|^2 +\frac{1}{2\rho}\|\lambda_k\|^2  \right)  \nonumber\\
&\overset{\text{Ass. \ref{assump3}}}{\leq }\frac{L_F\rho}{\beta_{k+1} - L_f} \left( \mathcal{L}_{\rho}(x_k,\lambda_k) + \frac{1}{2 \rho} \|\lambda_k\|^2  - \ubar{f} \right) - L_F\rho\|\Delta x_{k+1}\|^2   \nonumber\\
&\overset{\eqref{eq_assu}}{\leq} \frac{\beta_{k+1}-L_f}{2L_F} - L_F\rho\|\Delta x_{k+1}\|^2. \label{bound_grad}
\end{align}

\noindent Using \eqref{bound_grad} in \eqref{decrease_feasib}, we further get:
\begin{align}
  &\langle \lambda_k, F(x_{k+1}) \rangle + \frac{\rho}{2}\|F(x_{k+1})\|^2 - \frac{\rho}{2}\|l_F(x_{k+1};x_k)\|^2 - \langle \lambda_k, l_F(x_{k+1},x_k) \rangle \nonumber\\
  & \leq   \frac{\beta_{k+1}-L_f}{2} \|\Delta x_{k+1}\|^2 - \left(\frac{\rho L_F^2}{2} - \frac{\rho L_F^2}{8} \right) \|\Delta x_{k+1}\|^4   \nonumber\\
  & =  \frac{\beta_{k+1}-L_f}{2} \|\Delta x_{k+1}\|^2  - \frac{3\rho L_F^2}{8} \|\Delta x_{k+1}\|^4. \label{to_use_next}
\end{align}
Moreover, we have:
\begin{align*} 
   &\mathcal{L}_{\rho}(x_{k+1},\lambda_k) - \bar{\mathcal{L}}_{\rho}(x_{k+1},\lambda_k;x_k) =f(x_{k+1}) - l_f(x_{k+1};x_k) +   \langle \lambda_k, F(x_{k+1}) \rangle  \\
   & \quad + \frac{\rho}{2}\|F(x_{k+1})\|^2 - \frac{\rho}{2}\|l_F(x_{k+1};x_k)\|^2 - \langle \lambda_k, l_F(x_{k+1},x_k) \rangle.
\end{align*}
Using  \eqref{smooth_object} and \eqref{to_use_next} in the previous relation, it follows that:

\begin{align*}
 \mathcal{L}_{\rho}(x_{k+1},\lambda_k) \leq& \bar{\mathcal{L}}_{\rho}(x_{k+1}, \lambda_k;x_k) + \frac{\beta_{k+1}}{2} \|\Delta x_{k+1}\|^2 - \frac{3\rho L_F^2}{8} \|\Delta x_{k+1}\|^4. 
\end{align*}
Therefore, using \eqref{initial_decrease}, we get:
\begin{align} \label{decrease_proof}
 \mathcal{L}_{\rho}(x_{k+1},\lambda_k) \leq & \mathcal{L}_{\rho}(x_{k},\lambda_k) - \frac{\beta_{k+1}}{2} \|\Delta x_{k+1}\|^2 - \frac{3\rho L_F^2}{8} \|\Delta x_{k+1}\|^4.
\end{align}
Finally, using the definition of $P_k$ in \eqref{lyapunov_function}, we have
\begin{align*}
P&_{k+1}-P_{k}\nonumber\\
=&\mathcal{L}_{\rho}(x_{k+1},\lambda_{k+1})-\mathcal{L}_{\rho}(x_{k+1},\lambda_{k})+\mathcal{L}_{\rho}(x_{k+1},\lambda_{k})-\mathcal{L}_{\rho}(x_{k},\lambda_{k})\nonumber\\
&+\frac{\beta_{k+1}}{4}\|x_{k+1}-x_{k}\|^2-\frac{\beta_{k}}{4}\|x_{k}-x_{k-1}\|^2\nonumber\\
\leq&\big\langle F(x_{k+1})-F(x_{k})-{J_F}(x_{k})(x_{k+1}-x_{k} )  ,\Delta\lambda_{k+1}\big\rangle+\frac{1}{\rho}\|\Delta\lambda_{k+1}\|^2\nonumber\\
&-\frac{2\beta_{k+1}-\beta_{k+1}}{4}\|\Delta x_{k+1}\|^2-\frac{\beta_{k}}{4}\|\Delta x_{k}\|^2 - \frac{3\rho L_F^2}{8}\|\Delta x_{k+1}\|^4 \nonumber\\
\overset{\eqref{inequality_prop}}{\leq}&\frac{\rho}{2}\|F(x_{k+1})-F(x_{k})-{J_F}(x_{k})\Delta x_{k+1} \|^2+\frac{1}{2\rho}\|\Delta \lambda_{k+1}\|^2+\frac{1}{\rho}\|\Delta \lambda_{k+1}\|^2\nonumber\\
&-\frac{\beta_{k+1}}{4}\|\Delta x_{k+1}\|^2-\frac{\beta_{k}}{4}\|\Delta x_{k}\|^2- \frac{3\rho L_F^2}{8}\|\Delta x_{k+1}\|^4 \nonumber\\
\overset{\text{Ass. \ref{assump2}}}{\leq}&\frac{\rho}{2}\left(\frac{L_F}{2} \|\Delta x_{k+1}\|^2 \right)^2+\frac{3}{2\rho}\|\Delta \lambda_{k+1}\|^2\nonumber\\
& -\frac{\beta_{k+1}}{4}\|\Delta x_{k+1}\|^2-\frac{\beta_{k}}{2}\|\Delta x_{k}\|^2- \frac{3\rho L_F^2}{8}\|\Delta x_{k+1}\|^4 \nonumber\\
=&\frac{3}{2\rho}\left\|\Delta\lambda_{k+1}\right\|^2-\frac{\beta_{k+1}}{4}\|\Delta x_{k+1}\|^2-\frac{\beta_{k}}{4}\|\Delta x_{k}\|^2 - \frac{\rho L_F^2}{4}\|\Delta x_{k+1}\|^4\\
\leq& \frac{3}{2\rho}\left\|\Delta\lambda_{k+1}\right\|^2-\frac{\beta_{k+1}}{4}\|\Delta x_{k+1}\|^2-\frac{\beta_{k}}{4}\|\Delta x_{k}\|^2,
\end{align*}
where the first inequality is obtained using \eqref{decrease_proof}  and the update of the dual multipliers in step 6 of  Algorithm \ref{alg1}. This proves our statement. \qed
}

\medskip

\noindent \textbf{Proof of Lemma \ref{conditional_decrease}}
Let $k\geq 1$. Using \eqref{lambda_squared} in \eqref{decrease_algorithm}, we obtain that: 
\begin{align*}
   P_{k+1}-P_{k} \leq&\left(\frac{3}{2\rho}c(\beta_{k+1})-\frac{\beta_{k+1}}{4}\right)\|\Delta x_{k+1}\|^2 + \left(\frac{3}{2\rho}c(\beta_{k})-\frac{\beta_{k}}{4}\right)\|\Delta x_{k}\|^2.
\end{align*}
 Therefore, in order to obtain \eqref{decrease_Lyapunov_lem}, the regularization parameter $\beta_{k}$ and $\beta_{k+1}$ should satisfy the following requirements:  
\begin{equation}\label{imp_ine}
   \beta_{k}\geq\frac{12}{\rho}c(\beta_k),  \quad \beta_{k+1}\geq\frac{12}{\rho}c(\beta_{k+1}). 
\end{equation}
  Let us check when \eqref{imp_ine} holds. To do so, we replace the expressions of $c(\beta_{k})$ and $c(\beta_{k+1})$  in \eqref{imp_ine} and reformulate the inequalities in \eqref{imp_ine}  as follows:
\begin{align}
\label{h_negk}
&  \frac{48(1+3\mu)^2\left(L_f M_F + M_f L_F\right)^2 }{\rho\sigma^4} + \frac{48(1+3\mu)^2 M_F^2}{\rho\sigma^4}(\beta_k -\mu L_f)^2 \nonumber \\
& \leq \mu L_f + (\beta_{k} - \mu L_f). 
\end{align}
\begin{align}
\label{h_neg}
&  \frac{48(1+3\mu)^2\left(L_f M_F + M_f L_F\right)^2 }{\rho\sigma^4} + \frac{48(1+3\mu)^2 M_F^2}{\rho\sigma^4}(\beta_{k+1} -\mu L_f)^2  \nonumber  \\
 & \leq \mu L_f + (\beta_{k+1} - \mu L_f). 
\end{align}
In fact, \eqref{h_negk} holds if:
\begin{equation} \label{rho_for_decrease}
    \rho \geq \!\max\!\left\{\frac{48(1\!+\!3\mu)^2\left(L_f M_F \!+\! M_f L_F\right)^2 }{\mu L_f\sigma^4}, \frac{48(1\!+\!3\mu)^2 M_F^2}{\sigma^4}(\beta_k \! -\! \mu L_f)\!\right\}.
\end{equation}
Similarly,  \eqref{h_neg} holds if:
\begin{equation} \label{rho_for_decrease1}
    \rho \geq \!\max \!\left\{ \!\frac{48(1\!+\!3\mu)^2\left(L_f M_F \!+\! M_f L_F\right)^2 }{\mu L_f\sigma^4},  \frac{48(1\!+\!3\mu)^2 M_F^2}{\sigma^4}(\beta_{k+1} \!-\! \mu L_f) \right\}.
\end{equation}
Therefore, if at some iterate \( k \geq 1 \), \( \rho \) satisfies \eqref{theright_choice_of_rho}, then the decrease \eqref{decrease_Lyapunov_lem} follows. This concludes our proof.
\qed

\medskip

\noindent \textbf{Proof of Lemma \ref{bounded_gradient}}
Using the optimality condition \eqref{eq2}, we have:
\[
    \nabla f(x_{k+1})=-{{J_F}(x_{k})}^T\lambda_{k+1}-\beta_{k+1}(x_{k+1}-x_{k}).
\]
It then follows, by exploiting the definition of $\mathcal{L}_{\rho}$ and the properties of the derivative, that:
\begin{align*}
     &\nabla_x\mathcal{L}_{\rho}(x_{k+1},\lambda_{k+1})=\nabla f(x_{k+1})+{{J_F}(x_{k+1})}^T\big(\lambda_{k+1}+\rho F(x_{k+1})\big)\\
     =  & \big({{J_F}(x_{k+1})}-{{J_F}(x_{k})}\big)^T\lambda_{k+1}+{{J_F}(x_{k+1})}^T\Delta\lambda_{k+1}-\beta_{k+1}\Delta x_{k+1}\\
     & +\rho{{J_F}(x_{k+1})}^T\big( F(x_{k+1})-F(x_{k})-{J_F}(x_{k})\Delta x_{k+1}\big).
\end{align*}
Using basic properties of the Euclidean norm, we further get:
\begin{align}
\label{needed1_lemma8}
  &  \|\nabla_x\mathcal{L}_{\rho}(x_{k+1},\lambda_{k+1})\|\nonumber\\
  &  \leq \|{{J_F}(x_{k+1})}-{{J_F}(x_{k})}\|\|\lambda_{k+1}\|+\|{{J_F}(x_{k+1})}\|\|\Delta\lambda_{k+1}\|\nonumber\\
    & \quad +\beta_{k+1}\|\Delta x_{k+1}\|+\rho\|{{J_F}(x_{k+1})}\|\| F(x_{k+1})-F(x_{k})-{J_F}(x_{k})\Delta x_{k+1}\|\nonumber\\
 &    {\overset{{\text{Ass.} \ref{assump2}, \eqref{lam}}}{\leq}} \;  \frac{M_fL_F}{\sigma}\|\Delta x_{k+1}\|+ \frac{L_F\beta_{k+1}}{\sigma}\|\Delta x_{k+1}\|^2+M_F\|\Delta\lambda_{k+1}\|\nonumber\\
    & \qquad+\beta_{k+1}\|\Delta x_{k+1}\|+\frac{\rho M_F L_F}{2} \|\Delta x_{k+1}\|^2\nonumber\\
    & = \frac{M_fL_F+ \sigma\beta_{k+1}}{\sigma}\|\Delta x_{k+1}\| \!+\! \frac{L_F(2\beta_{k+1} + \rho M_F\sigma)}{2\sigma}\|\Delta x_{k+1}\|^2 \!+\! M_F\|\Delta\lambda_{k+1}\|  \nonumber \\
  &  {\overset{\hspace{-0.5cm}\eqref{delta_lambda}}{\hspace{-0.5cm}\leq}} \!\! \left(\! \frac{M_f\!L_F\!+\!\sigma \beta_{k\!+\!1}}{\sigma}\!+\!\frac{M_F}{\sigma}\frac{\!(1\!+\! 3\mu)(L_f \!M_F \! +\! M_f\! L_F\!) \!+\! (1 \!+\! 3\mu) M_F (\beta_{k\!+\!1}\!-\! \mu\! L_f\!)}{\sigma} \!\right) \! \|\!\Delta x_{k\!+\!1}\!\|  \nonumber\\
    & \quad +\frac{M_F}{\sigma}\;\frac{(1+ 3\mu)(L_f M_F + M_f L_F) + (1+ 3\mu) M_F (\beta_{k}- \mu L_f)}{\sigma}\|\Delta x_{k}\| \nonumber\\
    & \quad  + \frac{L_F(2\beta_{k+1} + \rho M_F\sigma)}{2\sigma}\|\Delta x_{k+1}\|^2 \nonumber\\
    \leq & \frac{(2+ 3\mu)(L_f M_F + M_f L_F) M_F + (2+ 3\mu) M_F^2 (\beta_{k+1}- \mu L_f)}{\sigma^2}\|\Delta x_{k+1}\|\nonumber\\
    & \quad +\frac{(1+ 3\mu)(L_f M_F + M_f L_F) M_F + (1+ 3\mu) M_F^2 (\beta_{k}- \mu L_f)}{\sigma^2}\|\Delta x_{k}\| \nonumber\\
      & \quad + \frac{L_F(2\beta_{k+1} + \rho M_F\sigma)}{2\sigma}\|\Delta x_{k+1}\|^2 \nonumber\\
   & \leq  \frac{(2+ 3\mu)(L_f M_F + M_f L_F) M_F + (2+ 3\mu) M_F^2 (\beta_{k+1}- \mu L_f)}{\sigma^2}\|\Delta x_{k+1}\|\nonumber\\
    & \quad +\frac{(2+ 3\mu)(L_f M_F + M_f L_F) M_F + (2+ 3\mu) M_F^2 (\beta_{k}- \mu L_f)}{\sigma^2}\|\Delta x_{k}\| \nonumber\\
      & \quad  + \frac{L_F(2\beta_{k+1} + \rho M_F\sigma)}{2\sigma}\|\Delta x_{k+1}\|^2.
\end{align}
Similarly, we have:
\begin{align}
   & \|\nabla_{\lambda}\mathcal{L}_{\rho}(x_{k+1},\lambda_{k+1})\|=\|F(x_{k+1})\|\nonumber\\
   & \leq \|F(x_{k+1})-F(x_{k})-{J_F}(x_{k})\Delta x_{k+1}\|+\frac{1}{\rho}\|\Delta\lambda_{k+1}\|\nonumber\\
  &  {\overset{{\text{Ass. } \ref{assump2}}}{\leq}} \frac{L_F}{2}\|\Delta x_{k+1}\|^2+\frac{1}{\rho}\|\Delta\lambda_{k+1}\|\label{needed2_lemma8} \\
   & {\overset{{\eqref{delta_lambda}}}{\leq}} \frac{L_F}{2}\|\Delta x_{k+1}\|^2+\frac{(1\!+ \!3\mu)(L_f M_F\! +\! M_f L_F)  \!+\! (1\!+\! 3\mu) M_F (\beta_{k}\!-\! \mu L_f)}{\rho\sigma^2}\|\Delta x_{k}\|\nonumber\\
    & \quad +\frac{(1+ 3\mu)(L_f M_F +\! M_f L_F)  + (1+ 3\mu) M_F (\beta_{k+1}- \mu L_f)}{\rho\sigma^2}\|\Delta x_{k+1}\|, \nonumber
\end{align}
where the first inequality is obtained from the multipliers update in Step 6 of  Algorithm \ref{alg1}.
 Hence, it follows that:
\begin{align*}
    \|\nabla\mathcal{L}_{\rho}(x_{k+1},\lambda_{k+1})\|&\leq\|\nabla_x\mathcal{L}_{\rho}(x_{k+1},\lambda_{k+1})\|+\|\nabla_{\lambda}\mathcal{L}_{\rho}(x_{k+1},\lambda_{k+1})\|  \\
    &\leq\Gamma_{k+1}\|\Delta x_{k+1}\|+\Gamma_{k}\|\Delta x_{k}\| + c_{k+1}\|\Delta x_{k+1}\|^2\\
     &\leq\Gamma_{k+1}\|\Delta x_{k+1}\|+\Gamma_{k}\|\Delta x_{k}\| + c_{k+1}\|\Delta x_{k+1}\|^2 + c_k\|\Delta x_k\|^2,
\end{align*}
where      \[\Gamma_{k}=\left(M_F+\frac{1}{\rho}\right)\frac{(2+ 3\mu)(L_f M_F + M_f L_F)  + (2+ 3\mu) M_F (\beta_{k}- \mu L_f)}{\sigma^2},  \]
and \[
c_k= \frac{L_F}{2}\left(1+\frac{2\beta_{k} +\rho M_F \sigma}{\sigma}\right)\quad \forall k\geq 1.
\]
This proves our claim.\qed
\medskip


\noindent \textbf{Proof of Lemma \ref{usefull_lemma}}
Let \( k \geq 1 \). From inequality \eqref{lam}, we have:
\begin{align}
    &\frac{3}{\rho} \|\lambda_k\|^2 - \frac{\beta_k}{4} \|x_k - x_{k-1}\|^2 
    \leq \frac{6 M_f^2}{\rho \sigma^2} + \frac{6\beta_k^2}{\rho \sigma^2} \|x_k - x_{k-1}\|^2 - \frac{\beta_k}{4} \|x_k - x_{k-1}\|^2 \nonumber \\
    &= \frac{6M_f^2}{\rho \sigma^2} - \frac{\beta_k}{4} \left(1 - \frac{24\beta_k}{\rho \sigma^2} \right) \|x_k - x_{k-1}\|^2 \nonumber \\
    &\overset{\beta_k \leq \bar{\beta}}{\leq} \frac{6M_f^2}{\rho \sigma^2} - \frac{\beta_k}{4} \left(1 - \frac{24\bar{\beta}}{\rho \sigma^2} \right) \|x_k - x_{k-1}\|^2 \overset{\eqref{rho_def}}{\leq} 1, \label{eq_to_be_used}
\end{align}
which proves the first statement. Furthermore, using the definition of $\mathcal{L}_{\rho}$, we get:
\begin{align*}
    \mathcal{L}_{\rho}(x_k, \lambda_k) + \frac{1}{2\rho} \|\lambda_k\|^2 
    &= P_k - \frac{\beta_k}{4} \|\Delta x_k\|^2 + \frac{1}{2\rho} \|\lambda_k\|^2 \\
    &\leq P_k + \frac{3}{\rho} \|\lambda_k\|^2 - \frac{\beta_k}{4} \|\Delta x_k\|^2 
    \overset{\eqref{eq_to_be_used}}{\leq} P_k + 1,
\end{align*}
which proves the second statement. This completes the proof. \qed

\medskip

\noindent \textbf{Proof of Lemma \ref{bbound}}
We prove these results using  induction arguments. Let $k = 0$. Using the definition of $\mathcal{L}_{\rho}$, we have: 
 \begin{align}
   P_0 & = \mathcal{L}_{\rho}(x_0,\lambda_0) =f(x_0)+\langle \lambda_0  ,F(x_0)\rangle+\frac{\rho}{2}\|F(x_0)\|^2 \nonumber\\
    &{\overset{{\eqref{inequality_prop}}}{\leq}} f(x_0)+\frac{\|\lambda_0\|^2}{2\rho}+\frac{\rho}{2}\|F(x_0)\|^2+\frac{\rho}{2}\|F(x_0)\|^2 {\overset{{\eqref{eq4}}}{\leq}}\bar{f}+\frac{1}{2\rho}\|\lambda_0\|^2+c_0 \nonumber\\
    &\overset{\rho\geq 1}{\leq} \bar{f} + \frac{1}{2}\|\lambda_0\|^2 + c_0\leq \bar{f} + \|\lambda_0\|^2 + c_0. \label{ine10}
\end{align}   
According to Lemma~\ref{conditional_decrease}, in order to show that the sequence $\left\{P_k\right\}_{k \geq 1}$ is decreasing (i.e.,  relation~\eqref{decrease_Lyapunov_lem} holds for $k = 1$), it is sufficient to prove that $x_0$, $x_1$ and $x_2$ are all in some compact set.
Indeed, we have:
\begin{align*}
    &f(x_0)+\langle \lambda_0  ,F(x_0)\rangle+\frac{\rho}{2}\|F(x_0)\|^2 \\
    &\overset{\eqref{inequality_prop}}{\geq}  f(x_0) - \frac{\|\lambda_0\|^2}{2(\rho - \rho_0)} - \frac{\rho -\rho_0}{2}\|F(x_0)\|^2 + \frac{\rho}{2}\|F(x_0)\|^2\\
    &= f(x_0) + \frac{\rho_0}{2}\|F(x_0)\|^2 - \frac{\|\lambda_0\|^2}{2(\rho - \rho_0)} \overset{(\rho\geq \rho_0+1)}{\geq} f(x_0) + \frac{\rho_0}{2}\|F(x_0)\|^2 - \|\lambda_0\|^2.
\end{align*}
Combining this with \eqref{ine10}, we further get: 
\begin{equation}\label{x0_in_S}
    f(x_0) + \frac{\rho_0}{2}\|F(x_0)\|^2 \leq \bar{f} + c_0 + 2\|\lambda_0\|^2\leq \bar{P}. 
\end{equation}
Hence, $x_0\in \mathcal{S}^0_{\bar{P}}$.
Moreover, since $x_0 \in \mathcal{S}^0_{\bar{P}}$, $\lambda_0$ is bounded and  the subproblem in Step 5 of Algorithm~\ref{alg1} is quadratic and strongly convex, it follows that there exists a compact set $\mathcal{S}$, which may differ from $\mathcal{S}^0_{\bar{P}}$, such that $x_0, x_1 \in \mathcal{S}$. 
Then, from Assumption~\ref{assump2}, there exist positive constants $L_f'$ and $L_F'$, which may differ from  $L_f$ and $L_F$ (corresponding to compact set $\mathcal{S}^0_{\bar{P}}$), respectively. Furthermore, since $\lambda_0$ is bounded, it follows from the proof of Lemma~\ref{lemma3} (specifically, inequality~\eqref{decrease_proof}) that if
\[
\beta_1 \geq L_f' + L_F'\sqrt{2\rho} \sqrt{\mathcal{L}_{\rho}(x_0,\lambda_0) + \frac{1}{2\rho}\|\lambda_0\|^2 - \ubar{f}},
\]
then the following inequality holds:
\begin{align*}
    \mathcal{L}_{\rho}(x_1, \lambda_0) \leq \mathcal{L}_{\rho}(x_0, \lambda_0) - \frac{\beta_1}{2} \|x_1 - x_0\|^2.
\end{align*}
Further, using the definition of $\mathcal{L}_{\rho}$, we have
\begin{align}\label{ine1110}
    &f(x_1)+\langle \lambda_0  ,F(x_1)\rangle+\frac{\rho}{2}\|F(x_1)\|^2+\frac{\beta_1}{2}\|x_1-x_0\|^2\nonumber\\
    &{\overset{{}}{\leq}}\hspace{0cm} f(x_0)+\langle \lambda_0  ,F(x_0)\rangle+\frac{\rho}{2}\|F(x_0)\|^2\nonumber\\
    &{\overset{{\eqref{ine10}}}{\leq}}\bar{f}+\|\lambda_0\|^2+c_0.
\end{align}
After some rearrangements, we obtain:
\begin{align*}
    f(x_1)+\frac{\rho_0}{2}\|F(x_1)\|^2&{\overset{{(\rho\geq3\rho_0)}}{\leq}}f(x_1)+\frac{\rho}{6}\|F(x_1)\|^2\\
    &{\overset{{\eqref{ine1110}}}{\leq}}\bar{f}+\|\lambda_0\|^2+c_0-\langle \lambda_0  ,F(x_1)\rangle-\frac{\rho}{3}\|F(x_1)\|^2\\
    &=\bar{f}+\|\lambda_0\|^2+c_0-\frac{\rho}{3}\|F(x_1)+\frac{3\lambda_0}{2\rho}\|^2+\frac{3\|\lambda_0\|^2}{4\rho}\\
    &\leq\bar{f}+c_0+2\|\lambda_0\|^2\leq\bar{P}.
\end{align*}
Therefore, we find that $ x_1\in\mathcal{S}^0_{\bar{P}}$.  Hence, we have $\mathcal{S} = \mathcal{S}^0_{\bar{P}}$, $L_f=L_f'$ and $L_F=L_F'$.
Furthermore, using  \eqref{bar_beta}, we get:
\begin{align}\label{beta1_bound}
\beta_1 &\leq    \mu \left(L_f+ L_F\sqrt{2\rho}\sqrt{\mathcal{L}_{\rho}(x_0,\lambda_0) + \frac{1}{2 \rho}\|\lambda_0\|^2 - \ubar{f}}\right)\nonumber\\
& \overset{\eqref{ine10}, \; \rho\geq 1 }{\leq} \mu \left(L_f+ L_F\sqrt{2\rho}\sqrt{\bar{f}  + c_0 + \|\lambda_0\|^2 - \ubar{f}}\right)\nonumber\\
& \leq \mu \left(L_f+ L_F\sqrt{2\rho}\sqrt{\bar{f} + c_0 + 2\|\lambda_0\|^2  - \ubar{f}}\right) \nonumber\\
& \overset{\eqref{alpha_hat}}{\leq} \mu \left(L_f+ L_F\sqrt{2\rho}\sqrt{\bar{P}  - \ubar{f}}\right) \overset{\eqref{beta_bar_def}}{=} \bar{\beta}. 
\end{align} 
Moreover, since $x_0, x_1\in\mathcal{S}^0_{\bar{P}}$ and $D_{\bar{P}}$ is the diameter of $\mathcal{S}^0_{\bar{P}}$ together with the fact that  $\beta_1\leq\bar{\beta}$, it follows from \eqref{lam} that:
\begin{align}
& \|\lambda_1\|^2\overset{\eqref{lam}}{\leq} \left(\frac{1}{\sigma}({\|\nabla f(x_0)\|+\beta_1\|x_1 - x_0\|})\right)^2 \overset{\text{Ass. \ref{assump2}}, \eqref{beta1_bound}}{\leq}\frac{1}{\sigma^2}({M_f+\bar{\beta} D_{\bar{P}}})^2  \nonumber\\
 &\overset{\eqref{beta_bar_def}}{\leq}\frac{2M_f^2 + 4 \mu^2L_f^2D_{\bar{P}}^2}{\sigma^2} +\frac{8\mu^2L_F^2D_{\bar{P}}^2(\bar{P}   - \ubar{f})}{\sigma^2}\rho  \nonumber\\
 &=\frac{2M_f^2 + 4 \mu^2L_f^2D_{\bar{P}}^2 + 8\mu^2L_F^2D_{\bar{P}}^2(\bar{P}   - \ubar{f})\rho_0}{\sigma^2} +\frac{8\mu^2L_F^2D_{\bar{P}}^2(\bar{P}   - \ubar{f})}{\sigma^2}(\rho-\rho_0)  \nonumber\\
 &\overset{\eqref{choice_rho}}{\leq} \left(\frac{16\mu^2L_F^2D_{\bar{P}}^2(\bar{P}  - \ubar{f})}{\sigma^2} + 1\right)\left(\rho-\rho_0\right)\overset{\eqref{bar_gamma_before}}{\leq} 2\bar{\gamma}\left(\rho-\rho_0\right). \label{25a0}
\end{align}
 Similarly to the case of $x_1$, we now prove that $x_2 \in \mathcal{S}^0_{\bar{P}}$. Since $x_1 \in \mathcal{S}^0_{\bar{P}}$ and $\lambda_1$ is bounded (see \eqref{25a0}), and given that the subproblem in Step 5 of Algorithm~\ref{alg1} is quadratic and strongly convex, there exists a compact set $\mathcal{S}$ such that $x_1, x_2 \in \mathcal{S}$.
Moreover, since $x_0, x_1 \in \mathcal{S}^0_{\bar{P}}$, it follows that there exists a compact set $\mathcal{S}' \subseteq \mathcal{S}^0_{\bar{P}} \cup \mathcal{S}$ such that $x_0, x_1, x_2 \in \mathcal{S}'$. Then, from Assumption~\ref{assump2}, there exist positive constants $L_f', L_F'$, which may differ from  $L_f$ and $L_F$, respectively.
Furthermore, since $\lambda_1$ is bounded, Lemma~\ref{lemma3} implies that if
\[
\beta_2 \geq L_f' + L_F'\sqrt{2\rho}\sqrt{\mathcal{L}_{\rho}(x_1,\lambda_1) + \frac{1}{2\rho}\|\lambda_1\|^2 - \ubar{f}},
\]
then the following inequality holds:
\begin{equation} \label{x_2_bounded}
    \mathcal{L}_{\rho}(x_2, \lambda_1) \leq \mathcal{L}_{\rho}(x_1, \lambda_1) - \frac{\beta_2}{2} \|x_2 - x_1\|^2.
\end{equation}
Let us now prove that $P_1=\mathcal{L}_{\rho}(x_1,\lambda_1)+\frac{\beta_1}{4}\|x_1-x_0\|^2$ is bounded. We have $\beta_1\leq\bar{\beta}$. Moreover,  since $x_0, x_1 \in \mathcal{S}^0_{\bar{P}}$ and $\lambda_0$ is bounded,  then from Lemma \ref{lemma3}, we have:
\begin{align}
 P_1 &\leq P_0 +\frac{3}{2\rho}\|\lambda_1 - \lambda_0\|^2 - \frac{\beta_1}{4}\|x_1 - x_0\|^2  - \frac{\beta_0}{4}\|x_0 - x_{-1}\|^2 \nonumber\\
    & \overset{(x_0=x_{-1})}{=} \mathcal{L}_{\rho}(x_0,\lambda_0)+ \frac{6}{2\rho}\|\lambda_0\|^ 2+ \frac{6}{2\rho}\|\lambda_1\|^2 - \frac{\beta_1}{4}\|x_1 - x_0\|^2 \\
    & \overset{\eqref{first_bound}}{\leq} \mathcal{L}_{\rho}(x_0,\lambda_0) + \frac{6}{2\rho}\|\lambda_0\|^ 2+ 1  \nonumber\\
    & \overset{\eqref{ine10}, \; \rho\geq 1}{\leq} \bar{f}  + c_0 + 4 \|\lambda_0\|^2 + 1  \overset{\eqref{alpha_hat}}{=}\bar{P} -1. \label{P10_bar}
\end{align}
Further, using \eqref{P10_bar} and the expression of $P_1$  and of $\mathcal{L}_{\rho}$ in \eqref{x_2_bounded}, we get:
\begin{align}\label{ine111}
    &f(x_2)+\langle \lambda_1  ,F(x_2)\rangle+\frac{\rho}{2}\|F(x_2)\|^2 \nonumber\\
    &\leq \bar{P} -1 - \frac{\beta_1}{4}\|x_1-x_0\|^2- \frac{\beta_2}{2}\|x_2-x_1\|^2 \leq \bar{P} -1 - \frac{\beta_1}{4}\|x_1-x_0\|^2.
\end{align}
After some rearrangements, we obtain:
\begin{align*}
    &f(x_2)+\frac{\rho_0}{2}\|F(x_2)\|^2{\overset{{(\rho\geq3\rho_0)}}{\leq}}f(x_2)+\frac{\rho}{6}\|F(x_2)\|^2\\
    &{\overset{}{\leq}}\bar{P} -1 - \frac{\beta_1}{4}\|x_1-x_0\|^2-\langle \lambda_1  ,F(x_2)\rangle-\frac{\rho}{3}\|F(x_2)\|^2\\
    &\leq\bar{P} -1 - \frac{\beta_1}{4}\|x_1-x_0\|^2-\frac{\rho}{3}\|F(x_2)+\frac{3\lambda_1}{2\rho}\|^2+\frac{3\|\lambda_1\|^2}{4\rho}\\
    &\leq\bar{P} -1 + \frac{3\|\lambda_1\|^2}{\rho}- \frac{\beta_1}{4}\|x_1-x_0\|^2 \overset{\eqref{first_bound}}{\leq} \bar{P}.
\end{align*}
Therefore, we find that $ x_2\in\mathcal{S}^0_{\bar{P}}$.  Hence, we have $x_0,x_1,x_2 \in \mathcal{S}^0_{\bar{P}}$ and it follows that $\mathcal{S}' = \mathcal{S}= \mathcal{S}^0_{\bar{P}}$,  $L_f'=L_f$ and $L_F'=L_F$. Then, using Lemma \ref{conditional_decrease}, it follows that if $\rho$ satisfies \eqref{theright_choice_of_rho} with $k=1$, we get:
\begin{equation}\label{decrease_k=1}
    P_2 - P_1 \leq -\frac{\beta_2}{8}\|x_2-x_1\|^2 - \frac{\beta_1}{8}\|x_1-x_0\|^2.
\end{equation}
Indeed, from our choice of $\rho$ in \eqref{choice_rho}, we have that:
\[
\rho \geq \max\bigg\{ \frac{48(1+3\mu)^2\left(L_f M_F + M_f L_F\right)^2 }{\mu L_f\sigma^4}, \frac{48(1+3\mu)^2 M_F^2}{\sigma^4}(\bar{\beta} - \mu L_f)\bigg\}.
\]
Moreover, since $\beta_1\leq \bar{\beta}$, it follows that:
\[
\rho \geq \max\bigg\{ \frac{48(1+3\mu)^2\left(L_f M_F + M_f L_F\right)^2 }{\mu L_f\sigma^4}, \frac{48(1+3\mu)^2 M_F^2}{\sigma^4}(\beta_1 - \mu L_f)\bigg\}.
\]
For \eqref{theright_choice_of_rho} to be valid for $k=1$, it remains to prove that:
\[
\rho \geq \frac{48(1+3\mu)^2 M_F^2}{\sigma^4}(\beta_2 - \mu L_f).
\]
Hence, it is sufficient to prove that $\beta_2\leq \bar{\beta}$.
Indeed, we have:
\begin{align}\label{beta2_bound}
&\beta_2 \overset{\eqref{bar_beta}}{\leq}   \mu \left(L_f+ L_F\sqrt{2\rho}\sqrt{\mathcal{L}_{\rho}(x_1,\lambda_1) + \frac{1}{2 \rho}\|\lambda_1\|^2 - \ubar{f}}\right)\nonumber\\
& \overset{\eqref{second_bound}}{\leq} \mu \left(L_f+ L_F\sqrt{2\rho}\sqrt{P_1 + 1   - \ubar{f}}\right) \nonumber\\
& \overset{\eqref{P10_bar}}{\leq} \mu \left(L_f+ L_F\sqrt{2\rho}\sqrt{\bar{P}  - \ubar{f}}\right) = \bar{\beta}. 
\end{align}
Hence, $\rho$ satisfies \eqref{theright_choice_of_rho} for $k=1$.
It then follows that for $k=1$, \eqref{important} is  verified. Now, assume that there exists some $k \geq 1$ such that \eqref{important} holds for all $j \leq k$ (induction hypothesis (IH)). We will prove that it also holds for $k+1$. To this end, we follow the same steps as in the case $k = 1$, and therefore only sketch the proof.
We begin by showing that $x_{k+1} \in \mathcal{S}^0_{\bar{P}}$. Since $x_k \in \mathcal{S}^0_{\bar{P}}$ and $\lambda_k$ is bounded (by IH), and given that the subproblem in Step 5 of Algorithm~\ref{alg1} is quadratic and strongly convex, there exists a compact set $\mathcal{S}$ (possibly different from $\mathcal{S}^0_{\bar{P}}$) such that $x_k, x_{k+1} \in \mathcal{S}$.
Moreover, since $x_{k-1}, x_k \in \mathcal{S}^0_{\bar{P}}$, it follows that there exists a compact set $\mathcal{S}' \subseteq \mathcal{S}^0_{\bar{P}} \cup \mathcal{S}$ such that $x_{k-1}, x_k, x_{k+1} \in \mathcal{S}'$. Then, by Assumption~\ref{assump2}, there exist positive constants $L_f', L_F'$, which may differ from  $L_f$ and $L_F$, respectively.
Furthermore, since $\lambda_k$ is bounded, Lemma~\ref{lemma3} implies that if
\[
\beta_{k+1} \geq L_f' + L_F'\sqrt{2\rho}\sqrt{\mathcal{L}_{\rho}(x_k,\lambda_k)+ \frac{1}{2\rho}\|\lambda_k\|^2 - \ubar{f}},
\]
then the following inequality holds:
\begin{equation} \label{x_kkk_bounded}
    \mathcal{L}_{\rho}(x_{k+1},\lambda_k) \leq \mathcal{L}_{\rho}(x_k,\lambda_k) - \frac{\beta_{k+1}}{2}\|x_{k+1} - x_k\|^2.
\end{equation}
Further, using the fact that $P_k \leq \bar{P}-1  $ (see IH) together with the expression of $P_k$  and of $\mathcal{L}_{\rho}$ in \eqref{x_kkk_bounded}, we get:
\begin{align}\label{inekkk}
    &f(x_{k+1})+\langle \lambda_k  ,F(x_{k+1})\rangle+\frac{\rho}{2}\|F(x_{k+1})\|^2 \nonumber\\
    &\leq \bar{P} -1 - \frac{\beta_k}{4}\|\Delta x_k\|^2- \frac{\beta_{k+1}}{2}\|\Delta x_{k+1}\|^2 \leq \bar{P} -1 - \frac{\beta_k}{4}\|\Delta x_k\|^2.
\end{align}
After some rearrangements, we obtain:
\begin{align*}
    &f(x_{k+1})+\frac{\rho_0}{2}\|F(x_{k+1})\|^2{\overset{{(\rho\geq3\rho_0)}}{\leq}}f(x_{k+1})+\frac{\rho}{6}\|F(x_{k+1})\|^2\\
    &\overset{}{\leq}\bar{P} -1 - \frac{\beta_k}{4}\|\Delta x_k\|^2-\langle \lambda_k  ,F(x_{k+1})\rangle-\frac{\rho}{3}\|F(x_{k+1})\|^2\\
    &\leq\bar{P} -1 - \frac{\beta_k}{4}\|\Delta x_k\|^2-\frac{\rho}{3}\|F(x_{k+1})+\frac{3\lambda_k}{2\rho}\|^2+\frac{3\|\lambda_k\|^2}{4\rho}\\
    &\leq\bar{P} -1 + \frac{3\|\lambda_k\|^2}{\rho}- \frac{\beta_k}{4}\|\Delta x_k\|^2 \overset{\eqref{first_bound}}{\leq} \bar{P}.
\end{align*}
Therefore, we have  $ x_{k+1}\in\mathcal{S}^0_{\bar{P}}$. Moreover, since $x_k,x_{k+1}\in\mathcal{S}^0_{\bar{P}}$ and $D_{\bar{P}}$ is the diameter of $\mathcal{S}^0_{\bar{P}}$, then we have:
\begin{align}
& \|\lambda_{k+1}\|^2\overset{\eqref{lam}}{\leq} \left(\frac{1}{\sigma}({\|\nabla f(x_{k})\|+\beta_{k+1}\|\Delta x_{k+1}\|})\right)^2 \overset{\text{Ass. \ref{assump2}}, \eqref{bound_beta}}{\leq}\frac{1}{\sigma^2}({M_f+\bar{\beta} D_{\bar{P}}})^2  \nonumber\\
 &\overset{\eqref{beta_bar_def}}{\leq}\frac{2M_f^2 + 4 \mu^2L_f^2D_{\bar{P}}^2}{\sigma^2} +\frac{8\mu^2L_F^2D_{\bar{P}}^2(\bar{P}   - \ubar{f})}{\sigma^2}\rho  \nonumber\\
 &=\frac{2M_f^2 + 4 \mu^2L_f^2D_{\bar{P}}^2 + 8\mu^2L_F^2D_{\bar{P}}^2(\bar{P}   - \ubar{f})\rho_0}{\sigma^2} +\frac{8\mu^2L_F^2D_{\bar{P}}^2(\bar{P}   - \ubar{f})}{\sigma^2}(\rho-\rho_0)  \nonumber\\
 &\overset{\eqref{choice_rho}}{\leq} \left(\frac{16\mu^2L_F^2D_{\bar{P}}^2(\bar{P}  - \ubar{f})}{\sigma^2} + 1\right)\left(\rho-\rho_0\right)\overset{\eqref{bar_gamma_before}}{\leq} 2\bar{\gamma}\left(\rho-\rho_0\right). \label{25a}
\end{align}
Furthermore, we have $x_{k-1}, x_k\in\mathcal{S}^0_{\bar{P}}$ and $\beta_k\leq \bar{\beta} $ (see IH). Hence, using \eqref{bar_beta} and \eqref{second_bound}, we get:
\begin{align}\label{betak_bound}
\beta_{k+1} & \leq \mu \left(L_f+ L_F\sqrt{2\rho}\sqrt{P_k  + 1 - \ubar{f}}\right)\nonumber\\
& \overset{(IH)}{\leq} \mu \left(L_f+ L_F\sqrt{2\rho}\sqrt{\bar{P} - \ubar{f}}\right) = \bar{\beta}. 
\end{align}
Further, from the induction hypothesis, we also have:
\begin{equation} \label{bound_Pkk}
    P_{k+1} \leq P_{k}  -\frac{\beta_{k+1}}{8}\|\Delta x_{k+1}\|^2 - \frac{\beta_{k}}{8}\|\Delta x_{k}\|^2 \overset{(\text{IH})}{\leq} \bar{P} -2.
\end{equation}
Furthermore, from \eqref{choice_rho}, we have:
\[
\rho \geq \max\bigg\{ \frac{48(1+3\mu)^2\left(L_f M_F + M_f L_F\right)^2 }{\mu L_f\sigma^4}, \; \frac{48(1+3\mu)^2 M_F^2}{\sigma^4}(\bar{\beta} - \mu L_f)\bigg\}.
\]
In addition, from \eqref{betak_bound}, we have $\beta_{k+1} \leq \bar{\beta}$. Similarly, it is easy to obtain that
\begin{align}\label{beta2_bound}
\beta_{k+2} &\leq \mu \left(L_f+ L_F\sqrt{2\rho}\sqrt{\bar{P} - \ubar{f}}\right) = \bar{\beta}.
\end{align}
It then follows from Lemma~\ref{conditional_decrease} that
\begin{equation}\label{decrease_k=1}
    P_{k+2} - P_{k+1} \leq -\frac{\beta_{k+2}}{8}\|\Delta x_{k+2}\|^2 - \frac{\beta_{k+1}}{8}\|\Delta x_{k+1}\|^2, 
\end{equation}
that is, \eqref{important} is  proved.  This completes our proof.  \qed

\medskip

\noindent \textbf{Proof of Lemma \ref{bounded_below}}
Let $k\geq 1$. Since $x_k\in\mathcal{S}^0_{\bar{P}}$, then using \eqref{lyapunov_function},  we have:
\begin{align*}
  P_{k}&\geq f(x_{k})+\frac{\rho}{2}\|F(x_{k})\|^2+\langle \lambda_{k}   , F(x_{k}) \rangle\nonumber\\
  &\geq f(x_{k})+\frac{\rho}{2}\|F(x_{k})\|^2-\frac{\|\lambda_{k}\|^2}{2(\rho-\rho_0)}-\frac{\rho-\rho_0}{2}\|F(x_{k})\|^2\nonumber\\
 & {\overset{{\eqref{25a}}}{\geq}}  f(x_{k})+\frac{\rho_0}{2}\|F(x_{k})\|^2-1{\overset{{(\text{Lemma } \ref{lem1})}}{\geq}}\ubar{P}-1. \label{bound}
\end{align*}
It follows that the sequence $\{P_{k}\}_{k\geq1}$ is bounded from below. This concludes our proof.\qed

\medskip

\noindent \textbf{Proof of Lemma \ref{bounded_grad}}
By exploiting the definition of $P(\cdot)$ defined in \eqref{P}, we have that for any $k\geq1$:
\begin{gather*}  \nabla_xP(x,\lambda,y,\gamma)=\nabla_x{\mathcal{L}_{\rho}}(x,\lambda)+{{\gamma}}(x-y), \hspace{0.2cm}\text{  }\hspace{0.2cm} \nabla_{\lambda}{P}(x,\lambda,y,\gamma)=\nabla_{\lambda}{\mathcal{L}_{\rho}}(x,\lambda)\\
      \nabla_{y}{P}(x,\lambda,y,\gamma)={{\gamma}}(y-x) \hspace{0.2cm}\text{ and }\hspace{0.2cm}
      \nabla_{\gamma}{P}(x,\lambda,y,\gamma)=\frac{1}{2}\|x-y\|^2.
\end{gather*}
Hence,
\begin{align*}
    &\| \nabla P(x_{k+1},\lambda_{k+1},x_{k},\frac{\beta_{k+1}}{2})\| \\
    & \leq\|\nabla{\mathcal{L}_{\rho}}(x_{k+1},\lambda_{k+1})\|
    +\beta_{k+1}\|\Delta x_{k+1}\| +\frac{1}{2}\|\Delta x_{k+1}\|^2\\
    &\leq \left(\Gamma_{k+1} +\beta_{k+1} \right)\|\Delta x_{k+1}\|  +\Gamma_{k}\|\Delta x_{k}\| + \left(c_{k+1} + \frac{1}{2}\right) \|\Delta x_{k+1}\|^2 + c_k\|\Delta x_{k}\|^2\\
&\leq(\bar{\Gamma}+(\bar{c} +1)D_{\bar{P}}+\bar{\beta})\left(\|\Delta x_{k+1}\|+\|\Delta x_{k}\|\right),
\end{align*}
where the last two inequalities follow from Lemma {\ref{bounded_gradient} and  \eqref{boundgam}}. \qed

\medskip

\noindent \textbf{Proof of Lemma \ref{added_lemma}}
\eqref{lem_item1} From Lemma \ref{bbound} and Lemma \ref{bounded_below}, it follows that $\{u_k\}_{k\geq1}$ is bounded and therefore, there exists   a convergent subsequence  $\{u_k\}_{k\in\mathcal{K}}$ such that $\lim_{k\in\mathcal{K}}{u_k}=u^{*}$. Hence $\Omega$ is nonempty. Moreover, $\Omega$ is compact since it is bounded and closed.
On the other hand, for any $u^{*}\in\Omega$, there exists a sequence of increasing integers $\mathcal{K}$ such that $\lim_{k\in\mathcal{K}}{u_k}=u^{*}$ and using Lemma \ref{bounded_grad} and \eqref{zero_limit}, it follows that:
\[
\|\nabla P(u^{*})\|=\lim_{k\in\mathcal{K}}{\|\nabla P(u_k)\|}=0.
\]
Hence, $u^{*}\in\texttt{Stat}P$ and $0\leq \lim_{k\to\infty}{\text{dist}(u_k,\Omega)}\leq\lim_{k\in\mathcal{K}}{\text{dist}(u_k,\Omega)}=\text{dist}(u^{*},\Omega)=0$. This proves the first claim. \\
\noindent \eqref{lem_item2} Since $P$ is continuous function and $ \{P(u_k)=P_k\}_{k\geq1}$ converges to $P^{*}$, then any convergent subsequence $\{P(u_k)=P_k\}_{k\in\mathcal{K}}$, it must converge to the same limit $P^{*}$.  This proves the second claim.\\
\noindent \eqref{lem_item3} Let $(x,\lambda,y,\gamma)\in\texttt{Stat}P$ that is $\nabla P(x,\lambda,y,\gamma)=0$. It then follows that:
\begin{gather*}
\nabla_xP(x,\lambda,y,\gamma) \!=\! \nabla_x{\mathcal{L}_{\rho}}(x,\lambda)+{{\gamma}}(x-y) = \!0, \;  \nabla_{\lambda}{P}(x,\lambda,y,\gamma) \!=\! \nabla_{\lambda}{\mathcal{L}_{\rho}}(x,\lambda)=\! 0, \\
      \nabla_{y}{P}(x,\lambda,y,\gamma)={{\gamma}}(y-x) =0\hspace{0.2cm}\text{ and }\hspace{0.2cm}
      \nabla_{\gamma}{P}(x,\lambda,y,\gamma)=\frac{1}{2}\|x-y\|^2=0.
\end{gather*}
With some minor rearrangements, we obtain:
\[
\nabla f(x)+{{J_F}(x)}^T\lambda=0 \hspace{0.2cm}\text{ and }\hspace{0.2cm} F(x)=0.
\]
Hence, $(x,\lambda)$ is a KKT point of \eqref{eq1}. This concludes our proof.\qed

\medskip

\noindent \textbf{Proof of Lemma \ref{le:feasibility}} 
Let us assume by contradiction that there exists $k\geq 1$ such that $\Vert F(x^k)\Vert >\epsilon$. 
Then, the dual multipliers at the $k$th iteration are updated as $\lambda^k = \lambda^{k-1} + \rho F(x^k)$.
Furthermore, since the dual sequnce $\{\lambda^k\}$ is bounded, there exists $M > 0$ such that $\Vert \lambda^k\Vert  \leq M$ for any $k \geq 0$. 
Therefore, we can show that
\begin{equation*}
	\epsilon < \Vert F(x^k)\Vert = \frac{\Vert \lambda^k -  \lambda^{k-1}\Vert }{\rho} \leq \frac{\Vert \lambda^k \Vert + \Vert \lambda ^{k-1}\Vert }{\rho} \leq \frac{2M}{\rho} \leq \epsilon,
\end{equation*}
which is a contradiction. Hence, we get that $\Vert F(x^k)\Vert \leq \epsilon$ for any $k\geq 1$. \qed

\medskip

\noindent \textbf{Proof of Lemma \ref{le:lemma_R1}}   By strong convexity  of the subproblem at  Line 4 in Algorithm \ref{alg:ALSSF_alg}, the smoothness of $f$ and $F$, the optimality condition of $x^{k+1}$, and  $\lambda^{k+1}=\lambda^k$  for all $k\geq 1$, we have
\begin{equation*}
\mathcal{L}_{\rho}(x^{k+1},\lambda^{k+1})-\mathcal{L}_{\rho}(x^{k},\lambda^{k})\leq -\frac{\beta}{2}\Vert x^{k+1}-x^k\Vert ^2.
\end{equation*}
Moreover, $x^{k+1}$ computed at Line 4 has  the following explicit expression:
\begin{equation*}
x^{k+1}=x^k-\left(\beta I_n+\rho J_F(x^k)^TJ_F(x^k)\right)^{-1}\nabla_x\mathcal{L}_{\rho}(x^k,\lambda^k).
\end{equation*}
Therefore, we get
\begin{equation*}
\Vert x^{k+1}-x^k\Vert \geq \frac{ \Vert \nabla_x\mathcal{L}_{\rho}(x^k,\lambda^k) \Vert }{\left\Vert \beta I_n+\rho J_F(x^k)^TJ_F(x^k)\right\Vert } >\frac{\alpha}{\left\Vert \beta I_n+\rho J_F(x^k)^TJ_F(x^k)\right\Vert }.
\end{equation*}
Furthermore, we also have
\begin{equation*}
\left\Vert \beta I_n+\rho J_F(x^k)^TJ_F(x^k)\right\Vert \leq \beta+\rho M_F^2\leq \beta +L_{\rho} \leq 2\beta.
\end{equation*}
This proves our claim.  \qed

\medskip

\noindent \textbf{Proof of Lemma \ref{le:lemma_R2}}  Let us choose a direction $d_k\in\mathbb{R}^n$ satisfying
\begin{equation*}
\Vert d_k\Vert  =1, \quad \langle \nabla_x\mathcal{L}_{\rho}(x^k,\lambda^k),d_k\rangle\leq0, \quad \textrm{and} \quad d_k^T\nabla^2_{xx}\mathcal{L}_{\rho}(x^k,\lambda^k)d_k\leq -\theta  \rho^{\zeta_1}.
\end{equation*}
Such a vector $d_k$ is well-defined since $x^k\in\mathcal{R}_2$.  Furthermore, let $\omega>0$ be small such that $x^k+\omega d_k$ belongs to a level set of $\Lc_{\rho}$. 
In fact,  since $\nabla^2_{xx}\mathcal{L}_{\rho}(\cdot,\lambda^k)$ is locally Lipschitz (see Assumptions \ref{as:Assu_SSF}), it follows that for any $0<\omega\leq\frac{\theta  \rho^{\zeta_1}}{H_{\rho}}$, we have
\begin{equation}\label{hessian_lipschitz}
\arraycolsep=0.2em
\begin{array}{lcl}
	\mathcal{L}_{\rho}(x^{k}+\omega d_k,\lambda^{k}) & \leq &  \hat{\mathcal{Q}}_{\mathcal{L}_{\rho}}(x^{k}+\omega d_k,\lambda^k;x^k) + \frac{H_{\rho}}{6} \omega^3 \vspace{1ex}\\
	& \leq &  \mathcal{L}_{\rho}(x^{k},\lambda^{k}) - \frac{1}{2} \theta  \rho^{\zeta_1} \omega^2 +  \frac{H_{\rho}}{6} \omega^3 \vspace{1ex}\\
	& \leq &  \mathcal{L}_{\rho}(x^{k},\lambda^{k}) - \frac{1}{2} \theta  \rho^{\zeta_1} \omega^2 +  \frac{H_{\rho}}{6}  \frac{\theta  \rho^{\zeta_1}}{H_{\rho}} \omega^2 \vspace{1ex}\\
        & = & \mathcal{L}_{\rho}(x^{k},\lambda^{k}) - \frac{1}{3} \theta  \rho^{\zeta_1} \omega^2. 
\end{array}     
\end{equation}
Hence, given that $x^k$ belongs to a level set of $\Lc_{\rho}$,  for any $0<\omega\leq\frac{\theta  \rho^{\zeta_1}}{H_{\rho}}$, the point $x^k+\omega d_k$ also belongs to the same level set of $\Lc_{\rho}$.
Therefore, by choosing  $H_{\rho} \leq  \upsilon \leq 2 H_{\rho}$ and   $0<\omega=\frac{\theta  \rho^{\zeta_1}}{2 \upsilon} \leq \frac{\theta  \rho^{\zeta_1}}{H_{\rho}}$, utilizing \cite[Lemma 4.1.5]{Nes:18}, and the Lipschitzness of the Hessian of $\Lc_{\rho}$, it follows that
\begin{equation*}
\arraycolsep=0.0em
\begin{array}{lcl}
	\mathcal{L}_{\rho}(x^{k+1},\lambda^{k+1}) & = &  \mathcal{L}_{\rho}(x^{k+1},\lambda^{k}) \vspace{1ex}\\
	& \overset{\tiny\textrm{\cite[Lemma 4.1.5]{Nes:18}} }{\leq} & \mathcal{L}_{\rho}(x^{k}+\omega d_k,\lambda^{k})+\frac{\upsilon}{3}\omega^3 \vspace{1ex} \\
	& \overset{\eqref{hessian_lipschitz}}{\leq} &  \mathcal{L}_{\rho}(x^{k},\lambda^{k}) - \frac{1}{3} \theta  \rho^{\zeta_1}\omega^2 + \frac{\upsilon}{3}\frac{\theta  \rho^{\zeta_1}}{2 \upsilon}\omega^2 \vspace{1ex}\\
	& = & \mathcal{L}_{\rho}(x^{k},\lambda^{k}) -\frac{\theta  \rho^{\zeta_1}}{6}\omega^2 \vspace{1ex}\\
	& = & \mathcal{L}_{\rho}(x^{k},\lambda^{k}) -\frac{\theta^3  \rho^{3\zeta_1}}{24 \upsilon^2}.
\end{array}
\end{equation*}
Finally, utilizing $\upsilon \leq 2 H_{\rho}$, our claim follows from the last expression.   \qed

\medskip

\noindent \textbf{Proof of Lemma \ref{le:quadratic}}     From the optimality condition of $x^{k+1}$ at Line 7, we get
\begin{equation*}
	x^{k+1} -x^{*} = x^k - x^{*} -\frac{1}{\beta}\nabla_x\mathcal{L}_{\rho}(x^k,\lambda^k).
\end{equation*}
Hence, we have
\begin{equation*}
	\Vert x^{k+1} - x^{*}\Vert ^2 = \Vert x^k - x^{*}\Vert ^2 - \frac{2}{\beta}\langle  \nabla_x \mathcal{L}_{\rho}(x^k, \lambda^k), x^{k} - x^{*} \rangle + \frac{1}{\beta^2}\Vert  \nabla_x \mathcal{L}_{\rho}(x^k, \lambda^k)\Vert ^2.
\end{equation*}
Moreover, since  the augmented Lagrangian function $\Lc_{\rho}$ is $\gamma \rho^{\zeta_2}$-strongly convex and $L_{\rho}$-smooth on  the region $\mathcal{R}_3\setminus\mathcal{R}_1$, and  $\nabla_x \mathcal{L}_{\rho}(x^{*}, \lambda^k)= 0 $,  we have
\begin{equation*}
	\langle \nabla_x \mathcal{L}_{\rho}(x^{k}, \lambda^k), x^{k} - x^{*} \rangle \geq \mathcal{L}_{\rho}(x^{k}, \lambda^k)- \mathcal{L}_{\rho}(x^{*}, \lambda^k) + \tfrac{\gamma \rho^{\zeta_2} }{2} \Vert x^{k} - x^{*}\Vert ^2
\end{equation*}
and
\begin{equation*}
	\tfrac{1}{2 L_{\rho}}   \Vert  \nabla_x \mathcal{L}_{\rho}(x^k, \lambda^k)\Vert ^2 \leq \mathcal{L}_{\rho}(x^{k}, \lambda^k)- \mathcal{L}_{\rho}(x^{*}, \lambda^k).
\end{equation*}
Therefore, we can derive that
\begin{equation*}
\arraycolsep=0.2em
\begin{array}{lcl}
        \Vert x^{k+1} - x^{*}\Vert ^2 & \leq & \left(1 - \frac{\gamma \rho^{\zeta_2} }{\beta}\right) \Vert x^k - x^{*}\Vert ^2 - \frac{2}{\beta} \left(1-\frac{L_{\rho}}{\beta}\right)\left(\mathcal{L}_{\rho}(x^{k}, \lambda^k)- \mathcal{L}_{\rho}(x^{*}, \lambda^k)\right) \vspace{1ex}\\
        & \leq &  \left(1- q_{\rho}\right) \Vert x^k - x^{*}\Vert ^2 \leq \Vert x^k - x^{*}\Vert ^2 \leq \xi^2.
\end{array}
\end{equation*}
This allows us to conclude that whenever $x^k$ is in $\mathcal{R}_3 \setminus \mathcal{R}_1$, the iterates remain in this region.   Finally, for any $N \geq 1$, using the characterization of  $\mathcal{R}_3$, we have  $\Vert x^k-x^{*}\Vert \leq \xi$ and thus
\begin{equation*}
	\Vert x^{k+N} - x^{*}\Vert  \leq  \left(1- q_{\rho}\right)^{\frac{N}{2}} \Vert x^k - x^{*}\Vert  \leq \left(1- q_{\rho}\right)^{\frac{N}{2}}\xi.
\end{equation*}
Consequently, for any $\epsilon > 0$, after $N= \mathcal{O}{\big( \frac{1}{q_{\rho}} \log{\big(\frac{L_{\rho}\xi}{\epsilon} \big)} \big)} = \mathcal{O}\big( \frac{1}{\epsilon^{1-\zeta_2}} \log\big(\frac{1}{\epsilon} \big) \big)$ iterations, we have $\Vert \nabla_x \mathcal{L}_{\rho}(x^{k+N}, \lambda^k)\Vert  \leq L_{\rho}\Vert x^{k+N} - x^{*}\Vert  \leq \epsilon$.
This completes our proof. \qed 

\medskip

\noindent \textbf{Proof of Lemma \ref{le:R3_enter}}    Let $K \geq $ be the first iteration such that $x^{K+1} \in \mathcal{R}_3\setminus \mathcal{R}_1$. 
Let us introduce
\begin{equation*}
S_1 \triangleq \{k \in [1: K] : x^k \in \mathcal{R}_1\}  \quad \textrm{and} \quad  S_2 \triangleq \{k \in [1: K] : x^k \in \mathcal{R}_2 \setminus \mathcal{R}_1\}. 
\end{equation*}
Since $\lambda^k = \lambda^1$ for all $k >1$, we have
\begin{align*}
& \mathcal{L}_{\rho}(x^{1},\lambda^{1}) - \mathcal{L}_{\rho}(x^{*},\lambda^{1})  \geq   \mathcal{L}_{\rho}(x^{1},\lambda^{1}) - \mathcal{L}_{\rho}(x^{K+1},\lambda^{K+1}) \vspace{1ex} \\
& =   \sum_{k=1}^{K} \left( \mathcal{L}_{\rho}(x^{k},\lambda^{k}) - \mathcal{L}_{\rho}(x^{k+1},\lambda^{k+1}) \right) \vspace{1ex} \\
& =   \sum_{k \in S_1} \left( \mathcal{L}_{\rho}(x^{k},\lambda^{k}) - \mathcal{L}_{\rho}(x^{k+1},\lambda^{k+1}) \right) + \sum_{k \in S_2} \left( \mathcal{L}_{\rho}(x^{k},\lambda^{k}) - \mathcal{L}_{\rho}(x^{k+1},\lambda^{k+1}) \right).
\end{align*}
By Lemma \ref{le:lemma_R1} and Lemma \ref{le:lemma_R2}, we obtain
\begin{equation*}
\mathcal{L}_{\rho}(x^{1},\lambda^{1}) - \mathcal{L}_{\rho}(x^{*},\lambda^{1}) \geq \vert S_1 \vert \frac{\alpha^2}{8\beta} + \vert S_2 \vert \frac{\theta^3  \rho^{3\zeta_1}}{96 H_\rho^2}.
\end{equation*}
This expression implies
\begin{equation*}
\vert S_1 \vert  \leq \! \frac{8\beta \! \left(\mathcal{L}_{\rho}(x^{1},\lambda^{1}) -\! \mathcal{L}_{\rho}(x^{*},\lambda^{1})\right)}{\alpha^2} \;\; \textrm{and} \;\; 
\vert S_2 \vert \leq \! \frac{96 H_\rho^2 \! \left(\mathcal{L}_{\rho}(x^{1},\lambda^{1}) -\! \mathcal{L}_{\rho}(x^{*},\lambda^{1})\right)}{\theta^3  \rho^{3\zeta_1}}.
\end{equation*}
Combining the last two bounds, we can show that
\begin{equation*}
K = \vert S_1 \vert + \vert S_2\vert \leq \left(\mathcal{L}_{\rho}(x^{1},\lambda^{1}) - \mathcal{L}_{\rho}(x^{*},y^{1})\right)\left(\frac{8\beta}{\alpha^2}+ \frac{96 H_\rho^2}{\theta^3  \rho^{3\zeta_1}}\right).
\end{equation*}
Moreover, since  our primal and dual iterates are assumed bounded,  it follows that  $\mathcal{L}_{\rho}(x^{1},y^{1})- \mathcal{L}_{\rho}(x^{*},\lambda^{1}) $ is bounded; additionally, since    $\beta \geq L_\rho$ and  $L_\rho, H_\rho$ are proportional to $\rho$ and since  $\rho$ is of order $\mathcal{O}\left(\frac{1}{\epsilon}\right)$,  our final  claim follows.   \qed


\end{document}